\documentclass[11pt,oneside,english]{amsart}
\usepackage[utf8]{inputenc}
\usepackage[letterpaper]{geometry}
\usepackage{color}
\usepackage{amstext}
\usepackage{amsthm}
\usepackage{amssymb}
\usepackage{setspace}

\usepackage{array}
\usepackage{makecell}
\usepackage{hyperref}
\usepackage{mathtools}
\usepackage{bm}
\usepackage{tikz-cd}
\usepackage{graphicx}
\usepackage{amssymb}   

\usepackage{amsmath}

\usepackage{mathrsfs}

\usepackage{stmaryrd}

\usepackage{newlfont}
\usepackage{amscd}
\usepackage{mathtools}
\usepackage{bm}
\usepackage{tikz-cd}
\usepackage{graphicx}

\usepackage{hyperref}
\usepackage{framed}
\usepackage{comment}
\makeatletter
\numberwithin{equation}{section}
\numberwithin{figure}{section}
\usepackage{cite}

\theoremstyle{plain}
\newtheorem{thm}{\protect\theoremname}[section]
\theoremstyle{plain}
\newtheorem{lem}[thm]{\protect\lemmaname}
\theoremstyle{plain}
\newtheorem{prop}[thm]{\protect\propositionname}
\theoremstyle{definition}
\newtheorem{defn}[thm]{\protect\definitionname}
\theoremstyle{remark}
\newtheorem{rem}[thm]{\protect\remarkname}
\theoremstyle{plain}
\newtheorem{cor}[thm]{\protect\corollaryname}
\theoremstyle{plain}

\providecommand{\factname}{Fact}
\theoremstyle{definition}

\theoremstyle{definition}

\usepackage{slashed}
\usepackage[all]{xy}

\newcommand{\dbar}{\bar{\partial}}

\newcommand{\HE}{\rm{Hermitian–Einstein}}
\newcommand{\pdv}{\partial}
\newcommand{\im}{\sqrt{-1}}
\newcommand {\Hom} {\operatorname{Hom}}

\newcommand {\Ind} {\operatorname{Ind}}
\newcommand {\Pro} {\operatorname{Pro}}

\newcommand {\coker} {\operatorname{coker}}

\newcommand {\Coh} {\mathrm{Coh}}

\def\E{\mathcal{E}}

\def\F{\mathcal{F}}

\def\R{\mathbb R}
\def\C{\mathbb C}
\def\Z{\mathbb Z}
\def\Q{\mathcal{Q}}
\def\S{\mathcal{S}}

\def\H{\mathbb H}

\def\N{\mathbb N}

\def\ep{\epsilon}

\newcommand{\<}{\langle}
\def\>{\rangle}
\makeatother

\usepackage{babel}
\providecommand{\corollaryname}{Corollary}
\providecommand{\definitionname}{Definition}
\providecommand{\lemmaname}{Lemma}
\providecommand{\propositionname}{Proposition}
\providecommand{\remarkname}{Remark}
\providecommand{\theoremname}{Theorem}
\providecommand{\examplename}{Example}
\providecommand{\problemname}{Problem}

\usepackage{caption}
\title[Constructing stable Hilbert bundles ]{Constructing stable Hilbert bundles via Diophantine approximation}
\author{Yucheng Liu} 
\address{College of Mathematics and Statistics, Center of Mathematics, Chongqing University, Chongqing, 401331, China}
\email{noahliu@cqu.edu.cn}

\author{Biao Ma}
\address{School of Mathematical Sciences, East China Normal University, Shanghai, 200241, China}
\email{bma@math.ecnu.edu.cn}
\begin{document}
\begin{abstract}
\selectlanguage{english}
    On any complex smooth projective curve with positive genus, we construct Hilbert bundles that admit Hermitian--Einstein metrics. Our main constructive step is by investigating the arithmetic property of the upper half plane in Bridgeland's definition of stability conditions and its homological counterparts. 

    The main analytic ingredient in our proof is a notion called a well-approximating sequence of stable bundles. 
    %This notion compares a constant $L(X)$ that can be bounded by the geometric information of the Riemann surface $X$ with a constant $L_0(\theta)$ that depends only on the arithmetic information of the irrational number $\theta$. 
    This notion helps us to apply Diophantine approximation to Donaldson's functional and bound the $L^\infty$ norm of Hermitian--Einstein metrics.

    We further study the continuous structures, smooth structures, and holomorphic structures on such Hilbert bundles. We hope that this construction can shed new light on the geometric background of quantum field theory.
\end{abstract}

\maketitle

\tableofcontents
\section{Introduction}

Mumford introduced the classical notion of slope stability of holomorphic vector bundles on a compact Riemann surface in the 1960s (see \cite{Mumfordstability}), and this notion bridges between algebraic geometry and differential geometry. The famous work of Narasimhan--Seshadri \cite{Stableandunitaryvectorbundles} established the relation between stable holomorphic vector bundles and representations of certain Fuchsian groups. Their method is essentially algebraic geometric. Using the result of Narasimhan and Seshadri, Atiyah and Bott established a direct relation between Yang-Mills connections and stable vector bundles on Riemann surfaces (see \cite{TheYMequationsonRiemannsurfaces}). Using this reformulation of Atiyah and Bott, Donaldson reproved the theorem of Narasimhan-Seshadri via the differential geometric method (see \cite{Donaldsoncurve}), by establishing the equivalence between the existence of the Hermitian--Einstein metric and poly-stability. This equivalence, also known as the Kobayashi-Hitchin correspondence, was extended to higher dimensions by the celebrated works of Donaldson \cite{Donaldson85}, Uhlenbeck-Yau \cite{UhlenbeckYau86} and Donaldson \cite{Donaldsoninfinite}. 

In the 2000s, motivated by Douglas's work on D-branes and $\Pi$-stability (see \cite{douglas2002dirichlet}). Bridgeland introduced a general theory of stability conditions, which combines the notion of $t$-structures on triangulated categories and a generalization of Mumford's classical slope stability. An important ingredient in the definition is the upper half plane $\H$. In \cite{Continuumenvelops}, the authors investigate the relationship between the arithmetic property of the upper half plane $\H$ and the homological property of the triangulated category. For example, a Farey triangle in the upper half plane corresponds to distinguished triangles with some special properties in the triangulated category. This leads the authors to consider the following sequence of vector bundles and its colimit object.

For any complex elliptic curve and any irrational number $\theta$, there exists the following sequence of holomorphic maps of holomorphic vector bundles on the elliptic curve:
\begin{align}
\mathcal{E}_{0}\overset{f_{0}}{\to}\mathcal{E}_{1}\overset{f_{1}}{\to}\mathcal{E}_{2}\to\cdots\to\mathcal{E}_{k}\overset{f_{k}}{\to}\cdots.
\end{align}
Here, each $\mathcal{E}_i$ is a slope stable bundle whose slope is the $2i$-th convergent of $\theta$. And each $f_i$ is injective as a bundle map, the colimit object $\mathcal{E}_\infty$ in the category of quasi-coherent sheaves is a vector bundle of infinite rank in the sense of Drinfeld (see \cite{Infinitedimensionalvectorbundles}). In the smooth category, the colimit object $E_\infty$ also exists as a smooth topological vector bundle and is compatible with $\mathcal{E}_\infty$. The intuition is that such an infinite dimensional vector bundle $\E_{\infty}$ should be slope stable with slope $\theta$. 

In this paper, we first generalize the construction in \cite{Continuumenvelops} to any compact Riemann surface with positive genus. In fact, we prove the following result. 

\begin{thm}\label{thm:first result in intro}
   Let $X$ be a compact Riemann surface of genus $g(X)>0$,
and let $\theta$ be an irrational number.
Then there exist sequences of holomorphic vector bundles on $X$,
\begin{equation}
\mathcal{E}_0 \xrightarrow{f_0} \mathcal{E}_1
\xrightarrow{f_1} \mathcal{E}_2
\longrightarrow \cdots
\longrightarrow \mathcal{E}_k
\xrightarrow{f_k} \mathcal{E}_{k+1}
\longrightarrow \cdots,
\label{eq:seqofhol intro}
\end{equation}
in which each $\mathcal{E}_i$ is slope-stable with slope equal
to the $(2i)$-th continued fraction convergent of $\theta$,
and each $f_i$ is an injective holomorphic bundle map.
For each of these sequences, the colimit
\[
\mathcal{E}_\infty := \varinjlim_i \mathcal{E}_i
\]
is a simple vector bundle of infinite rank.
\end{thm}

%Here we briefly recall some terminologies in Diophantine approximation in number theory. See \cite[Charpter I]{LangbookonDiophantineapproximations} for this topic. Let $\alpha$ be a real number. A rational number $p/q$ with $p,q$ coprime is a \emph{best Diophantine approximation} if  for every rational $p'/q'\not=p/q$  with $0< q'\leq q$, \[
%|q'\alpha-p'|>|q\alpha-p|.
%\]
%The best Diophantine approximation is computed via continued fraction: 
%Here, if we use $\alpha=[a_0;a_1,a_2,a_3,\cdots]$ to denote the continued fraction. \[
%\theta =
%a_0+\cfrac{1}{a_1+\cfrac{1}{a_2+\cfrac{1}{a_3+\cdots}}}
%\]
 % The rational number $\beta_i:=[a_0;a_1,\cdots,a_i]=\frac{p_i}{q_i}$ is called $i$-th \emph{convergent} of $\theta$.  %Moreover, best Diophantine approximations are precisely convergents of  $\alpha$. If $\alpha $ is irrational, one defines the \emph{Lagrange number } of $\alpha$ by 
%\[L(\alpha):=\sup\{L>0|\exists\text{ infinite many }\frac{p}{q}\in\mathbb{Q}\text{ with}\ |\alpha-\frac{p}{q}|<\frac{1}{Lq^2}\}.\]
%Note $L(\alpha)$ can be $\infty$ in this  definition and is always no samller than $\sqrt{5}$ by Hurwitz's theorem, \cite[Theorem 1.21]{aigner2015markov}. It turns out that Diophantine approximation is closely related to arithmetic stability conditions in triangulated category. See Section \ref{Section:Arithmetic stability conditions and Lagrange numbers} and \cite{Continuumenvelops} for more details. 

Note that in general, for a fixed $\theta$, these colimit objects have gigantic moduli (see \cite{Continuumenvelops}). Motivated by the examples in Theorem \ref{thm:first result in intro}, we further investigate the slope stability for infinite dimensional vector bundles via the Kobayashi-Hitchin correspondence, i.e., constructing Hermitian--Einstein metrics on such infinite dimensional vector bundles. This construction can be viewed as the infinite dimensional generalization of Narasimhan--Seshadri's theorem, and this generalization leads us naturally to Alain Connes theory of noncommutative geometry.

The basic intuition is that rational numbers correspond to slopes of Hermitian--Einstein vector bundles of finite rank, whereas irrational numbers correspond to slopes of Hermitian--Einstein Hilbert bundles of infinite rank. For a sequence in (\ref{eq:seqofhol intro}), we hope to use the Hermitian--Einstein metrics on $\E_i$ to construct a Hermitian--Einstein metric on $\E_\infty$. Moreover, since the slopes of $\E_i$ are the best Diophantine approximations of $\theta$, we expect that the corresponding Hermitian--Einstein metric of $\E_i$ also has a certain good approximation property. 

%We show that this strategy indeed produces the desired Hilbert bundle after metric completion if we assume that the irrational number $\theta$ can be approximated nicely.

%To be more precise, we introduce two invariants. For $\theta\in\R\backslash\mathbb{Q}$, we assign a constant $L_0(\theta)\in[1,\infty]$ called \emph{even Langrange number} (see Definition \ref{defn:lagrange numbers}). Similarly to the standard Lagrange number, $L_0(\theta)$  depends only on the arithmetic information of $\theta$. And it is called \emph{attainable} if the supremum in Definition \ref{defn:lagrange numbers}  is a maximum. For each closed Riemann surface $(X,\omega )$, we assign a positive constant $L(X,\omega)$ called \emph{convergent threshold} (see Definition \ref{def:approx}). $L(X,\omega)$ is a geometric invariant of $(X,\omega)$ and can be bounded by a finite constant depending only on geometric information of $X$. 
%\begin{defn}[Geometrically well-approximable pair] Given $L_0(\theta)$ and $L(X,\omega)$ as above, we call  $(X,\omega)$ and $\theta$  a \emph{geometrically well-approximable pair} if  either
%\[L_0(\theta)>L(X,\omega),\] or $L_0(\theta)$ is attainable and equal to $L(X,\omega)$.
%\end{defn}
The main result of this paper is the following theorem.
\begin{thm}\label{thm:main thm in intro}
    Let $X$ be a compact Riemann surface with genus $g(X)>0$, $\omega$ be a K\"ahler form on $X$, and $\theta$ be an irrational number. Let $\E_\infty$ be a vector bundle constructed in Theorem \ref{thm:first result in intro} and let $E_\infty$ be the underlying smooth topological vector bundle. %If $(X,\omega,\theta)$ is geometrically well-approximable, then 
    Then there exists a holomorphic separable Hilbert bundle $(\textbf{E},H_\infty)$ with a unitary connection $\textbf{A}$ such that the following statements hold:
    \begin{enumerate}
         \item $H_\infty$ defines a Hermitian metric on $E_\infty$ and is constructed as a limit of a subsequence of Hermitian--Einstein metrics on $\mathcal{E}_k$ as $k\to\infty$. $\textbf{E}$ is given by fiberwise completion of $E_\infty$ with respect to $H_\infty$.\label{enu:1}
        \item   $\textbf{A}$ gives the holomorphic structure on $\textbf{E}$ (equivalently, $\textbf{A}$ is the Chern connection) that makes each $\mathcal{E}_k$ a holomorphic subbundle of $\textbf{E}$.
        \item   $H_\infty$ is a Hermitian--Einstein metric on $\textbf{E}$. \label{enu:4}
    \end{enumerate}

\end{thm}

The definition of a (separable) Hilbert bundle $\textbf{E}$ over $X$ is in the usual sense, which means a set of locally trivializations $\textbf{E}|_U\simeq U\times W$ for a fixed separable Hilbert space $W$ and transition maps satisfy additional continuous/smooth/holomorphic conditions. See Section \ref{section:Holomorphic Hilbert bundles} for more precise definitions.

%\textcolor{blue}{Add some basics of Diophantine approximation?}

%One can show that $L_0(\theta)\geq 1$, and $L_0(\theta)=1$ implies that $L_0(\theta)$ is attainable. Hence in the case where $L(X,\omega)=1,$ the results of Theorem \ref{thm:main thm in intro} hold for any irrational $\theta$. For example, $L(X,\omega)=1$ when $X$ is a complex elliptic curve with a flat K\"ahler metric (see Example \ref{example:elliptic curve}). 

%For any compact Riemann surface $X$ with $g(X)\geq 1$ and a generic choice of $\theta\in\R$, $(X,\omega,\theta)$ is geometrically well-approximable, hence, results in Theorem \ref{thm:main thm in intro} hold.  This is because of the following proposition.
%\begin{prop} \label{prop:full lebesg}
 % Let $X$ be any compact Riemann surface and $\omega$ be a K\"ahler class on $X$. The following subset of $\mathbb{R}$ \begin{align}
%      \{\theta| (X,\omega,\theta)\text{ is geometrically well-approximable} \}\label{eq:fullmeasure}
%  \end{align} is of full Lebesgue measure.    
%\end{prop}

\begin{rem}    
In fact, we will prove in \cite{Holomorphicsimplicity} that, for  almost all irrational numbers $\theta$ (in the full Lebesgue measure sense), every holomorphic Hilbert bundles in Theorem \ref{thm:main thm in intro} is holomorphically simple, i.e., the space of holomorphic sections of bounded operators $H^0_{bd}(X,End(\mathbf{E}))$ on $\mathbf{E}$ is $\mathbb{C}$. Whether the holomorphic simplicity holds for all irrational numbers remains open. Nevertheless, this holomorphic simplicity does not hold for rationals. For example, on flat elliptic curves, if $\theta\in\mathbb{Q}$,  a similar vector bundle of infinite rank can be constructed. However, one cannot construct indecomposable Hermitian--Einstein Hilbert bundle in this case. See Section \ref{subsec:final1}. 
\end{rem}

\subsection{Well-approximating sequences and Hermitian--Einstein Hilbert bundles}
To explain the main ingredients in our construction and proofs, we start with a brief account of Hermitian--Einstein metrics and slope stability on holomorphic vector bundles.

Let $(X,\omega)$ be a compact K\"ahler manifold of dimension $n$. Let $\mathcal{E}$ be a holomorphic vector bundle and $H$ be a Hermitian metric on $\E$. 

\begin{defn}\label{def:HEmetric}
    A unitary connection $A$ is called a Hermitian--Einstein connection if the curvature $F(A)$ is of $(1,1)$ type and satisfies 
\[\sqrt{-1}\Lambda F(A)=\gamma \text{Id}_\E.\]
where $\Lambda$ is the contraction with respect to $\omega$, $\gamma$ is a constant, and $\text{Id}_\E$ is the identity endomorphism. A Hermitian metric $H'$ is called the Hermitian--Einstein metric if its Chern connection is Hermitian--Einstein.
\end{defn}
Let $\deg\E=\int_X c_1(\E)\wedge\frac{\omega^{n-1}}{(n-1)!}$. The \emph{slope} of $\E$ is defined to be 
\begin{align}
    \mu(\E)=\frac{\deg\E}{\text{rk}\E}. 
\end{align} If $\mathcal{F}$ is a coherent sheaf of rank $r\geq 1$, then define $\deg{\mathcal{F}}=\deg(\det \F)$ where $\det\F$ is the determinant line bundle of $\F$, and define its slope as above.
\begin{defn} \label{def:slopestable}
    A holomorphic vector bundle $\E$ is called \emph{slope stable} if any proper coherent subsheaf $\mathcal{F}$ of $\E$ satisfies
    \[\mu(\F)<\mu(\E/\F).\]
    $\E$ is called \emph{slope poly-stable} if $\E$ is a direct sum of multiple slope stable bundles with the same slope.
\end{defn}

The classic Donaldson-Uhlenbeck-Yau Theorem states the following:
\begin{thm}[\cite{Stableandunitaryvectorbundles,Donaldson85,UhlenbeckYau86,Donaldsoninfinite}]
    A holomorphic vector bundle over a compact K\"ahler manifold admits a Hermitian--Einstein metric if and only if it is slope poly-stable.
\end{thm}

In the literature, the above theorem is sometimes called the (set-theoretic) Kobayashi-Hitchin correspondence. 

Now suppose that we have a sequence of stable holomorphic vector bundles on $X$
\begin{equation}
\mathcal{E}_{0}\xrightarrow[]{\iota_{0}}\mathcal{E}_{1}\xrightarrow[]{\iota_{1}}\mathcal{E}_{2}\xrightarrow[]{\iota_{2}}\mathcal{E}_{3}\to\cdots, \label{eq:-28 intro}
\end{equation} where $\iota_i$ is an injective holomorphic bundle map, and not an isomorphism for any $i\in\mathbb{N}$. Let $\E_{\infty}$ denote the colimit object of such a sequence, and denote 
\[
\theta=\mu(\mathcal{E_\infty})=\lim_{i\to\infty}\mu(E_{i}).
\]
%Since $\E_{i}$ is stable for any $i\in\N$, we have $\mu(\E_{\infty})=\sup_{\F}\mu(\F)$ 
%where the supremum runs through all the coherent subsheaves of $\mathcal{E}_{\infty}$. 
Moreover, $\coker(\iota_i)$ is a holomorphic vector bundle for any $i\geq 0$. In the special case where $X$ is a smooth projective variety over $\C$, for any affine open set $\mathrm{U}=\text{Spec}R$
of $X$, $\mathrm{H}^{0}(\mathrm{U},\mathcal{E}_{\infty})$ is a projective $R$-module
(see \cite[Corollary 5.5]{Continuumenvelops}), hence free by \cite[Corollary 4.5]{Bigprojectivemodulesarefree}. Therefore, $\mathcal{E}_{\infty}$
is an infinite rank vector bundle in the sense of Drinfeld (see \cite{Infinitedimensionalvectorbundles}). 

We now introduce the notion of a well-approximating sequence of stable bundles, which is one of the main ingredients in our construction.
\begin{defn}
\label{def:We-say-intro}A locally free subsheaf $\mathcal{S}$ of $\mathcal{E}_\infty$ is a \emph{good approximation} of $(\mathcal{E}_{\infty},\theta)$ if for any proper coherent subsheaf $\mathcal{S}_{0}\subset\mathcal{S}$ with
$1\leq\text{rk}\S_0<\text{rk}\S$, we have
\begin{equation}
(\theta-\mu(\S))\text{rk}\S<\text{rk}\mathcal{S}_{0}(\mu(\S)-\mu(\mathcal{S}_{0})).\label{eq:-30 intro}
\end{equation}
%where $L=L(X,\omega)$ is the convergence threshold given in Definition \ref{def:approx}. 
We call sequence
(\ref{eq:-28 intro}) a well-approximating sequence of $(\mathcal{E}_{\infty},\theta)$, if $\E_{i}$ is
a good approximation of $(\mathcal{E}_{\infty},\theta)$ for every $i\gg 0$. 
\end{defn}
A good approximating sequence of stable vector bundles plays a similar role as a sequence of convergents in Diophantine approximation but in the context of coherent sheaves. In particular, if $\mu(\E_i)$ is a convergent of $\theta$, then $\E_i$ is a good approximation for $(\E_\infty,\theta)$. Therefore, Theorem \ref{thm:main thm in intro} is a direct consequence of the following theorem.

\begin{thm}
  \label{thm:main2}  Notations as above. Suppose that $X$ is a compact Riemann surface. If (\ref{eq:-28 intro}) is a well-approximating sequence of $(\mathcal{E}_\infty,\theta)$, then there exists a holomorphic Hilbert bundle $(\textbf{E},H_\infty)$ with a unitary connection $\textbf{A}$ such that (\ref{enu:1})-(\ref{enu:4}) in Theorem \ref{thm:main thm in intro} hold.
\end{thm}

We briefly explain the proof of Theorem \ref{thm:main2}. If $\S$ is a good approximation of $(\E_{\infty},\theta)$, we use Donaldson's functional to prove that the sequence of metrics $\{H_i|_{\S}\}$ converges in a suitable norm, where $H_i$ is the Hermitian--Einstein metric on $\E_i$ with a suitable scaling. We are inspired by Donaldson's observation that slope stability implies the properness of Donaldson's functional \cite{Donaldson85}. Instead of applying stability assumption on a fixed vector bundle, our argument hinges on Diophantine approximation on Donaldson's functional for a sequence of vector bundles.

For a well-approximating sequence, let $H_{i,\infty}$ denote the convergent metric on $\E_i$ from the sequence $\{H_j|_{\E_i}\}_{j\geq i}$. We use a diagonal argument to normalize the metrics $\{H_{i,\infty}\}$ to be compatible with restrictions. 

An intermediate step in the proof is to construct a colimit object of (\ref{eq:-28 intro}) in the smooth category. We use convenient calculus (see \cite{kriegl1997convenient}) to construct a smooth topological vector bundle $E_\infty$. Then $\{H_{i,\infty}\}$  gives us a metric $H_\infty$ on $E_{\infty}$. We prove that such a metric is Hermitian--Einstein.

We complete the infinite dimensional vector bundle $E_{\infty}$ with respect to such a metric $H_\infty$. To construct a Hilbert bundle structure, a key ingredient is Uhlenbeck's gauge fixing theorem \cite{UhlenbeckLp}. A version of Uhlenbeck's theorem that is independent of rank of bundles is proved to construct continuous and smooth local trivializations. We then use the Koszul-Malgrange theorem \cite{KoszulMalgrange58} to obtain a holomorphic structure. 

Many of our arguments can be applied to K\"ahler manifolds of higher dimensions. Hopefully, we can generalize part of our construction to higher dimensions in future research. The main problem in higher dimensions is that the total curvature tensor $F$ of a Hermitian--Einstein connection has no \emph{a priori} bound.  Without such bound, we cannot directly apply Bochner-Lichnerowicz-Weitzenb\"ock formula to bound Donaldson's functional. Also, we cannot use Uhlenbeck's gauge fixing lemma, and it will be troublesome to find a smooth trivialization of Hilbert bundles. Additional technical input is needed in higher dimensions to overcome this problem. For instance, our construction should work for a class of stable bundles with a uniform total curvature bound.

A natural question is the uniqueness of the Hermitian--Einstein Hilbert bundle in our construction. For a given sequence (\ref{eq:-28 intro}), our current construction relies on repeatedly choosing subsequences. The question remains whether the resulting bundle is independent of such choices. We hope to address this question and study some related uniqueness problems in the future. 

\subsection{Stable Hilbert bundles and quantum physics}
In this subsection, we discuss the correspondences between some ingredients in quantum physics and geometry of stable Hilbert bundles in Theorem \ref{thm:main thm in intro}. 

Let $(X,\omega)$ be a compact K\"ahler Riemann surface with $g(X)\geq 1$ and normalized with $\mathrm{vol}(X)=2\pi$, $\textbf{E}$ be a stable Hilbert bundle as in Theorem \ref{thm:main thm in intro} and $\mathbf{A}$ be a unitary connection with \(\sqrt{-1} \mathbf{F}=\theta\omega\mathrm{Id}\). Denote $H=\sqrt{-1}\mathbf{A}$. Let $\gamma:[0,1]\to X$ be a smooth curve. The parallel transport of  a local section $ \phi(t)$ along $\gamma$ becomes $$\frac{\partial\phi(t)}{\partial t}=\sqrt{-1}(\iota_{\dot{\gamma}}H)\phi(t).$$ Here $\iota_{\dot\gamma}H$ is self-adjoint as $\mathbf{A}$ is a unitary connection. Hence this equation closely resembles Schr\"odinger equation, and the parallel transport resembles the Schr\"odinger evolution of states. This viewpoint, regarding the Hamiltonian as the contracted Chern connection 1-form also appeared in Simon's work \cite{simon1983holonomy}, which illustrates Berry's phase as the holonomy of a Chern connection.

Denote by $J$ the complex structure on $X$. Let $(e,Je)$  be an orthonormal basis of the tangent space $T_pX$ at a point $p\in X$, and let $\frac{\partial}{\partial x}, J\frac{\partial}{\partial x}$ be two local vector fields that extend $(e,Je)$ with  $[\frac{\partial}{\partial x},J\frac{\partial}{\partial x}]_p=0.$ Then the curvature formula shows that \begin{equation}\label{eq:nabla}
    [\nabla_{\frac{\partial}{\partial x}}, \nabla_{J\frac{\partial}{\partial x}}]_p=-\sqrt{-1} \theta \operatorname{Id}, \notag
\end{equation} which resembles the Heisenberg uncertainty principle $$[\hat{x},\hat{p}]=\sqrt{-1}\hbar.$$ Intuitively, Heisenberg uncertainty principle corresponds to the fact that the curvature $\theta$ of the Chern connection is nonzero. And the pair of canonical conjugate observables is related by the complex structure $J$ on the tangent space. 

\begin{table}[h]
\centering
\renewcommand{\arraystretch}{1.8} 
\setlength{\tabcolsep}{14pt}      

\begin{tabular}{|
>{\centering\arraybackslash}m{0.38\textwidth}|
>{\centering\arraybackslash}m{0.42\textwidth}|
}
\hline
\makecell{\bfseries Quantum physics}
&
\makecell{\bfseries Geometry of stable\\ \bfseries Hilbert bundles}
\\
\hline

\makecell{Schrödinger evolution}
&
Parallel transport
\\
\hline

\makecell{Heisenberg uncertainty\\ principle}
&
Curvature $\neq 0$
\\
\hline

\makecell{Canonical conjugate observables}
&
$\nabla_{\frac{\partial}{\partial x}},$ $\nabla_{J\frac{\partial}{\partial x}}$
\\
\hline
\makecell{Annihilation/Creation  operators}
&
$\nabla_{\frac{\partial}{\partial z}},$ $\nabla_{\frac{\partial}{\partial \bar{z}}}$
\\
\hline

\makecell{Fourier transformation of\\  canonical conjugate observables}
&
Complex structure on tangent space
\\
\hline

\makecell{Foldy--Wouthuysen \\ gauge  transformation}
&
Gauge transformation of  connection 1-form
\\
\hline

\end{tabular}\caption{A dictionary between quantum physics and geometry of $\mathbf{E}$}
\label{tab:qm-geometry}
\end{table}

Moreover, let $\mathcal{F}$ be the classical Fourier transform. It is well known that $$\mathcal{F}^{-1}\hat{x}\mathcal{F}=\hat{p},\ \mathcal{F}^{-1}\hat{p}\mathcal{F}=-\hat{x}.$$ So the Fourier transform on canonical conjugate observables behaves like the complex structure $J$.

If we take  $(\nabla_{\frac{\partial}{\partial x}}, \nabla_{J\frac{\partial}{\partial x}})$ to be a pair of canonically conjugate observables, then the associated annihilation and creation operators are, up to a constant factor, $(\nabla_{\frac{\partial}{\partial z}}, \nabla_{\frac{\partial}{\partial \bar{z}}})$. 

Besides Berry's connection and curvature, there is also a well-known gauge transformation in quantum mechanics, named Foldy--Wouthuysen gauge transformation (see \cite{foldy1950dirac}). In this case, the Hamiltonian transforms in the same way as a connection 1-form.
\subsection{Stable Hilbert bundles and Noncommutative geometry}

In this subsection, we discuss the relations among stable Hilbert bundles, noncommutative geometry, and noncommutative field theory. As shown in Table \ref{tab:qm-geometry}, the nonzero curvature of our stable Hilbert bundle is analogous to the Heisenberg uncertainty principle. This analogy suggests that these three subjects may be connected through a common underlying noncommutative structure.

First, we recall the flat noncommutative spacetime $\mathbb{R}_{\mathrm{NC}}^d$ (see \cite{Instantonsonnoncommutativespacetime}). The coordinate functions $x_i$ on this noncommutative manifold satisfy the following  commutation relation $$[x^i, x^j]=\sqrt{-1}\theta^{ij},$$ where $(\theta^{ij})$ is a constant antisymmetric matrix. The algebra (over $\mathbb{C}$) generated by these $x_i$ together with convergence conditions on the allowed expressions in the $x_i$ is called the noncommutative spacetime. If moreover $d=2r$ and the matrix $\{\theta^{ij}\}$ is nondegenerate, it produces a complex structure $J$. This allows us to choose complex coordinates 
 $\{z_a,\bar z_{a}\},\ a=1,2,\cdots,r$ such that \begin{equation}
[z_a, \bar{z}_b]=-2\theta_a\delta_{ab}, \ \theta_a>0.\label{eq:commutator}\end{equation}

If each $\theta_a$ is an irrational numbers for $1\leq a\leq r$, let $X_a$ be a closed Riemann surface with $g(X_a)\geq 1$ and  $\mathrm{vol}(X_a)=2\pi$, and let $\textbf{E}_a$ be a stable Hilbert bundle as in Theorem \ref{thm:main thm in intro} of slope $\theta_a$.  The projective flat Chern connection of $\textbf{E}_a$ locally gives operators $\nabla_{z_a}$ and $\nabla_{\bar z_a}$ that satisfy the bracket condition \eqref{eq:commutator}. Therefore, it makes the tangent space $T_pX$ of $X_a$ at any point $p$ into a noncommutative space $\mathbb{R}^2_{\mathrm{NC}}$. We consider $\mathcal{X} \coloneqq X_1\times X_2\times \cdots \times X_r$ with $\pi_i$ the projection to the $a$-th piece. Take the vector bundle $$\mathbf{E}_1\boxtimes\textbf{E}_2\boxtimes \cdots\boxtimes \textbf{E}_r:=\bigotimes _{a=1}^r\pi_a^*\mathbf{E}_a$$ 
on $\mathcal{X}$. Then the tangent space of $\mathcal{X}$ at a point $p=(p_1,\cdots ,p_r)$ corresponds to the noncommutative $\mathbb{R}^{2r}_{\mathrm{NC}}$ with the standard form $\{\theta_a\}$. Moreover, we can view the germ of smooth sections of this bundle at the point $p$ as a module over the algebra of noncommutative space $\mathbb{R}_{\mathrm{NC}}^{2r}$.

Stable Hilbert bundles also provide many concrete examples of noncommutative Riemann surfaces. We consider the parallel transport associated with a stable Hilbert vector bundle. The projectively flat connection gives rise to a monodromy representation, yielding  a projective unitary representation of the fundamental group \(\pi_1(X)\) on a separable Hilbert space. The associated projective unitary representation naturally determines a \(C^*\)-algebra, which is generated by $U_1,V_1,\cdots,U_g,V_g$ and satisfies \[\prod_{i=1}^g U_i V_i U_i^{-1}  V_i^{-1}=\exp(2\pi \sqrt{-1}\theta)\mathrm{Id}.\] The $C^*$-algebra corresponds to a noncommutative Riemann surface introduced in  \cite{BERTOLDI2000323}. 

In the special case where $X$ is a complex elliptic curve, we recover the classical example of a noncommutative torus $X_\theta$.  We state three possible approaches to go from our construction to the noncommutative tori (see Section \ref{subsection: a maze to noncommutative tori} for details). The first is categorical: it uses a theorem of Polishchuk \cite{Classficationofholomorphicvectorbundles}, which identifies the category of holomorphic vector bundles over a noncommutative torus with a tilted heart of \(D^b(X)\). The second proceeds through the monodromy representations of the stable Hilbert bundles constructed above. The last one conceptually uses homological mirror symmetry counterparts of the infinite dimensional vector bundles $\E_{\infty},$ which are nonclosed Lagrangians in a torus. Given the same initial data, these three ways should give us the noncommutative tori which are Morita equivalent.

\subsection{Inverse limits and rigged Hilbert bundles}

There is a parallel construction for an inverse system of vector bundles. Suppose that we have a sequence of stable  holomorphic vector bundles
\begin{align}
\cdots\to\mathcal{E}_{k}\overset{f_{k}}{\to}\cdots\to \mathcal{E}_{2}\overset{f_{2}}{\to}\mathcal{E}_{1}\overset{f_{1}}{\to}\mathcal{E}_0.\label{eq:inverseseqofhol}
\end{align}
Here, each $f_k$ is surjective as a bundle map and not an isomorphism. Let $\mathcal{E}_{\infty}:=\varprojlim\mathcal{E}_{i}$ be the inverse limit formed in the category of quasi-coherent sheaves. We denote $\theta=\lim_{i\to\infty}\mu(\mathcal{E}_{i})$. Since $\mu(\mathcal{E}_{i})$ is a decreasing sequence, $\theta\in[-\infty,\infty)$ exists. 

If  $\theta\not=-\infty$, we take the dual bundle $\E_i^\vee$ of $\E_i$, then we have a direct system of bundles 
\begin{align}
\mathcal{E}^\vee_{0}\overset{f_{1}^*}{\to}\mathcal{E}^\vee_{1}\overset{f_{2}^*}{\to}\mathcal{E}^\vee_{2}\to\cdots\overset{f_{k}^*}{\to}\mathcal{E}^\vee_{k}{\to}\cdots.\label{eq:dualsequ}
\end{align}
Each $\E_i^\vee$ is also stable with slope equal to $-\mu(\E_i)$, and each $f_i^*$ is injective. So we return to the case considered before. 

Suppose that (\ref{eq:dualsequ}) is a well-approximating sequence (we then also call (\ref{eq:inverseseqofhol}) a well-approximating sequence). By Theorem \ref{thm:main2}, there exists a holomorphic Hermitian Einstein Hilbert bundle $\textbf{E}^\vee$ together with holomorphic injection maps $\iota_i:\E_i^\vee\to\textbf{E}^\vee$ such that the following diagram commutes. 
\begin{align}
\xymatrix{ &  & \mathbf{} & \mathbf{E}^\vee\\
\mathcal{E}^\vee_{0}\ar[r]_{f_1^*}\ar[rrru]^{\iota_{0}} & \mathcal{E}^\vee_{1}\ar[r]_{f_2^*}\ar[rru]_{\iota_{1}} & \mathcal{E}^\vee_{2}\ar[r]_{f_3^*}\ar[ru]_{\iota_{2}} & \mathcal{E}^\vee_{3}\ar[r]\ar[u]_{\iota_{3}} & \cdots\ar[r]_{f_k^*} & \mathcal{E}^\vee_{k}\ar[llu]^{\iota_{k}}\ar[r] & \cdots
}
\end{align}
We take the complex linear dual $\textbf{E}$ of $\textbf{E}^\vee$. Then $\textbf{E}$ is also a Hermitian--Einstein Hilbert bundle on $X$ together with holomorphic surjective bundle maps $\iota_k^*:\textbf{E}\to\E_i.$
\begin{align}
 \xymatrix{ &  & \mathbf{} & \mathbf{E}\ar[llld]_{\iota_{0}^{*}}\ar[lld]^{\iota_{1}^{*}}\ar[ld]^{\iota_{2}^{*}}\ar[d]^{\iota_{3}^{*}}\ar[rrd]^{\iota_{k}^{*}}\\
\mathcal{E}_{0} & \mathcal{E}_{1}\ar[l]^{f_{1}} & \mathcal{E}_{2}\ar[l]^{f_{2}} & \mathcal{E}_{3}\ar[l]^{f_{3}} & \ar[l]\cdots & \mathcal{E}_{k}\ar[l]^{f_{k}} & \ar[l]\cdots
}
\end{align}

The above constructions fit nicely in the language of rigged Hilbert spaces. A rigged Hilbert space or Gel'fand triple is a triple \[V\overset{i}{\to}H\overset{i^\times}{\to}V^\times.\] Here $V$ is a topological vector space, $H$ is a Hilbert space and $V^\times$ is the conjugate dual of $V$. The map $i$ is a continuous injective map with dense image and $i^\times$ is the conjugate adjoint of $i$. Roughly speaking, since the conjugate $\overline{\textbf{E}}$ and $\textbf{E}^\vee$ are isometric, we should expect 2 rigged Hilbert bundles or Gel'fand triples: $\E^\vee_\infty\to\textbf{E}^\vee\to\overline{\E_\infty}$ and $\overline{\E^\vee_\infty}\to\textbf{E}\to\E_\infty$. Also note that for sequence (\ref{eq:seqofhol intro}), the dual sequence is a sequence constructed via the odd convergents of $-\theta$ (see \cite[Section 4]{Continuumenvelops} for more details of these inverse limit objects). Since the construction of rigged Hilbert bundles is certainly an interesting topic itself, we will leave the detailed discussion for later work.

% \\ \hspace*{\fill} \\
\subsection{Comparison with  Hilbert bundles in the literature} There is a plethora of papers that study Hilbert bundles. We can only discuss a few of them due to the authors' ignorance. 

In \cite{Directimages}, the authors provide a very general framework for studying Hilbert bundles. This framework is very suitable for the results in \cite{FlatconnectionsandGQ} and \cite{GeometricquantizationofCSgaugetheory}. Our construction is different from their framework in the following two aspects.

The first aspect is that the stability of vector bundle plays a vital role in our construction, while \cite{Directimages} provides a much more general framework of the fields of Hilbert spaces. Therefore, our Hilbert bundles are special cases of the fields of Hilbert spaces in \cite{Directimages}, and enjoy better properties.

The second aspect is the difference of the setups. In \cite{Directimages}, the authors consider a family of Riemannian manifolds, and they associate a Hilbert space for each member in this family. They get a field of Hilbert spaces by putting these Hilbert spaces together. In this paper, we construct a family of Hilbert bundles for a given compact Riemann surface. Although it is very interesting to combine these two ideas, we will not pursue this direction in this paper.

We should mention a particular family of holomorphic Hilbert bundles called Bergman bundles that fits the framework in \cite{Directimages}. Given a holomorphic fibration $\pi:Y\to X$ between two K\"ahler manifolds and a Hermitian holomorphic vector bundle $\E$ over $Y$, one considers a holomorphic Hilbert vector bundle $\textbf{E}$ over $X$ whose fiber at $x$ is the $L^2$ holomorphic sections of $\E|_{\pi^{-1}(x)}$. Note that $\textbf{E}$ can be viewed as the direct image of $\E$ under $\pi$. If $\pi$ is not proper, then $\textbf{E}$ in general has infinite rank. There have been many important applications and a huge amount of studies of Bergman bundles in several complex variables and complex geometry. See \cite{berndtsson2009curvature,berndtsson2011,guan2015solution,Berndtsson2016,Wang2017AIF,demailly2020bergman} for a rather incomplete list. 

For finite rank holomorphic bundles over K\"ahler manifolds with boundaries, Donaldson \cite{Donaldson1992BoundaryVP} proved the existence of Hermitian--Einstein metrics with Einstein factor $0$ and prescribed boundary values. It amounts to the solvability of a Dirichlet boundary value problem. In \cite{Wu20Dirichlet}, K.-R. Wu generalizes Donaldson's results to holomorphic Hilbert bundles over Riemann surfaces with boundary. In contrast, our work focuses on the construction of holomorphic Hilbert bundles whose bundle structures are not given. Moreover, there is no obstruction to solvability of boundary value problems while, for closed Riemann surfaces, one has to consider stability conditions.
\\ \hspace*{\fill} \\
\noindent\textbf{Outline of the paper.} In Section \ref{Section:Arithmetic stability conditions and Lagrange numbers}, we slightly generalize some constructions and results in \cite{Continuumenvelops}. In Section \ref{Section:Estimates for second fundamental forms}, we estimate the second fundamental form from a short exact sequence of holomorphic bundles. In Section \ref{section:uniform} and Section \ref{Section:Convergence of metrics}, we use well-approximating sequences to show the convergence of Hermitian--Einstein metrics to a sequence of metrics $H_\infty=\{H_{i,\infty}\}$.  We also show that the sequence (\ref{eq:seqofhol intro}) has a well-approximating subsequence under the assumption of Theorem \ref{thm:main thm in intro}.  In Section \ref{sec:smoothtopo}, we construct the underlying smooth vector bundle $E_\infty$ of $\E_{\infty}$. We show that $H_{\infty}$  is a Hermitian metric on $E_{\infty}$. In  Section \ref{Section:Curvature convergence and gauge fixing},  we prove that the metric $H_{\infty}$ is Hermitian--Einstein on $E_\infty$ in some sense. In Section \ref{section:Holomorphic Hilbert bundles}, we study the continuous, smooth and holomorphic structures on the completion of $\E_{\infty}$ ($E_\infty$)  with respect to the metric $H_{\infty}$. We prove that its completion is a projectively flat vector bundle. This finishes the proof of Theorem \ref{thm:main thm in intro} and Theorem \ref{thm:main2}. In Section \ref{section:final remarks}, we discuss some related topics, such as counterexamples for rational $\theta$ and triple ways to noncommutative tori from our construction. In the appendix, we include some discussions of function spaces with values in Banach spaces and a proof of Uhlenbeck's gauge fixing theorem for Hilbert bundles.
 \\ \hspace*{\fill} \\
\noindent\textbf{Notation and Conventions.} We use $\mathbb{H}$ to denote the complex upper half plane. For a number $r$ in $\mathbb{Q}_{\infty}\coloneqq \mathbb{Q}\cup\{\infty\}$, we always write it in the standard form $r=\frac{p}{q}$, where $\gcd(p,q)=1$ and $q\geq0$.  For such an expression, the complex number $-p+\sqrt{-1}\cdot q$ is a primitive integral vector in $\mathbb{H}\cup \mathbb{R}_{< 0} $. We use the numbers in $\mathbb{Q}_{\infty}$ and the primitive integral vectors in $\mathbb{H}\cup \mathbb{R}_{< 0} $ interchangeably.

We frequently switch our point of view between algebraic geometry and complex differential geometry. To emphasize the difference, we usually use letters in calligraphic font such as $\mathcal{E},\mathcal{F},\mathcal{G}$ to denote holomorphic vector bundles (and coherent sheaves in occasions), and use $E,F,G$ to denote the associated smooth vector bundles. We use $\text{rk}E$ to denote the rank of a finite rank vector bundle $E$.

As any closed Riemann surface is algebraic, we use compact Riemann surfaces and smooth complex projective curves as synonyms unless topological issues are involved. We denote $\omega$ a fixed K\"{a}hler form and always assume $\text{vol}(X,\omega)=1$. 

Let $V$ be a Banach space. Let $M$ be a Riemannian manifold.  We denote $\Omega^p(V)=\Lambda^p T^*M\otimes V$. Fix an orthonormal frame $\{\eta_i\}$ for $\Lambda^p T^*M$. For any $\alpha\in \Omega^p(V)$, we can write $\alpha =\sum_i \eta_i\otimes v_i$. Then, we put a norm on $\Omega^p(V)$ by taking $|\alpha|_{\Omega^p(V)}=\sum_i|v_i|_V$. Similar construction extends to Banach bundles over $M$.

We use various function spaces with values in Banach spaces such as $L^p$ spaces, Sobolev spaces, and H\"older spaces. For their definitions, see Appendix \ref{subsec:norms}.  
\hspace*{\fill} \\

\section{Arithmetic stability conditions}\label{Section:Arithmetic stability conditions and Lagrange numbers}
\subsection{Arithmetic stability conditions} In this subsection, we recall and generalize some constructions in \cite{Continuumenvelops}, and all categories in this subsection are assumed to be $k$ linear over a field $k$.
Firstly, let us recall what is a (weak) stability condition on a triangulated category. We need the following notion of (weak) stability functions to define a (weak) stability condition.

\begin{defn} \label{defn:weak stability function}
			Let $\mathcal{A}$ be an abelian category. We call a group homomorphism $Z:K_0(\mathcal{A})\rightarrow \mathbb{C}$ a \textit{weak stability function} on $\mathcal{A}$ if, for $E\in \mathcal{A}$, we have $Im(Z(E))\geq0$,with $Im(Z(E))=0 \implies  Re(Z(E))\leq0$. 
			Moreover, if for $E\neq 0,$ $Im(Z(E))=0\implies Re(Z(E))< 0$, we say that $Z$ is a \textit{stability function}.
		\end{defn}
		\begin{rem}
		    For any object $E\in\mathcal{A}$, we also call $Z(E)$ the central charge of $E$. 
		\end{rem}
		\begin{defn}\label{defn:weak stability condition}
			A \textit{(weak) stability condition} on a triangulated category $\mathcal{D}$ is a pair $\sigma=(\mathcal{A},Z)$ consisting of the heart of a bounded t-structure $\mathcal{A}\subset\mathcal{D}$ and a (weak) stability function $Z:K_0(A)\rightarrow \mathbb{C}$ such that the following is satisfied:
			
			 The function $Z$ allows us to define a slope for any object $E$ in the heart $\mathcal{A}$ by
			
			$$\mu_{\sigma}(E):=\begin{cases} -\frac{Re(Z(E))}{Im(Z(E))}\ &\text{if} \  Im(Z(E))> 0,\\ +\infty &\text{otherwise.} \end{cases}$$

			The slope function gives a notion of stability: A nonzero object $E\in \mathcal{A}$ is $\sigma$ semistable if for every proper subobject $F$, we have $\mu_{\sigma}(F)\leq  \mu_{\sigma}(E/F)$. Moreover, if $\mu_{\sigma}(F)<  \mu_{\sigma}(E/F)$ holds for any proper subobject $F\subset E$, we say that $E$ is $\sigma$ stable.
			
			We require any nonzero object $E$ of $\mathcal{A}$ to have a Harder--Narasimhan filtration in $\sigma$ semistable ones, i.e., there exists a unique filtration $$0=E_0\subset E_1 \subset E_2\subset \cdots \subset E_{m-1} \subset E_m=E$$ such that $E_i/E_{i-1}$ is  $\sigma$ semistable and $\mu_{\sigma}(E_i/E_{i-1})>\mu_{\sigma}(E_{i+1}/E_i)$ for any $1\leq i\leq m$. The quotient $E_{i}/E_{i-1}$ is called the $i$-th HN factor of $E$.
			
			%(b) (Support property) The central charge $Z$ factors as $K_0(\mathcal{D})\xrightarrow{v} \Lambda\xrightarrow{g} \mathbb{C}$. And there exists a quadratic form $Q$ on $\Lambda_{\mathbb{R}}$ such that $Q|_{ker(g)}$ is negative definite and $Q(v(E))\geq 0$ for any $\sigma$ semistable object $E\in\mathcal{A}$.
		\end{defn}
\begin{rem}\label{remark:main example}
    For any smooth projective variety $X$ in dimension $n$ over an algebraically closed field $k$, let $\mathrm{H}$ be an ample divisor on $X$, the pair $(\mathrm{CohX},-c_1\cdot \mathrm{H}^{n-1}+\sqrt{-1}\cdot \mathrm{rk})$ is a weak stability condition on $D^b(X)$. And it is a stability condition when $n=1$.
\end{rem}

\begin{rem}
    In fact, the usual definition of stability condition contains a technical condition called the support property, introduced in \cite{kontsevich2008stability}. This condition is essential to the deformation theorem of stability conditions. However, it will not play a role in this paper, so we do not include it in the definition.
\end{rem}
\begin{rem}
    Given a weak stability condition $\sigma=(\mathcal{A},Z)$ on a triangulated category $\mathcal{D}$, one can construct full subcategories $\mathcal{P}_{\sigma}(\phi)$ for any $\phi\in\mathbb{R}$ in the following way. For any $\phi\in(0,1]$, let $$\mathcal{P}_{\sigma}(\phi)\coloneqq \{E\in\mathcal{A}| \text{$E$ is semistable of slope $-\cot(\pi\phi)$}\}.$$ More generally, for any $\phi\in(n,n+1]$, we let $\mathcal{P}_{\sigma}(\phi)\coloneqq \mathcal{P}_{\sigma}(\phi-n)[n]$, by conevention the zero object is in $\mathcal{P}(\phi)$ for any $\phi\in\R$. In fact, these full subcategories form a slicing of $\mathcal{D}$ (see \cite[Definition 3.3]{bridgeland2007stability} for the definition of slicing). Moreover, there is an equivalent way to define the stability condition in terms of slicing and central charge (see \cite[Proposition 5.3]{bridgeland2007stability}).
\end{rem}

Note that there is a distinguished subcategory $\mathcal{A}_0$ for any weak stability condition $(\mathcal{A},Z)$. Here $\mathcal{A}_0$ is the full subcategory of $\mathcal{A}$ consisting of all objects whose central charges are $0$. By \cite[lemma 4.11]{Filtrationsandtorsionpairs}, we know that $\mathcal{A}_0$ is an abelian subcategory of $\mathcal{A}$.
\begin{defn}\label{Defn:quasisemistability}
		We call $\mathcal{A}_0$ the \textit{abelianizer} of the weak stability condition $\sigma=(\mathcal{A},Z)$. If there exists a short exact sequence $$0\rightarrow Q'\rightarrow E\rightarrow K'\rightarrow 0,$$ where $Q'\in\mathcal{A}_0$ and $K'\in\mathcal{P}_{\sigma}(\phi)$, we say that $E$ is \textit{quasi-semistable}.
	\end{defn}

    \begin{rem}
    By \cite[Proposition 4.16]{Filtrationsandtorsionpairs}, we know that the full subcategory $$\mathcal{A}_{\sigma}(\phi)\coloneqq\{E\in\mathcal{A}\ |E\ \text{is quasi-semistable and $Z(E)\in exp(i\pi \phi)\cdot \mathbb{R}_{\geq 0}$}\}$$  is an abelian subcategory.
 \end{rem}
\begin{defn}
    A triple $$(\frac{p_1}{q_1},\frac{p_2}{q_2},\frac{p_3}{q_3})$$ of numbers in $\mathbb{Q}$ is called a \textit{Farey triangle} if the following conditions are met. \begin{enumerate}
        \item we have $\frac{p_1}{q_1}<\frac{p_2}{q_2}<\frac{p_3}{q_3}$, and $p_2=p_1+p_3, \ q_2=q_1+q_3;$
        \item $p_2q_1-q_2p_1=p_3q_2-p_2q_3=1.$
    \end{enumerate}
\end{defn}

\begin{rem}
    For two numbers $\frac{p_1}{q_1},\frac{p_2}{q_2}$ in $\mathbb{Q}$, we say these two numbers are connected by a Farey geodesic if $p_2q_1-q_2p_1=1$. One can consult \cite[Section 3]{Continuumenvelops} for a quick review of the geometric picture of the Farey geodesics and the Farey triangles.
\end{rem}

The following definition is taken from \cite[Definition 5. 36]{Continuumenvelops}.
\begin{defn}\label{Defn: arithmetic stability condition}
    A weak stability condition $\sigma=(\mathcal{A},Z)$ is called an \textit{arithmetic weak stability condition} if the following conditions are satisfied. \begin{enumerate}
        
        \item The image of the central charge $Z$ lies in $\mathbb{Z}+\sqrt{-1}\cdot\mathbb{Z}\subset\mathbb{C}$.
        \item For any primitive integral vector $(p,q)\in\mathbb{H}\cup\mathbb{R}_{<0}$, there exists a stable object $E$ with $Z(E)=-p+\sqrt{-1}\cdot q$.
        \item For any two primitive integral vectors $(p_1,q_1), (p_2,q_2)\in\mathbb{H}\cup\mathbb{R}_{<0}$ that are connected by a Farey geodesic and $\frac{p_1}{q_1}<\frac{p_2}{q_2}$, and a given stable object $E$ with $Z(E)=-p_1+\sqrt{-1}\cdot q_1$, there exists a stable object $F$ with $Z(F)=-p_2+\sqrt{-1}\cdot q_2$ and $\Hom(E,F)\neq 0$.
    \end{enumerate}
\end{defn}

In the following proposition,  when we say that a morphism $f$ is an injection or a surjection,  we implicitly mean that $f$ is an injection or a surjection in $\mathcal{A}$. 
    \begin{prop}\label{prop:morphisms in farey triangles}
        Let $\sigma=(\mathcal{A},Z)$ be a weak arithmetic stability condition on a triangulated category $\mathcal{D}$, a triple $$(\frac{p_1}{q_1},\frac{p_2}{q_2},\frac{p_3}{q_3})$$ be any Farey triangle, and $E_1,E_2,E_3$ be stable objects in $\mathcal{A}$ with $Z(E_i)=-p_i+\sqrt{-1}\cdot q_i$ for $i=1,2,3$. The following statements are true. \begin{enumerate}
            \item Any nontrivial morphism $f:E_1\rightarrow E_2$ is an injection, and $\coker(f)$ is quasi-semistable with central charge $-p_3+\sqrt{-1}\cdot q_3$.
            \item For any nontrivial morphism $f:E_2\rightarrow E_3$, we have that $\ker(f)$ is stable with central charge $-p_1+\sqrt{-1}\cdot q_1$, and $Z(\coker(f))=0$.
            \item For any nontrivial morphism $f:E_3\rightarrow E_1[1]$, we have that  either $\mathrm{Cone}(f)$ is stable, or there exists  a nonzero object $G\in\mathcal{A}_0$ such that we have the following  short exact sequences in $\mathcal{A}$,

            $$0\rightarrow E_1\rightarrow E_1'\rightarrow G\rightarrow 0$$
            
            $$0\rightarrow E_3'\rightarrow E_3\rightarrow G\rightarrow 0$$ 
            
            $$0\rightarrow E_3'\rightarrow \mathrm{Cone}(f)[-1]\rightarrow E_1'\rightarrow 0,$$ and $E_1', E_3'$ are stable objects.
        \end{enumerate}
    \end{prop}
\begin{proof}
    The proof is similar to the proof of \cite[Theorem 4.3]{Continuumenvelops}. We include the proof of (3) for the reader's convenience.

    Firstly, letting $E_2$ denote $\mathrm{Cone}(f)[-1]$, one can see that $E_2\in\mathcal{A}$. Suppose that $$E_{2,1}\subset E_{2,2}\subset \cdots \subset E_{2,n}=E_2$$ is the Harder--Narasimhan filtration of $E_2$, and let $Q_i$ denotes the $i$-th Harder--Narasimhan factor of $E_2$. Applying the functors $\Hom(E_{2,1},-)$ and $\Hom(-,Q_n)$ to the short exact sequence $$0\rightarrow E_1\rightarrow E_2\rightarrow E_3\rightarrow 0,$$ one gets $\mu_{\sigma}(E_{2,1})\leq \frac{p_3}{q_3}$ and $\mu_{\sigma}(Q_n)\geq \frac{p_1}{q_1}$.

    As the triple $$(\frac{p_1}{q_1},\frac{p_2}{q_2},\frac{p_3}{q_3})$$ is a Farey triangle, we know that the parallelogram in the complex plane with four vertices $0, -p_1+\sqrt{-1}\cdot q_1,-p_2+\sqrt{-1}\cdot q_2,-p_3+\sqrt{-1}\cdot q_3$ has no integral points in its interior. Hence, we see that either $n=1$, i.e., the object $\mathrm{Cone}(f)[-1]$ is semistable, hence stable by the primitivity of $(p_2,q_2)$, or $n=2$ and $Z(E_1')=-p_1+\sqrt{-1}\cdot q_1$, $Z(E_3')=-p_3+\sqrt{-1}\cdot q_3$. In the latter case, we have the following diagram. \[
    \begin{tikzcd}
          0\arrow[r] & E_1 \arrow[r] & E_2\arrow[r] \arrow[d,"="] & E_3\arrow[r] & 0 \\ 0\arrow[r]& E_{2,1} \arrow[r] & E_2\arrow[r] & Q_2\arrow[r] & 0
    \end{tikzcd} \]

    If the composed map $e:E_1\rightarrow Q_2$ in the diagram is nonzero, considering its image, by the fact that $E_1, Q_2$ are semistable objects and $Z(E_1)=Z(Q_2)$ is a primitive integral vector, we find that $\mathrm{im}(e)$ is isomorphic to $E_1$ and $Z(\coker(e))=0$. If $\coker(e)$ is the zero object, then the morphism $e$ is an isomorphism, and we get a splitting of the short exact sequence, which contradicts the assumption $f\neq 0$. Hence $\coker(e)\neq 0$.

    By the octahedral axiom of triangulated category, we have the following commutative diagram. \[\begin{tikzcd}
        & E_{2,1} \arrow[d]\arrow[r,"\simeq"] & E_{2,1}\arrow[d] \\ E_1 \arrow[d,"="]\arrow[r]& E_2\arrow[r]\arrow[d] & E_3 \arrow[d]\\ E_1\arrow[r] &Q_2\arrow[r] & \coker(e)
    \end{tikzcd}\]
with each row and column being a short exact sequence in $\mathcal{A}$. Hence, we get the second case of (3) by setting $E_3'=E_{2,1}, E_1'=Q_2$.

    We get the same result if the composed map $g:E_{2,1}\rightarrow E_3$ is nonzero. Hence, we can assume that the composed maps $f:E_1\rightarrow Q_2$ and $g:E_{2,1}\rightarrow E_3$ are $0$. Hence we have the following commutative diagram \[\begin{tikzcd}
          0\arrow[r] & E_1 \arrow[r] & E_2\arrow[r]  & E_3\arrow[r] & 0 \\ 0\arrow[r]& E_{2,1} \arrow[u]\arrow[r] & E_2\arrow[r]\arrow[u,"="] & Q_2\arrow[r] \arrow[u] & 0
    \end{tikzcd}\]  which is impossible as $E_{2,1}$ is a semistable object with a slope greater than the slope of the stable object $E_1$.
    \end{proof}
\begin{rem}\label{Rmk:stable case}
   This proposition is a slight generalization of \cite[Theorem 4.3]{Continuumenvelops}. In fact, if $\sigma$ is a stability condition, the results can be strengthened in the following way. In (1), $\coker(f)$ is a stable object. In (2), we see that $\coker(f)=0$. In (3), only the first possibility could happen.
\end{rem}

\begin{cor}\label{cor:farey triangle}
    Let $\sigma=(\mathcal{A},Z)$ be a weak arithmetic stability condition on a triangulated category $\mathcal{D}$, and a triple  $$(\frac{p_1}{q_1},\frac{p_2}{q_2},\frac{p_3}{q_3})$$ be any Farey triangle, let $E_{m,n}$ be a semistable object with $$Z(E_{m,n})=-(mp_1+np_3)+\sqrt{-1}\cdot (mq_1+nq_3), $$ where $m,n$ are two nonnegative integers. The following statements hold.
    \begin{enumerate}
        \item Any nontrivial morphism $f:E_{1,0}\rightarrow E_{1,n}$ is an injection for any $n\geq 1$, and $\coker(f)$ is quasi-semistable with central charge $-np_3+\sqrt{-1}\cdot nq_3$.
        \item For any $m\geq 1$, and any nontrivial morphism $f:E_{m,1}\rightarrow E_{0,1}$, we see that $\ker(f)$ is semistable with central charge $-mp_1+\sqrt{-1}\cdot mq_1$, and $Z(\coker(f))=0$.
        \item For any nontrivial morphism $f:E_{0,n}\rightarrow E_{m,0}[1]$, with positive integers $m,n$ satisfying $(m-1)(n-1)=0$, we have that either $\mathrm{Cone}(f)$ is semistable,  or  the following  short exact sequences in $\mathcal{A}$ is the Harder-Narasimhan filtration of $\mathrm{Cone}(f)[-1]$,

            $$0\rightarrow E_{m_1,n_1}'\rightarrow \mathrm{Cone}(f)[-1]\rightarrow E_{m_3,n_3}'\rightarrow 0,$$ where $m_1,n_1,m_3,n_3$ are nonnegative integers with $m_1+m_3=m,n_1+n_3=n$, and $E_{m_1,n_1}', E_{m_3,n_3}'$ are semistable objects with central charge equal to $$Z(E_{m_1,n_1}), \ Z(E_{m_3,n_3})$$respectively.
    \end{enumerate}
\end{cor}
\begin{proof}
    The proof is similar to the previous proposition, one can consult \cite[Proposition 4.10]{Continuumenvelops} for the details.
\end{proof}
\begin{rem}
    Note that this result can be used to prove the fundamental lemma in p-adic Hodge theory (see \cite{Modularvariants}).

    %In fact, if we delete the condition $(m-1)(n-1)=0$, one can show that we have at most $min{n,m}+1$ HN factors of $\mathrm{Cone(f)}[-1]$.
\end{rem}
Let $\theta$ be an irrational number, and let $[a_0;a_1,a_2,\cdots]$ denote the continued fraction representation of $\theta$ \[
\theta =
a_0+\cfrac{1}{a_1+\cfrac{1}{a_2+\cfrac{1}{a_3+\cdots}}}.
\] We call the rational number

    $$\beta_i\coloneqq [a_0; a_1,\cdots, a_i]=\frac{p_i}{q_i}, \ \text{where $p_i, q_i$ are coprime and $q_i>0$},$$
the $i$-th convergent of $\theta$. We also have the notion of semiconvergents, i.e. any rational number in the  form $$\beta_{i,m}=\frac{p_{i,m}}{q_{i,m}}\coloneqq \frac{p_{i}+mp_{i+1}}{q_{i}+mq_{i+1}}$$ where $m$ is an integer such that $0\leq m\leq a_{i+2}$, is called a semiconvergent of $\theta$. We know that $\beta_{i,0}=\beta_i, \ \beta_{i,a_{i+2}}=\beta_{i+2}$, and $p_iq_{i+1}-p_{i+1}q_i=(-1)^{i+1}$ (see \cite[Chapter 1]{LangbookonDiophantineapproximations} for these basic facts). %Also note the following common convention, we usually use $p_{-1}=1$ and $q_{-1}=0$ to denote the starting convergent $\lambda_{-1}$ of any continued fractions. This corresponds to the infinite point $\frac{1}{0}$ in the Poincar\'e disc.

\begin{defn}\label{definition: colimit and limit objects}
    Let $\sigma=(\mathcal{A},Z)$ be a weak arithmetic stability condition on a triangulated category $\mathcal{D}$, and $\theta=[a_0;a_1,a_2\cdots ]$, $\lambda_i=\frac{p_i}{q_i}$, $\lambda_{i,m}=\frac{p_{i,m}}{q_{i,m}}$ be as above. Let $E_{i}$ be a stable object with $Z(E_{i})=-p_{i}+\sqrt{-1}\cdot q_{i}$, and $f_{2i}:E_{2i}\rightarrow E_{2i+2}$ be a nontrivial morphism, we use $E(\theta^-,\{f_{2i}\})$ to denote the colimit of the following sequence \begin{equation}
         E_0\xrightarrow{f_0} E_2\xrightarrow{f_2} E_4\rightarrow \cdots \end{equation} in $\Ind(\mathcal{A})$, the Ind-completion of $\mathcal{A}$.

    Similarly, letting $g_{2i+1}:E_{2i+1}\rightarrow E_{2i-1}$ be a nonzero morphism, we use $E(\theta^+,\{g_{2i+1}])$ to denote the limit of the following sequence \begin{equation}
        \cdots \rightarrow E_5\xrightarrow{g_5}E_3\xrightarrow{g_3} E_1 \end{equation}in $\Pro(\mathcal{A})$, the Pro-completion of $\mathcal{A}$. 
\end{defn}

%\begin{rem}\label{rem:stable case 2}
    %If the pair $\sigma=(\mathcal{A},Z)$ is an arithmetic stability condition, the object $E_{2i}$ is stable for any $i\in\N$.
%\end{rem}
%\begin{rem}
    %Note that in some cases the Ind-completion $\Ind(\mathcal{A})$ is also pro-complete, under such circumstances, one can also consider the inverse limit object $''\varprojlim''E_{2i+1}$ of the sequence (2) in $\Ind(\mathcal{A})$. Sometimes, we just write it as $E(\theta^+,\{g_{2i+1}])$ by abuse of notation.

%\end{rem}

We shall show some examples of arithmetic stability conditions. 

\begin{prop}\label{prop:arithmetic stability condition on curves}
    Let $X$ be any complex smooth projective curve with its genus $g(X)>0$, the stability condition $(\Coh X,-\deg+\sqrt{-1}\cdot \text{rk})$ is an arithmetic stability condition.
\end{prop}

\begin{proof}
    Given a complex projective curve $X$ with genus $g(X)\geq 1$, we need to show that the pair $\sigma=(\Coh \mathrm{X},-\deg+\sqrt{-1}\cdot \text{rk})$ satisfies three conditions in Definition \ref{Defn: arithmetic stability condition}. The first condition holds trivially for $\sigma$. 

 For condition (2), by \cite{Atiyahclassification} and \cite{Stableandunitaryvectorbundles}, we know that the moduli space of stable vector bundles with rank $q>0$ degree $p$ is nonempty for any primitive integral vector $(p,q)\in\mathbb{H}$. For the vector $(1,0)$, one can just take the skyscraper sheaves. Hence $\sigma$ satisfies condition (2).

    For condition (3), if $g(X)=1$, one can take $F$ in condition (3) to be any stable coherent sheaf with $Z(F)=-p_2+\sqrt{-1}q_2$. Then condition (3) follows directly from a calculation of Riemann--Roch (see also \cite{Differentheartsonellipticcurves}). 
    
    If $g(X)\geq 2$, for any two primitive integral vectors $$(p_1,q_1), (p_2,q_2)\in\mathbb{H}\cup\mathbb{R}_{<0}$$ which are connected by a Farey geodesic and $\frac{p_1}{q_1}<\frac{p_2}{q_2}$, and a given stable object $E$ with $Z(E)=-p_1+\sqrt{-1}\cdot q_1$. One can assume that $q_2\geq q_1$ (the other case $q_2<q_1$ can be dealt with similarly), and let $(p_3,q_3)$ be the third vertex in a Farey triangle such that $$p_3+p_1=p_2, \ q_3+q_1=q_2.$$ By the proof of condition (2), we know that there exists a stable object $G$ such that $Z(G)=-p_3+\sqrt{-1}\cdot q_3$. Next, we show that for any such $G$, the space $\Hom(G,E[1])$ is nonzero. Assume this for now, one can take a nonzero morphism $f$ in $\Hom(G,E[1])$, the object $\mathrm{Cone}(f)[-1]$ is stable (see Remark \ref{Rmk:stable case}) and satisfies condition (3).

Indeed, the dimension $h^{1}(G,E)$ of $\Hom(G,E[1])$ can be computed in the following way. 
    $$0-h^1(G,E)=h^0(G,E)-h^1(G,E)=-1+q_1q_3(1-g(X)),$$ here the first equality holds because $E,G$ are stable objects and $\frac{p_3}{q_3}>\frac{p_1}{q_1}$, the second equality holds by the Riemann--Roch theorem and the fact $\frac{p_1}{q_1}-\frac{p_3}{q_3}=\frac{-1}{q_1q_3}$. As $g(X)\geq 2$, this implies that $h^1(G,E)$ is strictly positive. Hence, the proof is complete.
\end{proof}

    \begin{cor}
Let $X$ be a compact Riemann surface of genus $g(X)>0$,
and let $\theta$ be an irrational number.
Then there exist sequences of holomorphic vector bundles on $X$,
\begin{equation}
\mathcal{E}_0 \xrightarrow{f_0} \mathcal{E}_1
\xrightarrow{f_1} \mathcal{E}_2
\longrightarrow \cdots
\longrightarrow \mathcal{E}_k
\xrightarrow{f_k} \mathcal{E}_{k+1}
\longrightarrow \cdots,
\label{eq:seqofhol intro}
\end{equation}
in which each $\mathcal{E}_i$ is slope-stable with slope equal
to the $(2i)$-th continued fraction convergent of $\theta$,
and each $f_i$ is an injective holomorphic bundle map.
For each of these sequences, the colimit
\[
\mathcal{E}_\infty := \varinjlim_i \mathcal{E}_i
\]
is a simple vector bundle of infinite rank.
\end{cor}

\begin{proof}
By Proposition~\ref{prop:arithmetic stability condition on curves},
the pair
\[
\bigl(\Coh X, -\deg+\sqrt{-1}\cdot\operatorname{rk}\bigr)
\]
is an arithmetic stability condition.
 Hence, Corollary \ref{cor:farey triangle} ensures the existence of such sequences. Moreover, the argument in the proof of
\cite[Corollary 5.5]{Continuumenvelops}
then shows that the colimit of each of these sequences
is a simple vector bundle of infinite rank.
\end{proof}

\section{Estimates for second fundamental forms}\label{Section:Estimates for second fundamental forms}

In the rest of this paper, $X$ denotes a closed Riemann surface with positive genus unless otherwise is specified. We normalize the K\"ahler metric $\omega$ on $X$ with $\text{vol}(X,\omega)=1$. 

In this section, we consider the second fundamental forms of subbundles of a Hermitian--Einstein bundle. 

Let $H$ be a Hermitian metric on a holomorphic vector bundle $\E$.  Let $\nabla=\pdv_{H}+\dbar$
be the Chern connection of $(\E,H)$. Then under a local holomorphic basis of $\E$, $H$ is represented by a Hermitian matrix and we have the following expressions of Chern connection and curvature:
 \[\pdv_H=\pdv+(\pdv H)H^{-1},\ F_H=\dbar(\pdv HH^{-1}).\] Suppose that 
\begin{align}
    0\to \S\to \E\to \Q\to0\label{eq:-3}
\end{align}
is a short exact sequence of holomorphic bundle morphisms.  Let $\pi$ be the orthogonal
projection to $\S$ with respect to $H$. Then, $\pi$ provides a smooth splitting  of (\ref{eq:-3}). Under local frames of $\S$ and $\Q$, we may write the connection using this splitting as
\begin{align}
    \nabla=\left[\begin{array}{cc}
\nabla_{\S} & -\beta^{\dagger}\\
\beta & \nabla_{\Q}
\end{array}\right].
\end{align}
Here $\nabla_S$ and $\nabla_Q$ are Chern connections of the induced metrics on $S$ and $Q$, respectively.  $\beta=(1-\pi)\pdv_{H}\pi$ is the second fundamental form of $S$,
and $\beta^{\dagger}=\pi\dbar(1-\pi)$ is the adjoint of $\beta$ with respect to $H$. 
The curvature of $(\E,H)$ can be written as 
\begin{equation}
\label{eq:curvatureSES} F_{H}=\left[\begin{array}{cc}
F_{\S}-\beta^{\dagger}\wedge\beta & -\pdv_{H}\beta^{\dagger}\\
\dbar\beta & F_{\Q}-\beta\wedge\beta^{\dagger}
\end{array}\right].
\end{equation}
If $H$ is a Hermitian--Einstein metric on $\E$. Then from (\ref{eq:curvatureSES}),
\begin{align}
2\pi\mu(\E)\text{Id}_{\E} 
 & =\im\left[\begin{array}{cc}
\Lambda F_{\S}-\Lambda\beta^{\dagger}\wedge\beta & -\Lambda\pdv_{H}\beta^{\dagger}\\
\Lambda\dbar\beta & \Lambda F_{\Q}-\Lambda\beta\wedge\beta^{\dagger}
\end{array}\right].\label{eq:-2}
\end{align}
Let $\xi\in A^{0,1}(\text{End}(E))$ and $\xi^\dagger$ the adjoint of $\xi$ with respect to $H$. Define the pointwise norm of $\xi$ by 
\[
|\xi|_{H}^{2}:=-\text{tr}_{E}\sqrt{-1}\Lambda\left(\xi\wedge\xi^{\dagger}\right).
\]
By Chern-Weil theory, one have the following lemma.
\begin{lem}
\label{lem:If--is}If $H$ is Hermitian--Einstein, we have
\[
\int_{X}|\beta^\dagger|^2_H\omega=\int_{X}|\dbar\pi|_{H}^{2}=2\pi(\mu(\E)-\mu(\S))\text{rk}\S.
\]
\end{lem}

The main goal of this section is to derive estimates of $\beta$ for
Hermitian--Einstein metrics.
\begin{thm}
\label{thm:Assume-the-same} Let $R$ be the scalar curvature of $X$ with respect to $\omega$.  If $(\E,H)$ is Hermitian--Einstein.
then 
\begin{equation}
\sup_{X}|\beta|_{H}^{2}\leq C\|\beta\|_{L^{2}}^{2}=2\pi C(\mu(\E)-\mu(\S))\text{rk}\S,\label{eq:-14-1}
\end{equation}
where $C=C(X,\sup_{X}R^-)$ and $R^-=\max\{0,-R\}$.
Moreover, for any $k\in\N$, there exists a constant $C=C(k,X,\|R\|_{C^{k-1}(X)},\mu)$
such that
\begin{equation}
\|\beta\|_{C^{k}}\leq C\|\beta\|_{L^{2}}.\label{eq:-32}
\end{equation}
\end{thm}

Theorem \ref{thm:Assume-the-same} is similar to the  $L^\infty$ estimates for $L^2$ integrable holomorphic functions via Bergman kernel.

%\begin{example}\label{example:elliptic curve}    If $X$  is an elliptic curve with flat metric, then for a Hermitian--Einstein metric of any stable bundle, the second fundamental forms associated to any proper subbundles are parallel and have constant norm (see Lemma \ref{lem:Suppose-.-Then-1}). Thus, $L(X,\omega)=1$ in this case. \end{example}

%We have x 
%
The rest of the section is devoted to the proof of Theorem \ref{thm:Assume-the-same}. By \cite[Lemma 4.3]{Hermitianvectorbundlesandequidistribution}, the following lemma holds.
\begin{lem}
\label{lem:If--is-1}If $H$ is a \HE{}  metric, then 
\begin{equation}
\dbar\beta^{\dagger}=\dbar_{H}^{*}\beta^{\dagger}\equiv0,\label{eq:-1}
\end{equation}
where $\dbar_{H}^{*}$ is the formal adjoint of $\dbar$. 
\end{lem}

We then use the Bochner-Lichnerowicz-Weitzenb\"ock formula to prove the following lemma.
\begin{lem}
\label{lem:Suppose-.-Then-1}
Suppose $H$ is \HE, then 
\begin{equation}
-\Delta|\beta|_{H}\leq\frac{1}{2}\sup_{X}(R^-)|\beta|_{H}\label{eq:-14}
\end{equation}
where $R$ is the scalar curvature on $X$ and $R^-=\max\{0,-R\}$. 
\end{lem}

\begin{proof}
Suppose $\dbar\xi=0,\dbar_{H}^{*}\xi=0$. Denote $\xi=\xi_{\bar{k}}d\bar{z}^{k}$. Using Bochner-Lichnerowicz-Weitzenb\"ock
formula, we have 
\begin{align}
0=2(\dbar+\dbar_{H}^{*})^{2}\xi & =\nabla_{H}^{*}\nabla_{H}\xi+\mathcal{R}(\xi),\label{eq:-8}
\end{align}
where 
\begin{equation}
\mathcal{R}(\xi)=\text{Ric}(\xi)-\sqrt{-1}\left(\left(\xi_{\bar{k}}\Lambda F-\Lambda F\xi_{\bar{k}}\right)d\bar{z}^{k}-\left(\xi_{\bar{j}}F_{i\bar{k}}-F_{i\bar{k}}\xi_{\bar{j}}\right)g^{\bar{j}i}d\bar{z}^{k}\right).\label{eq:-16}
\end{equation}

Now, we have 
\begin{align*}
-\frac{1}{2}\Delta|\xi|_{H}^{2} & =-|\nabla\xi|_{H}^{2}+\<\nabla_{H}^{*}\nabla_{H}\xi,\xi\>\\
 & =-|\nabla\xi|_{H}^{2}-\<\mathcal{R}(\xi),\xi \>_{H}.
\end{align*}
Let $\lambda(x)$  be the smallest eigenvalue of $\mathcal{R}$ at $x$.   Using Kato's inequality we have 
\begin{align*}
-|\xi|_{H}\Delta|\xi|_{H} & =-\frac{1}{2}\Delta|\xi|_{H}^{2}+|\nabla|\xi|_{H}|^{2}\\
 & \leq-|\nabla\xi|_{H}^{2}-\lambda|\xi|_{H}^{2}+|\nabla|\xi|_{H}|^{2}\\
 & \leq \sup_X\{-\lambda\}|\xi|_{H}^{2}.
\end{align*}

If $H$ is Hermitian--Einstein, $F_{H}=\mu\omega\text{Id}$
and the only non-vanishing term in $\mathcal{R}(\beta)$ is $\text{Ric}(\beta)=\frac{1}{2}R\beta$.
Thus, we have (\ref{eq:-14}).
\end{proof}
\begin{rem}
    For higher dimensional K\"ahler manifolds, (\ref{eq:-16}) involves $Ric$ and full curvature tensors $F_H$. Thus, we need to use the a constant depending on $|F_H|$ in place of $\frac{1}{2}\sup_X(R^-)$ in (\ref{eq:-14}). The problem is that we cannot  obtain \emph{a prior} estimates on $|F_H|$  in terms of the slope and the geometric information on $X$  for Hermitian--Einstein metrics. Nevertheless, it is still possible to bound the second fundamental form by restricting to a class of  vector bundles whose total curvature is naturally bounded by a uniform constant.
\end{rem}
The proof of the following lemma uses  Nash-Moser iteration. See \cite[Theorem 8.15]{GilbargTrudinger}. 

\begin{lem}
\label{lem:Suppose-that-}Suppose $u\geq0$ is a $C^{2}$ function
on $X$ that satisfies
\(
-\Delta u\leq au
\)
for some constant $a\geq0$. Then
\(
\sup_{X}u\leq C(a,X)\|u\|_{L^{2}(X)}.
\)
\end{lem}

By Lemma \ref{lem:Suppose-that-} and (\ref{eq:-14}), we have the
following estimate of $|\beta|_{L^{\infty}}$. 
\begin{prop}
\label{prop:Suppose--and}Suppose  $(\E,H)$ is Hermitian--Einstein.
Then 
\begin{equation}
\sup_{X}|\beta|_{H}^{2}\leq C\|\beta\|^2_{L^{2}},\label{eq:-14-1-1}
\end{equation}
where $C=C(X,\sup_{X}R^-)$.
\end{prop}

Proposition \ref{prop:Suppose--and} completes the uniform estimate
of $\beta$. Now, higher order estimates can be derived similarly. 
\begin{proof}
[Proof of Theorem \ref{thm:Assume-the-same}]We show for $k=1$ since the case $k\geq2$
can be done similarly using induction. In the following, we denote $\nabla$ the induced connection
on $A^{0,1}(\text{End}(\E))$ with respect to $H$. By (\ref{eq:-8}),
we have 
\begin{equation}
\|\nabla\beta\|_{L^{2}}\leq\sup_{X}|R|\cdot\|\beta\|_{L^{2}}.\label{eq:-20}
\end{equation}
By Bochner-Lichnerowicz-Weitzenbock formula,
\begin{align*}
\nabla^{*}\nabla\nabla\beta & =\nabla\nabla^{*}\nabla\beta+\mathcal{R}_{1}(\nabla\beta)\\
 & =\nabla(2(\dbar+\dbar_{H}^{*})^{2}\beta-\mathcal{R}(\beta))+\mathcal{R}_{1}(\nabla\beta)\\
 & =\mathcal{R}_{2}(\nabla\beta)+\mathcal{R}_{3}(\beta),
\end{align*}
where $\mathcal{R}_{1},\mathcal{R}_{2},\mathcal{R}_{3}$ are linear
forms and only depend on local curvature tensors $F_{H}$, $R$ and
their covariant derivatives. Since $\sqrt{-1}F_{H}=2\pi\mu\omega\text{Id}$, $\sup|\mathcal{R}_{2}|$
and $\sup|\mathcal{R}_{3}|$ are bounded by $C(\|R\|_{C^{1}(X)},|\mu|)$.
Then, we argue similarly as in Lemma \ref{lem:Suppose-.-Then-1} to
show that 
\[
-\Delta(|\nabla\beta|^{2}+|\beta|^{2})^{\frac{1}{2}}\leq C_{1}\left(|\nabla\beta|^{2}+|\beta|^{2}\right)^{\frac{1}{2}},
\]
for some $C_{1}=C(\|R\|_{C^{1}},|\mu|)$. Then, by Lemma \ref{lem:Suppose-that-}
and (\ref{eq:-20}),
\[
\sup_{X}(|\nabla\beta|+|\beta|)\leq C_{2}\left(\|\nabla\beta\|_{L^{2}}+\|\beta\|_{L^{2}}\right)\leq C_{3}
\]
for some $C_{2},C_{3}$ depends on $\|R\|_{C^{1}},|\mu|$. We have
finished the proof.
\end{proof}

For later use, we record some corollaries. 
\begin{cor}\label{cor: an easy lemma}
  Notations as above. For any $\epsilon>0$, there exists some $\delta(\epsilon)>0$ such that \[\|\beta\|^2_{L^2}\leq \|\beta\|^2_{L^{2+2\delta}}\leq (1+\epsilon)\|\beta\|^2_{L^2}.\]
\end{cor}
\begin{proof}
 Let $f(x)=|\beta(x)|^2_H$. If $f\equiv 0$ then it is trivial. Otherwise,  we may normalize such that $\|f\|_{L^1}=1$. Since $\text{vol}(X)=1$, the first inequality is due to H\"{o}lder's inequality. For the second inequality, because $\|\beta\|^2_{L^\infty}\leq C\|\beta\|^2_{L^2}$ for some constant $C\geq 1$ (Proposition \ref{prop:Suppose--and}), we have $|f|^{1+\delta}\leq C^\delta|f|$ for any $\delta>0$. For any $\epsilon>0$, we may find $\delta$ sufficiently small such that $C^\frac{\delta}{1+\delta}\leq (1+\epsilon).$ Then, we have \[\|f\|_{L^{1+\delta}}\leq C^{\frac{\delta}{1+\delta}}\|f\|_{L^1}^{\frac{1}{1+\delta}}\leq (1+\epsilon).\] We have finished the proof.
 \end{proof}
 
Let $H|_{\S}$ be the restriction of a Hermitian--Einstein metric of $\E$ on $\S$. Let
$H_{0}$ be the Hermitian--Einstein metric on $\S$. As an application of Theorem \ref{thm:Assume-the-same},
we show that after a conformal change $e^{\phi}H|_{\S}$ has constant
determinant with respect to $H_{0}$. Moreover, $\phi$ is bounded in 
$C^{r}$ with estimates.

\begin{prop}
\label{prop:There-exists-a}There exists a smooth function $\phi$
such that $\det(e^{\phi}H|_{\S}H_{0}^{-1})\equiv 1,$ and
\begin{equation}
\|\phi\|_{C^{k}}\leq C(\mu(\E)-\mu(\S)),\quad\int_{X}\phi=0,\label{eq:-18}
\end{equation}
for $C=C(k,X,\|R\|_{C^{k-1}(X)},\mu(\E))$.
\end{prop}

\begin{proof}
Let $s$ be any $H_{0}$-Hermitian endomorphism.
If 
\begin{equation}
\text{tr}\sqrt{-1}\Lambda(F_{H_{0}}-F_{e^{s}H_{0}})=0,\label{eq:-26}
\end{equation}
then there exists a constant $c$ such that 
$\text{tr}\left(s-\frac{c}{\text{rk}S}\text{Id}\right)=0.$
So we only need to show the existence
of $\phi$ satisfying (\ref{eq:-18}), and 
\begin{equation}
\text{tr}\left(\sqrt{-1}\Lambda F_{e^{\phi}H|_{\S}}-\mu(\S)\text{Id}\right)=0.\label{eq:-17}
\end{equation}
Since $F_{e^{\phi}H|_{\S}}=F_{H|_{\S}}+\dbar\pdv\phi\text{Id}\label{eq:-4},$  (\ref{eq:-17}) is equivalent to solve
\begin{align}\label{eq:-37}
\Delta\phi  
  =\left(\mu(\E)-\mu(\S)\right)+\frac{1}{\text{rk}\S}\text{tr}\sqrt{-1}\Lambda\beta^{\dagger}\wedge\beta.
\end{align}
It is always solvable if we choose the normalization condition $\int_X\phi=0$ to ensure that $\det(e^\phi H|_{\S}H_0 ^{-1})=1$. Since $\text{tr}\sqrt{-1}\Lambda\beta^{\dagger}\wedge\beta=|\beta|_{H}^{2}$
which is uniformly bounded in $C^{k}$ by $C(k,X)(\mu(\E)-\mu(\S))\text{rk}\S$.
We see that $\|\phi\|_{C^{k}}$ is also bounded by $C(k,X)(\mu(\E)-\mu(\S))$.
\end{proof}

\section{Donaldson's functional and estimates for restricted metrics \label{section:uniform}}

Assume that we have a sequence of stable holomorphic vector bundles on $X$
\begin{equation}
\mathcal{S}=\mathcal{E}_{0}\overset{\iota_{0}}{\to}\mathcal{E}_{1}\overset{\iota_{1}}{\to}\mathcal{E}_{2}\overset{\iota_{2}}{\to}\mathcal{E}_{3}\to\cdots\label{eq:-28}
\end{equation}
Each $\iota_i$ is assumed to be an injective holomorphic bundle map.
Denote $S$, $E_{i}$ the smooth $\C$-vector bundles of $\mathcal{S}$
and $\mathcal{E}_{i}$. By Narasimhan-Seshadri-Donaldson-Uhlenbeck-Yau theorem, there
exits a Hermitian Einstein metric $\tilde{H_{i}}$ on each $\mathcal{E}_{i}$.
Let $H_{i}$ be the restriction of $\tilde{H}_{i}$ on $\mathcal{S}$. 

The goal of this section is to obtain estimates for $H_i$. Since Hermitian--Einstein metric is unique up to a constant, we first make a normalization. We  fix a point $p\in X$ and normalize $H_{i}$ so that 
\begin{equation}
\det(H_{i}H_{0}^{-1})|_{p}=1.\label{eq:-27}
\end{equation}
Let $\pi_{i}$ be the orthogonal projection from $E_{i}$ to $S$
with respect to $\tilde{H}_{i}$. Let $\beta_{i}=(1-\pi_i)\pdv_{\tilde{H}_{i}}\pi_{i}$
be the second fundamental form.

As defined in the introduction, we set
\[
\theta=\mu(\mathcal{E_\infty})=\lim_{i\to\infty}\mu(E_{i}).
\]

The first theorem of this section provides $L^\infty$-bound for $H_i$.
\begin{thm}
\label{thm:Suppose-that-} Let $(\mathcal{E}_i,\tilde{H}_i)$ be a sequence of Hermitian--Einstein bundles as in (\ref{eq:-28}).   Suppose
that $\mathcal{S}$ is a good approximation of $(\mathcal{E}_{\infty},\theta)$.
Let $H_{0}$ be the Hermitian Einstein metric on $\S$ and $H_i=\tilde{H}_i|_S$, and $h_{i}=$$H_{i}H_{0}^{-1}$
be the sequence of Hermitian endomorphisms with $\det(h_i)|_p=1$ for some fixed point $p\in X$. Then $\|\log h_{i}\|_{\rho, H_0,L^{\infty}}$
is uniformly bounded by some constant $C$. 
\end{thm}

We explain the meaning of the $L^\infty$-norm. Let $H$ be a Hermitian metric on $S$. Let $s$ be an endomorphism of $S$. Let $|s(x)|_{\rho,H}$
be the spectral radius/operator norm of $s(x)$ with respect to $H$ at $x$. We define its operator $L^{q}$-norm for $q\in [1,\infty]$ as
\begin{equation}
\|s\|_{\rho,H,L^{q}}:=\frac{1}{\text{vol}(X,\omega)}\||s|_{\rho,H}\|_{L^q}.\label{eq:-47}
\end{equation}
We will also consider the fiberwise matrix $2$-norm $|\cdot|_{2,H}$  where $|s(x)|^2_{2,H}=\text{tr}(s(x)s^\dagger(x))$ and $s^\dagger=HsH^{-1}$.  Globally, we may define
\begin{equation}
\|s\|_{2,H,L^{q}}:=\frac{1}{\text{vol}(X,\omega)}\||s|_{2,H}\|_{L^q}.\label{eq:-47-1}
\end{equation}
$C^{k,\alpha}$ and Sobolev norms with respect to $|\cdot|_{\rho,H}$  and $|\cdot|_{2,H}$ can also be defined. See Appendix  \ref{subsec:norms}.

\begin{rem}In finite ranks, all fiberwise norms are essentially equivalent up to a constant depending on the rank. However, if the rank is large enough, using matrix 2-norm may introduce a big constant (especially in Lemma \ref{lem:Notations-as-above.}) that will eventually invalidate the proof. So we will be stringent about the choice of the fibrewise norm.
%    We also remark that since we derive Theorem \ref{thm:Suppose-that-} using the proof by contradiction, we will not be able to obtain effective estimates for the $L^\infty$ bounds.  Such bounds generally depend on $X,\omega,\mathcal{S}$, and the sequence $\mathcal{E}_i$.  
\end{rem}
Once we obtained the $L^\infty$ bound for all $h_i$,  a standard blow-up argument can be applied to obtain gradient estimates. Then the standard elliptic theory provides all higher order estimates. By Arzela-Ascoli theorem, we can obtain convergent subsequences of $h_i$ in any $C^k$. 
\begin{thm}
\label{thm:For-any-,}Assume the same conditions as in Theorem \ref{thm:Suppose-that-}.
For any $k\in\N$, there exists a subsequence of $H_{i}$ which converges
in $C^{k}$ to a Hermitian metric. 
\end{thm}

A key ingredient of the $L^\infty$-estimates relies on the understanding
of Donaldson's functional \cite{Donaldson85,Donaldsoninfinite} for a sequence of Hermitian metrics. Let
us first recall the definition of Donaldson's functional. 
\begin{defn}
Let $K,H$ be two Hermitian metrics on $\S$. Let $H(t)$ be a smooth
one parameter family of Hermitian metrics on $E$ such that $H(0)=H,H(1)=K$.
Let 
\[
M(H,K):=\int_{0}^{1}\int_{X}\text{tr}\left[\left(\sqrt{-1}\Lambda F_{H(t)}-\mu(\S)\text{Id}\right)\dot{H}H^{-1}\right]\omega dt.
\]
We call $M(H,K)$ the Donaldson's functional for the pair $(H,K)$. 
\end{defn}

We gather some properties for $M(K,H)$ in the following lemma (see \cite{Donaldson85}).
\begin{lem}
\label{lem:Let--be}Let $K,H$ be two Hermitian metrics on $\S$. Then,
\begin{enumerate}
\item $M(K,K)=0$.
\item If $H'$ is another Hermitian metric, then 
\[
M(K,H)+M(H,H')=M(K,H').
\]
\item \label{enu:Let--be}Let $\lambda>0$ be a constant. Then
\[
M(K,\lambda K)=0,\ M(K,\lambda H)=M(K,H).
\]
\item Let $\mathrm{End}_{K}(S)$ be the Hermitian endomorphism with respect
to $K$. Let $s\in\mathrm{End}_{K}(S)$. Then, $t\mapsto M(K,e^{ts}K)$
is a convex function. 
\item $H$ is a critical point of $M(K,\cdot)$ iff $H$ is a Hermitian--Einstein
metric. 
\end{enumerate}
\end{lem}

Let $A$ be a Hermitian metric on $\mathbb{C}^n$. Let $e_{i}$ be a unitary basis for
$A$ so that $A=\sum_{i=1}^{n}\lambda_{i}e_{i}e_{i}^{*}$. Let 
\[
\phi(\lambda_{i},\lambda_{j}):=\begin{cases}
\frac{e^{(\lambda_{i}-\lambda_{j})}-(\lambda_{i}-\lambda_{j})-1}{(\lambda_{i}-\lambda_{j})^{2}}, & \lambda_{i}\not=\lambda_{j},\\
\frac{1}{2}, & \lambda_{i}=\lambda_{j}.
\end{cases}
\]
  We have an operator $\Phi_{A}$ sending a matrix $B=\sum_{i,j}B_{ij}e_{i}e_{j}^{*}$ to the following matrix.
\[
\Phi_{A}[B]_{ij}=\begin{cases}
\frac{e^{(\lambda_{i}-\lambda_{j})}-(\lambda_{i}-\lambda_{j})-1}{(\lambda_{i}-\lambda_{j})^{2}}B_{ij}, & \lambda_{i}\not=\lambda_{j},\\
\frac{1}{2}, & \lambda_{i}=\lambda_{j}.
\end{cases}
\]
Let $s=\log(HK^{-1})$. We can write $M(K,e^{s}K)$ in the following
explicit form (see\cite{simpson88constructVarofhodge}, pp. 882). 
\begin{lem}
\label{lem:Here-if-.}
\[
M(K,e^{s}K)=\int_{X}\<{} \Phi_{s}[\dbar s],\dbar s\>_{K}\omega +\int_{X}\text{tr}(\sqrt{-1}\Lambda F_{K}s)\omega.
\]
Here 
\begin{align*}
\<{}\Phi_{s}[\dbar s],\dbar s\>_{K} & =-\sqrt{-1}\Lambda\text{tr}(\Phi_{s}[\dbar s](\dbar s)^{\dagger}).
\end{align*} 
\end{lem}

A remarkable observation made by Donaldson relates the slope stability of $\mathcal{S}$ and the properness of $M(K,\cdot)$ \cite{Donaldsoninfinite}.

\begin{prop}
Suppose that $\S$ is stable. Let $B>0$ be a constant. Let
\begin{equation}
\mathcal{T}_{B}:=\{s\in\mathrm{End}_{H}(S):\text{tr}s=0,|\Lambda F_{e^{s}H}|_{\rho,H}<B\}.\label{eq:-68}
\end{equation}
Then there exist constants $C_{1},C_{2}$ which depend on $(X,H,\S,B)$
such that for $\forall s\in\mathcal{T}_B$, we have
\begin{equation}
\|s\|_{\rho,L^{\infty}}\leq C_{1}\left(M(H,e^{s}H)+C_{2}\right).\label{eq:-19}
\end{equation}
\end{prop}

We remark that any $s$ in the  set $\mathcal{T}_{B}$  has $L^{\infty}$-norm bounded by constant multiple of its $L^{1}$-norm. We include the proof of this lemma since we use the operator norm and the constant is independent of the rank of the vector bundle. 

\begin{lem}
\label{lem:Suppose-that-.}Suppose that $s\in\mathcal{T}_{B}$. Then,
there exists a constant $C(X,B)$ such that 
\begin{equation}
\|s\|_{\rho,H,L^{\infty}}\leq C(X,B)(\|s\|_{\rho,H,L^{1}}+1).\label{eq:-70}
\end{equation}
\end{lem}

\begin{proof}
Let $h=e^{s}$. For all $p\in \N_{>1}$, let 
\[
u_{p}:=\frac{1}{p}\log\text{tr}h^{p}=\frac{1}{p}\log|h^{\frac{p}{2}}|_{2,H}^{2}.
\]
By picking a unitary basis at $x\in X$, we can diagonalize $h$
with respect to $H$. At $x$, we may assume that $H=\text{Id}$, $h=\text{diag}\{\lambda_{1},\cdots,\lambda_{r}\}$,
and locally $\dbar h=Pd\bar{z}$ with $P$ a $r\times r$ complex
matrix valued function. Let 
\[
Q(\lambda_{i},\lambda_{j})=\begin{cases}
\frac{\lambda_{i}^{p-1}-\lambda_{j}^{p-1}}{\lambda_{i}-\lambda_{j}}, & i\not=j,\\
(p-1)\lambda_{i}^{p-2}, & i=j.
\end{cases}
\]
Then, 
\begin{align*}
\frac{1}{2}\Delta u_{p} & =\frac{\text{tr}\left(\sqrt{-1}(\Lambda F_{H}-\Lambda F_{hH})h^{p}\right)}{\text{tr}h^{p}}+\sum_{i,j}\lambda_{i}^{p-1}\lambda_{j}^{-1}|P_{ij}|^{2}.\\
 & +\sum_{i,j}Q(\lambda_{j},\lambda_{i})|P_{ij}|^{2}-p\frac{|\sum_{i}\lambda_{i}^{p-1}P_{ii}|^{2}}{(\text{tr}h^{p})^{2}}\\
 & \geq-2B+p\sum_{i}\lambda_{i}^{p-2}|P_{ii}|^{2}-p\frac{|\sum_{i}\lambda_{i}^{p-1}P_{ii}|^{2}}{(\text{tr}h^{p})^{2}}\\
 & \geq-2B.
\end{align*}
The last inequality is due to Cauchy-Schwarz inequality. Thus, using Green function representation, we have 
\begin{equation}
\sup_{X}u_{p}\leq C(X,B)(\|u_{p}\|_{L^{1}}+1).\label{eq:-69}
\end{equation}
Notice that $u_{p}$ increases to $\max\log\lambda_{i}$ as $p$
goes to $\infty$.   Hence, $\sup_X\max\log\lambda_i$ is bounded by $L^1$-norm.  If we repeat the same procedure for $H'=hH$ and $h'=h^{-1}$, then we also have $\sup_X\max(-\log \lambda_i)$ is bounded by $L^1$-norm. Therefore,
we have (\ref{eq:-70}). 
\end{proof}

If $s$ is a $H$-Hermitian endomorphism, we call $e^{ts}H$ a \emph{geodesic
ray} starting from $H$. We denote Donaldson's functional of a geodesic ray as
\[
M({ts})=M(H,e^{{ts}}H),
\]
whenever the initial metric $H$ is assumed. If $H$ is the Hermitian--Einstein
metric on $\S$, for any $s\in\text{End}_H(S)$, $M({ts})$ is \emph{convex}
in $t$ and achieves its minimum at ${t}=0$.  See \cite{jonsson2022geodesic} for more discussions on the geometry of geodesic rays and their connections to Donaldson's functional.

In the rest of this section,
we write $H=H_{0}$, which is the Hermitian--Einstein metric on $S$.

Now, we return to the proof of Theorem \ref{thm:Suppose-that-}. In general, $\det(h_{i})$ is not constant. However, by Proposition \ref{prop:There-exists-a},
there exists a sequence of conformal weights $e^{\phi_{i}}$ such that
\[
\det(e^{\phi_{i}}h_{i})=\exp\text{tr}\log\left(e^{\phi_{i}}h_{i}\right)\equiv 1.
\]
Moreover, 
\[
|\Lambda F_{e^{\phi_{i}}H_{i}}|_{\rho,H}\leq|\Lambda F_{H_{i}}|_{\rho,H}+|\Delta\phi|<B.
\]
for some $B>0$ depending only on $X,R,\mu(E),\mu(S)$. Thus, for this choice of $B$, $e^{\phi_{i}}h_{i}\in\mathcal{T}_B.$

 % Since $\|\phi_{i}\|_{C^{0}}$ are bounded by $C(\mu(E_{i})-\mu(S))$
 % which is uniformly bounded by $C(\theta-\mu(\S))$, we can still   control $\|\log h_{i}\|_{\rho,L^{\infty}}$

If $\S$ is a good approximation of $(\E_\infty,\theta))$, we can choose some $L>1$ such that \begin{align}
L(\theta-\mu(\S))\text{rk}\S<\text{rk}\mathcal{S}_{0}(\mu(\S)-\mu(\mathcal{S}_{0}))\label{eq:-30}
\end{align} for any proper coherent subsheaf $\S_0$ with $1\leq \text{rk} \S_0<\text{rk}\S$. In fact, as $\text{deg}(\mathcal{S}_0)$ is an integer for any proper coherent subsheaf $\S_0\subset \S$, for any fixed number $k\leq \text{rk}\mathcal{S}$, we can find a proper coherent subsheaf $\mathcal{S}_k\subset \mathcal{S}$ such that $\text{rk}(\mathcal{S}_k)=k$ and $$\mu(\mathcal{S}_k)=\text{max}\{\mu(\F)|\F \subset \S \text{ and $\text{rk}\F=k$}\}.$$ Hence, we have $$\text{inf}\{\text{rk}\mathcal{S}_{0}(\mu(\S)-\mu(\mathcal{S}_{0}))| \S_0\subset \S \}=\text{min}\{\text{rk}\mathcal{S}_{k}(\mu(\S)-\mu(\mathcal{S}_{k}))|1\leq k<\text{rk}\S\}.$$ Therefore, we can find a number $L>1$ such that (\ref{eq:-30}) holds for any proper coherent subsheaf $\S_0\subset \S$.

Let $\epsilon=L-1$. We can find some $\delta>0$ as in Corollary \ref{cor: an easy lemma} such that for any orthogonal splitting of the Hermitian--Einstein connection with the second fundamental form $\beta$, we have 
\begin{align}
  \label{eq:beta_delta}  \|\beta\|^2_{L^{2+2\delta}}\leq L\|\beta\|^2_{L^2}.
\end{align}
Choose $q=1+\frac{1}{\delta}$ to be the conjugate H\"older index of $1+\delta$.

%\textcolor{blue}{
We argue by contradiction. Suppose that $h_i$ is not uniformly bounded. By Lemma \ref{lem:Suppose-that-.}, the sequence $\|\log h_{i}\|_{\rho,L^{q}}$
is unbounded. Passing to a subsequence, we assume that 
\[
m_{i}=\|\log(e^{\phi_{i}}h_{i})\|_{\rho,L^{q}}\to\infty.
\] Without loss of generality, assume $m_{i}>1$. Let 
\[
s_{i}=\frac{1}{m_{i}}\left(\log h_{i}+\phi_{i}\text{Id}\right).
\]Then $\text{tr}s_i=0$, and $\|s_i\|_{\rho,L^q}=1$. By Lemma \ref{lem:Suppose-that-.}, $\|s_{i}\|_{\rho,L^{\infty}}\leq C(X,B)$.
\begin{lem}
\label{lem:Notations-as-above.}Notations as above. We have
\begin{align}
M(m_{i}s_{i}) & \leq  Lm_{i}2\pi(\mu(\E_{i})-\mu(\S))\text{rk}\S.\label{eq:-23}
\end{align}
%where  $L=L(X,\omega)$  is the convergence threshold  in Definition  \ref{def:approx}. 
Moreover, there exist a subsequence
$\{s_{i_{k}}\}$ and $s_{\infty}\in W^{1,2}\cap L^\infty(X,|\cdot|_\rho)$ such that 
\begin{align*}
s_{i}  \rightharpoonup s_{\infty}\in W^{1,2}\text{ weakly,}\quad
s_{i}  \to s_{\infty}\ \text{in }L^{q}(X,|\cdot|_\rho).
\end{align*}
Here, we use $\|s\|_{W^{1,2}}:=\|\dbar s\|_{2,H,L^{2}}+\|s\|_{2,H,L^{2}}$
to be the Sobolev norm.
\end{lem}

\begin{proof}
Let $H=H_{0}$ be the Hermitian--Einstein metric on $\S$ and $s\in \mathrm{End}_{H}S$. Since the functional $M(ts)$ is convex in $t$ and $M(0)=0$ is the global minimum,
we have 
\begin{align*}
M(m_{i}s_{i}) & \leq\frac{d}{dt}M(tm_{i}s_{i})|_{t=1}\\
 & =\int_{X}\text{tr}\left[\sqrt{-1}\Lambda F_{H_{i}}m_{i}s_{i}\right]\omega\\
 & =\int_{X}\text{tr}\left[\left(\sqrt{-1}\Lambda\beta_{i}^{\dagger}\wedge\beta_{i}\right)m_{i}s_{i}\right]\omega.
\end{align*}
We have used the fact that $\text{tr}(s_{i})=0$.  Since
$\sqrt{-1}\Lambda\beta_{i}^{\dagger}\wedge\beta_{i}\leq0$, we have
\[
|\text{tr}\left[\left(\sqrt{-1}\Lambda\beta_{i}^{\dagger}\wedge\beta_{i}\right)s_{i}\right]|\leq|\beta_{i}|_{2,H_{i}}^{2}\cdot|s_{i}|_{\rho,H}.
\]
Thus, by H\"older's inequality and (\ref{eq:beta_delta}),
\begin{align}
M(m_{i}s_{i})  \leq m_{i}\|s_{i}\|_{\rho,L^{q}}\|\beta_{i}\|_{2,H_{i},L^{2+2\delta}}^{2}\label{eq:-23-1}
  \leq Lm_{i}2\pi(\mu(\E_{i})-\mu(\S))\text{rk}\S.
\end{align}
Notice that $\frac{e^{x}-x-1}{x^{2}}\geq\frac{1}{2(|x|+1)}$, we have
\begin{align*}
\<\Phi_{m_{i}s_{i}}[m_{i}\dbar s_{i}],m_{i}\dbar s_{i}\>_{H}  \geq\frac{m_{i}^{2}|\dbar s_{i}|_{2,H}^{2}}{2(m_{i}|s_{i}|_{\rho,H}+1)}  \geq m_{i}\frac{|\dbar s_{i}|_{2,H}^{2}}{C(X,B)},
\end{align*}
where $C(X,B)$ is the constant in Lemma \ref{lem:Suppose-that-.}.
Thus, by Lemma \ref{lem:Here-if-.}
\begin{align}
\frac{m_{i}}{C}\|\dbar s_{i}\|_{2,L^{2}}^{2} & \leq M(H,e^{m_{i}s_{i}}H).\label{eq:-22} 
\end{align}
From (\ref{eq:-22}), we have
\begin{equation}
\|s_{i}\|_{W^{1,2}}\leq C_{2}(C+(\mu(\E_{i})-\mu(\S))\text{rk}\S).\label{eq:-21}
\end{equation}
Thus, $s_{i}$ is bounded in $W^{1,2}$. After replacing by a subsequence,
we may assume that there exists $s_{\infty}\in W^{1,2}\cap L^{\infty}(X,|\cdot|_{\rho})$
so that
\begin{align*}
s_{i}  \rightharpoonup s_{\infty}\in W^{1,2}\text{ weakly,}\quad
s_{i}  \to s_{\infty}\ \text{in }L^{q}(X,|\cdot|_{\rho}).
\end{align*}
\end{proof}
We claim that $s_{\infty}$ actually has constant eigenvalues almost everywhere.
\begin{lem}
Notations as above. $s_{\infty}$ has constant eigenvalue a.e.
\end{lem}

\begin{proof}
Donaldson's functional is lower semicontinuous under weak convergence
in $W^{1,2}$ (\cite[Proposition 4.2]{jonsson2022geodesic}). Thus, for any $t>0$, we have 
\begin{equation}
M(ts_{\infty})\leq\liminf_{i\to\infty}M(ts_{i}).\label{eq:-24}
\end{equation}
Since $t\mapsto M(ts)$ is increasing and convex in $t\in(0,m_{i}]$ ,
\begin{equation}
M(\frac{t}{m_{i}}m_{i}s_{i})\leq\frac{t}{m_{i}}M(m_{i}s_{i})\leq t2\pi L(\theta-\mu(\S))\text{rk}\S,\label{eq:-25}
\end{equation}
where the second inequality comes from (\ref{eq:-23}). Hence, for $t\in(0,\infty)$,
\begin{equation}
M(ts_{\infty})\leq tL2\pi(\theta-\mu(\S))\text{rk}\S.\label{eq:-38}
\end{equation}

Let $\lambda_{1},\cdots,\lambda_{n}$ be the local eigenvalues of
$s\in\mathrm{End}_{H}(S)$ with respect to $H$. Since $e^{x}-x-1\geq\frac{1}{2}x^{2}$
for $x\geq0$, we have
\begin{align*}
M(ts) & =\int_{X}\sum_{i,j}|(\dbar s)_{j}^{i}|^{2}\frac{e^{t(\lambda_{i}-\lambda_{j})}-t(\lambda_{i}-\lambda_{j})-1}{(\lambda_{i}-\lambda_{j})^{2}}\omega
 \geq\int_{X}\sum_{j\leq i}|(\dbar s)_{j}^{i}|^{2}\frac{t^{2}}{2}\omega.
\end{align*}
Since $M(ts_{\infty})\leq Ct$, we have $(\dbar s_{\infty})_{j}^{i}=0$
a.e. for $\lambda_{j}\leq\lambda_{i}$. Then for any $C^{1}$ real
function $\psi,$ it holds that 
\[
\dbar\text{tr}\psi(s_{\infty})=\text{tr}(d\psi(s_{\infty})[\dbar s_{\infty}])=0,\
\]
almost every where. Thus, we see that $s_{\infty}$ has constant eigenvalue a.e.
\end{proof}
Since eigenvalues of $s_\infty$ are constant a.e., we denote $\lambda_{1}<\lambda_{2}<\cdots<\lambda_{m}$  the distinct
eigenvalues of $s_{\infty}$. Let $\chi_{m}\equiv1$.  For $i\in\{1,\cdots,m-1\}$, denote
\begin{equation}
\chi_{i}(t)=\begin{cases}
1, & t<\frac{\lambda_{i}+\lambda_{i+1}}{2},\\
0, & t>\frac{\lambda_{i}+\lambda_{i+1}}{2},
\end{cases}\qquad \pi_{i}:=\chi_{i}(s_{\infty}).\label{eq:-31}
\end{equation}
Then $\pi_{i}$ is the orthogonal projection to the direct sum of
eigenspaces of $s_{\infty}$ with eigenvalue smaller than $\lambda_{i+1}$.
Let $\mathcal{S}_{i}=\text{Im}\pi_{i}$.

The concept of
weakly holomorphic projection was introduced in Uhlenbeck-Yau \cite{UhlenbeckYau86} and the lemma below was proved in \cite{UhlenbeckYau89}. 
\begin{lem}
\label{lem:-is-a}$\pi_{i}$ is a weakly holomorphic projection and
$\text{Im}\pi_{i}$ defines a reflexive coherent subsheaf of $S$. 
\end{lem}

We now finish the proof Theorem \ref{thm:Suppose-that-}.

\begin{proof}
[Proof of Theorem \ref{thm:Suppose-that-}] Suppose that the sequence $\{\|\log h_{i}\|_{\rho,L^{q}}\}$
is unbounded. From the construction of $\pi_{i}$ and Lemma \ref{lem:-is-a},
we obtain a filtration of $\mathcal{S}$:
\[
0=\mathcal{S}_{0}\subset\mathcal{S}_{1}\subset\cdots\subset\mathcal{S}_{m}=\mathcal{S}.
\]
Let $\mathcal{F}_{i}=\mathcal{S}_{i}/\mathcal{S}_{i-1}$ for $i=1,\cdots,m$.
Then by \cite[Theorem 3.7]{jonsson2022geodesic}, we have

\[
M(ts_{\infty})=2\pi\sum_{i=1}^{m}\lambda_{i}\text{rk}\mathcal{F}_{i}(\mu(\mathcal{F}_{i})-\mu(S))t-\sum_{1\leq i<j\leq m}B_{ij}(1-e^{-t(\lambda_{j}-\lambda_{i})})
\]
where $B_{ij}$ are non-negative integers. Let $a_{i}=-(\deg\mathcal{S}_{i}-\mu(S)\text{rk}\mathcal{S}_{i})$.
Since $\mathcal{S}$ is stable, $a_{i}>0$ for $1\leq i\leq m-1$
and $a_{0}=a_{m}=0$. We have 
\[
\text{rk}\mathcal{F}_{i}(\mu(\mathcal{F}_{i})-\mu(\S))=-(a_{i}-a_{i-1}).
\]
Let $a=\min\{a_{i}\}_{i=1}^{m-1}$. Then
\begin{align*}
\sum_{i=1}^{m}\lambda_{i}\text{rk}\mathcal{F}_{i}(\mu(\mathcal{F}_{i})-\mu(\S)) & =-\sum_{i=1}^{m}\lambda_{i}(a_{i}-a_{i-1}) 
 =\sum_{i=1}^{m-1}(\lambda_{i+1}-\lambda_{i})a_{i}\\
 & \geq\sum_{i=1}^{m-1}(\lambda_{i+1}-\lambda_{i})a =(\lambda_{m}-\lambda_{1})a.
\end{align*}
Since $\|s_{\infty}\|_{\rho,L^{q}}=1$ and $\text{tr}s_{\infty}=0$ a.e., we have $m>1$ and $\lambda_{m}-\lambda_{1}\geq1$.
On the other hand, we have the following. 
\[
M(ts_{\infty})\leq t2\pi L(\theta-\mu(\S))\text{rk}\S,
\]
from (\ref{eq:-38}). Hence, 
\begin{align}
L(\theta-\mu(\S))\text{rk}\S & \geq a=\min_{1\leq i\leq m-1}\{\text{rk}\mathcal{S}_{i}(\mu(\S)-\mu(\mathcal{S}_{i}))\}.\label{eq:-29}
\end{align}
This violates the assumption that $\S$ is a
good approximation of $(\E_{\infty},\theta)$ at level $L$. Therefore, $\|\log h_{i}\|_{\rho,L^{q}}$
is bounded, and $\|\log h_{i}\|_{\rho,L^{\infty}}$ is bounded by
Lemma \ref{lem:Suppose-that-.}. 
\end{proof}

Next, we discuss higher order estimates.  We fix the Hermitian--Einstein metric $H_{0}$ on $\S$. For a smooth
Hermitian metric $H$ on $S$, we denote $s=\log(HH_{0}^{-1})$. As the reference metric $H_0$ is fixed, to keep the notation simple, we define
\[
\|s\|_{C^{k}}=\sup_{X}\sum_{l=0}^{k}|\nabla^{l}s|_{2,H_{0}}.
\]
Here $\nabla$ is the induced connection from $H_{0}$. Fix a finite
open ball covering $\{U_{\gamma}\}_{\gamma\in I}$ of $X$ where $S$ is trivial
on each $U_{\gamma}$. We may define the Holder semi-norm by 
\[
[s]_{C^{k,\alpha}}:=\sup_{\gamma\in I;x,y\in U_{\gamma}}\frac{|\nabla^{k}s(x)-\nabla^{k}s(y)|_{2,H_{0}}}{|x-y|^{\alpha}}.
\]
We define $\|s\|_{C^{k,\alpha}}=\|s\|_{2,H_0,C^{k}}+[s]_{2,H_0,C^{k,\alpha}}$. 
For simplicity, we will abuse the notation and denote $\|H\|_{C^{k,\alpha}}:=\|s\|_{C^{k,\alpha}}$.
Technically speaking, the space of Hermitian metrics on $X$ is not
a vector space. However, we notice that if $C^{-1}H_{0}<H<CH_{0}$
for some uniform constant $C>0$, then let $h=HH_0^{-1}$ and $s=log(h)$, we have
\[
(C')^{-1}\sum_{l=1}^{k}|\nabla^{l}h|<\sum_{l=1}^{k}|\nabla^{l}s|<C'\sum_{l=1}^{k}|\nabla^{l}h|
\]
for some uniform constant $C'$ depending only on $C$ and $X$. Thus,
we call $H$ bounded in $C^{k,\alpha}$ if $\|H\|_{C^{k,\alpha}}=\|s\|_{C^{k,\alpha}}<\infty$.
If $K_{i}$ is a sequence of Hermitian metrics and $k_{i}=\log(K_{i}H_{0}^{-1})$,
we say $K_{i}$ convergent to $K_{\infty}$ in $C^{k,\alpha}$ if
$\lim_{i\to\infty}\|k_{i}-k_{\infty}\|_{C^{k,\alpha}}=0$.

Given notation as above, $\{h_{i}\}$ in Theorem
\ref{thm:Suppose-that-} has a convergent subsequence in $C^{0}$. See \cite[Proposition 2.1]{simpson88constructVarofhodge} for the proof.
\begin{lem}
\label{lem:Assume-the-same}Assume the same conditions as in Theorem
\ref{thm:Suppose-that-}. Then $\{h_{i}\}$ has a convergent subsequence
in $C^{0}$.
\end{lem}

The proof for $C^1$ bound is also standard. See for instance \cite[Lemma 19]{Donaldson85}.

\begin{prop}
Assume the same conditions as in Theorem \ref{thm:Suppose-that-}.
$|\dbar h_{i}|_{H_{0}}$ is uniformly bounded. 
\end{prop}

% \begin{proof}
% We argue by contradiction. If not, there exist a subsequence $i_{k}$,
% a sequence of points $p_{i_{k}}\in X$ such that 
% \[
% \sup_{X}|\dbar h_{i_{k}}|_{H_{0}}=|\dbar h_{i_{k}}(p_{i_{k}})|_{H_{0}}\to\infty.
% \]
% Replace $h_{i},p_{i}$ by $h_{i_{k}},p_{i_{k}}$. We may further assume
% that $\lim_{i\to\infty}p_{i}=p\in X$ as $X$ is compact. Let $m_{i}=\sup_{X}|\dbar h_{i}|_{H_{0}}$.
% Pick a normal coordinate in a disk neighborhood $D$ centered at $p$.
% We may identify $D$ with $D_{1}(0)\subset\C$ and $H_{i}(x)$ with
% Hermitian matrix valued functions in $D_{1}(0)$. By Lemma \ref{lem:Assume-the-same},
% we may assume $H_{i}$ converges to some continuous $H_{\infty}$
% in $C^{0}.$ Let $\tilde{H}_{i}(x):=H_{i}(m_{i}^{-1}x)$. Then, $|\nabla\tilde{H}_{i}|$
% is uniformly bounded. Notice that $\tilde{H}_{i}$ satisfies
% \begin{align*}
% \Delta\tilde{H}_{i} & =\sqrt{-1}\Lambda F_{\tilde{H}_{i}}\tilde{H}_{i}+\sqrt{-1}\Lambda\pdv\tilde{H}_{i}\tilde{H}_{i}^{-1}\dbar\tilde{H}_{i}
% \end{align*}
% The right hand side is uniformly bounded in $D_{1}(0)$, since $\log\tilde{H}_{i}$,
% $|\nabla\tilde{H}_{i}|$ and $\|\beta_{i}\|_{C^{0}}$ are uniformly
% bounded. By standard elliptic $W^{2,p}$ estimate, $\tilde{H}_{i}$
% has uniformly bounded $W^{2,p}$ norm for all $p<\infty$. By Rellich's
% Theorem, $\tilde{H}_{i}$ has convergent subsequence in $C^{2,\alpha}$.
% However, since $H_{i}$ converges to $H_{\infty}$ in $C^{0}$, $\tilde{H}_{i}$
% must converge to a constant matrix, which violates the assumption
% that $\lim_{i\to\infty}|\nabla\tilde{H}_{i}(p_{i})|=1$.
%\end{proof}
Finally we finish the proof of  Theorem \ref{thm:For-any-,}.
\begin{proof}[Proof of Theorem \ref{thm:For-any-,}]
The matrix valued function $h_{i}$ is uniformly bounded in $C^{1}$ and solves
\begin{equation}
\sqrt{-1}\Lambda F_{H_{i}}=\mu(E_{i})I-\sqrt{-1}\Lambda\beta_{i}^{\dagger}\wedge\beta_{i}.\label{eq:-35}
\end{equation}
The left hand side is now a uniformly elliptic operator. Denote the
right hand side of (\ref{eq:-35}) by $g_{i}$. By  standard Schauder
estimates, for $\alpha\in(0,1),k\geq2$, we have 
\begin{equation}
\|H_{i}\|_{C^{k,\alpha}}\leq C(k,\alpha,X)(\|H_{i}\|_{C^{0}}+\|g_{i}\|_{C^{k-2,\alpha}}).\label{eq:-36}
\end{equation}
By Lemma \ref{thm:Assume-the-same}, $\|g_{i}\|_{C^{k-2,\alpha}}$
is uniformly bounded. Hence, $\|H_{i}\|_{C^{k,\alpha}}$ is uniformly
bounded. Therefore, by Arzela-Ascoli theorem, $\{H_{i}\}$ has a convergent
subsequence in $C^{k}$. 
\end{proof}

\section{Convergence of metrics}\label{Section:Convergence of metrics}

Consider a sequence of stable holomorphic vector bundles and injective holomorphic bundle
morphisms  as in (\ref {eq:-28}):
\begin{equation}
\mathcal{E}_{0}\overset{\iota_{0}}{\to}\mathcal{E}_{1}\overset{\iota_{1}}{\to}\mathcal{E}_{2}\to\cdots\to\mathcal{E}_{k}\overset{\iota_{k}}{\to}\cdots.\label{eq:-72}
\end{equation}

Recall that $\mathcal{E}_\infty=\varinjlim \mathcal{E}_i$  is a quasi-coherent sheaf and $\theta:=\sup_{\mathcal{F}} \mu(\mathcal{F})$ where the supremum is taken over all the coherent subsheaf $\mathcal{F}$ in $\mathcal{E}_\infty$. In this section, we assume that $\theta$ is a finite number and $\{\E_{i}\}_i$ is a
well-approximating sequence of $(\E_{\infty},\theta)$ (Definition \ref{def:We-say-intro}).

By Donaldson-Uhlenbeck-Yau theorem, $\E_{i}$
admits a Hermitian--Einstein metric $H_{i}$. In the beginning of Section \ref{section:uniform}, we considered a normalization such that
$\det((H_{i}|_{E_{0}})H_{0}^{-1})=1$
at a given point $p\in X$.
Then Theorem \ref{thm:For-any-,} provides a convergent
subsequence of metrics. However, this normalization is suitable
only for $E_0$. Our goal is to obtain a suitable uniform normalization
that after choosing a subsequence, $H_{i}|_{E_{k}}$ is a convergent
sequence for each $k$. 
\begin{thm}
\label{thm:For-each-,}For   any  well-approximated sequence as in (\ref{eq:-72}) and any fixed $r\in\N$, there exists a subsequence
$(E_{k_{i}},H_{k_{i}})$, a sequence of smooth Hermitian metrics $H_{k_{i},\infty}$
on $E_{k_{i}}$, and a sequence of constants $b_{i}\in\mathbb{R}$
such that for any fixed $l\in\N$, $e^{b_{i}}H_{k_{i}}|_{E_{k_{l}}}$ converges
in $C^{r}$ to $H_{k_{l},\infty}$ with respect to $H_{k_l}$ as $i\to \infty$. Moreover, $H_{k_{l},\infty}|_{E_{k_{m}}}=H_{k_{m},\infty}$
for any $l\geq m$. 
\end{thm}

\begin{proof}
We prove the case $r=0$. The proof for $r\geq1$ is essentially the same
by Theorem \ref{thm:For-any-,}. We argue by induction. In the first
step, we choose for each $H_{k}$ a real constant $c_{k,0}$ such
that at $p$, $\det((e^{c_{k,0}}H_{k}|_{E_{0}})H_{0}^{-1})=1$. Theorem
\ref{thm:For-any-,} produces a subsequence $(E_{k_{i}},e^{c_{k_{i},0}}H_{k_{i}})$
such that $\log(e^{c_{k_{i},0}}H_{k_{i}}|_{E_{0}}H_{0}^{-1})$ converges
in $C^{0}$ to some $s_{0,\infty}$ on $E_{0}$ and $H_{0,\infty}=e^{s_{0,\infty}}H_{0}$
is a Hermitian metric on $E_{0}$. Denote $b_{0}=0$. We then replace
$\{(E_{k},H_{k})\}$ by $\{(E_{0},H_{0}),(E_{k_{1}},H_{k_{1}}),\cdots\}$.

In the second step, we choose $c'_{k,1}$ such that at $p$, $\det(e^{c'_{k,1}}e^{c_{k,0}}H_{k}|_{E_{1}}H_{1}^{-1})=1$.
Similar procedure produces a subsequence $(E_{k_{i}},e^{c_{k_{i},0}+c'_{k_{i},1}}H_{k_{i}})$
such that $\log(e^{c_{k_{i},0}+c'_{k_{i},1}}H_{k_{i}}|_{E_{1}}H_{1}^{-1})$
converges in $L^{\infty}$ to some $\tilde{s}_{1,\infty}$ on $E_{1}$
and $\tilde{H}_{1,\infty}=e^{\tilde{s}_{1,\infty}}H_{1}$ is a Hermitian
metric on $E_{1}$. Since $e^{c_{k_{i},0}+c'_{k_{i,1}}}H_{k_{i}}|_{E_{0}}$
also converges. There exists a constant $c_{0,\infty}$ such that
\[
\lim_{i\to\infty}e^{c'_{k_{i}}+c_{k_{i,1}}}H_{k_{i}}|_{E_{0}}=e^{c_{0,\infty}}H_{0,\infty}.
\]
The convergence is in $C^0$ with respect to $H_1$. Thus, $c'_{k_{i},1}\to c_{0,\infty}$. So we may choose $c_{k_{i},1}:=c'_{k_{i},1}-c_{0,\infty}$
and $H_{1,\infty}=e^{-c_{0,\infty}}\tilde{H}_{1,\infty}$. Further
replacing by a subsequence, we may assume that $|c_{k_{i},1}|<2^{-i}$
for all $i\geq1$. Then, 
\[
\lim_{i\to\infty}e^{c_{k_{i},1}+c_{k_{i},0}}H_{k_{i}}|_{E_{0}}=H_{0,\infty},\ \lim_{i\to\infty}e^{c_{k_{i},1}+c_{k_{i},0}}H_{k_{i}}|_{E_{1}}=H_{1,\infty}.
\]
Denote $b_{1}=c_{1,0}$. We then replace $\{(E_{k},H_{k})\}$ by $\{(E_{0},H_{0}),(E_{1},H_{1}),(E_{k_{2}},H_{k_{2}}),\cdots\}$.

Suppose for $m\geq2$, we have found a sequence $\{(E_{i},e^{\sum_{j=0}^{m-1}c_{i,j}}H_{i})\}_{i=0}^\infty$
and $\{(E_{j},H_{j,\infty})\}_{j=0}^{m-1}$ such that
\begin{equation}
\lim_{i\geq l,\ i\to+\infty}e^{\sum_{j=0}^{m-1}c_{i,j}}H_{i}|_{E_{l}}=H_{l,\infty},\label{eq:-5}
\end{equation}
for all $l=0,1,\cdots,m$. Moreover, $H_{l,\infty}|_{E_{k}}=H_{k,\infty}$
for $l\geq k$. Denote $d_{i,m-1}=\sum_{j=0}^{m-1}c_{i,j}$. We first
pick $c'_{i}$ such that $\det(e^{c'_{i}+d_{i,m-1}}H_{i}|_{E_{m}}H_{m}^{-1})=1$.
Then, replacing $\{(E_{i},H_{i})\}_{i\geq m+1}$ by a subsequence
for $i>m$, $e^{c'_{i}+d_{i,m-1}}H_{i}|_{E_{m}}$ converges to some
Hermitian metric $\tilde{H}_{m,\infty}$ on $E_{m}$. Then there exists
a constant $c_{\infty}$ such that 
\[
\lim_{i\geq m}e^{c_{i}'+d_{i,m-1}}H_{i}|_{E_{m-1}}=e^{c_{\infty}}H_{m-1,\infty}.
\]
By (\ref{eq:-5}), $\lim_{i\to\infty}c_{i}'=c_{\infty}$. We denote
$c_{i,m}:=c_{i}'-c_{\infty}$. Replacing $\{(E_{i},H_{i})\}_{i\geq m+1}$
by a subsequence, we may further assume that 
\begin{equation}
|c_{i,m}|<2^{-i-m},\label{eq:-7}
\end{equation}
 for all $i\geq m$. Let $H_{m,\infty}:=e^{-c_{\infty}}\tilde{H}_{m,\infty}$.
Then, 
\[
\lim_{i\geq m,i\to\infty}e^{c_{i}+d_{i,m-1}}H_{i}|_{E_{m}}=H_{m,\infty}.
\]
Denote $b_{m}=d_{m,m-1}$. 

By induction, we find infinite sequences $\{(E_{i},H_{i})\}_{i=0}^{\infty}$,
$\{H_{i,\infty}\}_{i=0}^{\infty}$, $\{d_{i,m}\}_{0\leq m\leq i}$,
and $\{b_{i}\}_{i=0}^{\infty}$. The following diagram shows the convergence
of the normalized metrics. 

\[
\begin{array}{ccccccc}
H_{0} & e^{d_{1,0}}H_{1} & e^{d_{2,0}}H_{2} & e^{d_{3,0}}H_{3} & \cdots & \to & H_{0,\infty}\\
 & e^{d_{1,0}}H_{1} & e^{d_{2,1}}H_{2} & e^{d_{3,1}}H_{3} & \cdots & \to & H_{1,\infty}\\
 &  & e^{d_{2,1}}H_{2} & e^{d_{3,2}}H_{3} & \cdots & \to & H_{2,\infty}\\
 &  &  & e^{d_{3,2}}H_{1} & \cdots & \to & H_{3,\infty}\\
 &  &  &  & \ddots & \vdots & \vdots
\end{array}
\]
Notice that by the construction, we have for $m>k$,
\[
\left|\sum_{l=k+1}^{m}c_{i,l}\right|\leq\sum_{l=k+1}^{m}2^{-i-l}\leq2^{-(i+k)}.
\]
For any $0\leq k<m$,
\begin{equation}
e^{d_{m,k}-\frac{1}{2^{m+k}}}H_{m}|_{E_{k}}\leq e^{d_{m,m-1}}H_{m}|_{E_{k}}\leq e^{d_{m,k}+\frac{1}{2^{m+k}}}H_{m}|_{E_{k}}.\label{eq:-6}
\end{equation}
Since for fixed $k\geq0$, $e^{d_{m,k}}H_{m}|_{E_{k}}$ converges
to $H_{k,\infty}$, $e^{d_{m,m-1}}H_{m}|_{E_{k}}$ converges to $H_{k,\infty}$
as $m$ goes to $\infty$. 

Finally, we show the smoothness of $H_{k,\infty}$. Let $\{(E_{k},H_{k,\infty})\}$
be constructed as above. By Theorem \ref{thm:For-any-,}, for each
$l\in\N$, $\{H_{k}|_{E_{l}}\}_{k\in\N}$ has a convergent subsequence
in each $C^{r}$ for all $r\in\N$. As $H_{k}|_{E_{l}}$ already converges
to $H_{k,\infty}$ in $C^{0}$, $H_{k,\infty}$ must be smooth. We
have completed the proof.
\end{proof}
\begin{rem}
For our main application, we choose $r=2$ in Theorem \ref{thm:For-each-,}. 
\end{rem}

% \begin{rem}
% % 
% After passing to the subsequence in Theorem \ref{thm:For-each-,}
% with the limit metric $H_{i,\infty}$, we are not able to show 
% \begin{equation}
% \|H_{i,\infty}\|_{e^{b_{i}}H_{i},C^{k,\alpha}}<C,\label{eq:-71}
% \end{equation}
% for some uniform constant $C$ independent of $i$. Fortunately, in our main applications, we do not require
% such uniform estimates.
% \end{rem}

The following proposition shows the existence of well-approximating sequences, which gives the main examples of such sequences used in Theorem \ref{thm:main thm in intro}.
 
\begin{prop}\label{prop:examples on curves}
    Let $X$ be a smooth projective curve over $\C$ with genus $g(X)>0$, $\theta$ be an irrational number, and $\E(\theta^{-},\{f_{2i}\})$ be the colimit of the sequence \begin{equation}
         \E_0\xrightarrow{f_0} \E_2\xrightarrow{f_2} \E_4\rightarrow \cdots \label{eq:-76}\end{equation} as in Definition \ref{definition: colimit and limit objects}.  Then there exists a well-approximated subsequence in sequence (\ref{eq:-76}).  \end{prop}
\begin{proof}
    By Proposition \ref{prop:arithmetic stability condition on curves}, we know that the sequence (\ref{eq:-76}) exists for any smooth projective curve $X$ over $\C$ with genus $g(X)>0$. 
    
     By Proposition \ref{prop:morphisms in farey triangles} and Remark \ref{Rmk:stable case}, any nontrivial morphism $f_{2i}$ in (\ref{eq:-76}) is injective and $\coker(f_{2i})$ is a stable coherent sheaf with positive rank. Hence $\coker(f_{2i})$ is torsion free. Moreover, it is locally free as $\mathrm{dim}(X)=1$. This shows that $f_{2i}$ is an injective holomorphic bundle map for any $i\in\N$. So are any compositions of them, since any extension of locally free sheaves is still a locally free sheaf.   
     
     It suffices to show that for any $\E_{2i}$ in (\ref{eq:-76}) and any coherent subsheaf $\mathcal{F}\subset \E_{2i}$ with strictly smaller rank, we have $$(\theta-\mu(\E_{2i}))\mathrm{rk}\E_{2i}<\mathrm{rk}\mathcal{F}(\mu(\E_{2i})-\mu(\mathcal{F})).$$ By Definition \ref{definition: colimit and limit objects}, we have $\mathrm{rk}\E_{2i}=q_{2i},\  \deg\E_{2i}=p_{2i}$ where $\frac{p_{2i}}{q_{2i}}$ is the $2i$-th convergent of $\theta$. As $\E_{2i}$ is stable and $\mu(\mathcal{F})\in\mathbb{Q}$ with a denominator smaller than $q_{2i}$, we have $$\mathrm{rk}\mathcal{F}(\mu(\E_{2i})-\mu(\mathcal{F}))\geq \frac{\mathrm{rk}\mathcal{F}}{q_{2i}\mathrm{rk}\mathcal{F}}=\frac{1}{q_{2i}}.$$ Thus, it suffices to show that \begin{equation} (\theta-\frac{p_{2i}}{q_{2i}})q_{2i}^2<1,\label{eq:-77}\end{equation} which can be proved by the classic Dirichlet's approximation theorem.
\end{proof}

 % From Proposition \ref{prop: density result}, we have the following direct corollary. 
% \begin{cor}
%   For any  smooth projective curve over $\C$ with genus $g(X)>0$, the set $$\{\theta\in\R|(X,\theta) \text{ is a geometrically well-approximable pair.}\}$$ is of full Lebesgue measure.
% \end{cor}

\section{A smooth topological vector bundle structure\label{sec:smoothtopo}}
Given a sequence of holomorphic vector bundles as in (\ref{eq:-28 intro}), the colimit object exists in the category of quasi-coherent sheaves.  In this section, we study the colimit objects in different categories. In particular, there exists a natural smooth topological vector bundle as the colimit of (\ref{eq:-28 intro}) in the category of smooth vector bundles. Throughout this section, we consider the usual analytic topology on $X$.

In the category of sheaves of abelian groups on $X$, we denote  $\mathcal{C}^{\infty}(E_{i})$ to be the sheaf of smooth sections
of $E_{i}$ and denote
\[
\mathcal{F}_{\infty}:={\varinjlim}\ \mathcal{C}^{\infty} (E_{i}).
\]
Note the stalk 
$
(\mathcal{F}_{\infty})_{x}={\varinjlim}\ \mathcal{C}^{\infty}(E_{i})_{x}.
$

The main theorem is the following. 
\begin{thm}
\label{thm:6.1tVS}    Given an increasing sequence of holomorphic vector bundles as in (\ref{eq:-28 intro}), there exists a smooth vector bundle $E_\infty$ with the following properties.
    \begin{enumerate}
        \item The total space of $E_\infty$ is the colimit of the total space of $E_k$ in the category of topological space. 
        \item \label{enu:smoothcat}There exists a canonical isomorphism $q:\mathcal{F}_\infty \to \mathcal{C}^\infty (E_\infty)$ in the category of sheaves of abelian groups. 
        \item If $\{\mathcal{E}_i\}_i$ is a well-approximating sequence, then $\{H_{k,\infty}\}$ constructed in Theorem \ref{thm:For-each-,} defines a positive non-degenerate Hermitian metric $H_\infty$ on $E_\infty$.
        \item  Given $E_\infty, H_\infty$ as in (3),  let $Q_i=E_i/E_{i-1}$. Then $H_\infty$ further induces an orthogonal splitting such that $
        E_\infty\simeq \hat{\oplus}_{i=0}^\infty Q_i.
        $ 
        %where $\hat{\oplus}_{i=0}^\infty Q_i$ is a smooth topological vector bundle as the inductive limit  $\varinjlim\oplus_{i=0}^k Q_i$ .  
             \end{enumerate}
\end{thm}
\begin{rem}
   For a sequence of closed embeddings of complex vector spaces of local convex spaces $V_1\to V_2\to V_3\to\cdots$, the inductive topology in the category of topological spaces does not need to be linear or locally convex (\cite[pp.44-45]{BierstedtAnintroductionlcInductivelimits}). However, if each $V_i$ has a finite rank, the inductive limit is isomorphic to $\C_0^\infty$. Here $\C^\infty
    _0$ is the space of finite sequences of complex numbers equipped with the inductive topology with respect to inclusion maps of its finite rank subspaces. It is well known that $\C^\infty_0$ is a locally convex space. See for instance  \cite[Theorem 6.1.4]{SakaiTopologyinfinite-dimensionalmanifolds}.
\end{rem}
\begin{rem}
    For locally convex topological vector spaces, the operator $\hat{\oplus}$ is the locally convex direct sum. See \cite[II.6]{SchaeferTVS99}. We use $\hat{\oplus}_{i=0}^\infty Q_i$ to denote the smooth topological vector bundle obtained by the same procedure in (1) as the colimit of bundles of finite rank $\oplus_{i=1}^kQ_i$. The notation emphasizes that each fiber $(\hat{\oplus}_{i=0}^\infty Q_i)_x$ is isomorphic to the locally convex direct sum $\hat{\oplus}(Q_i)_x\simeq \C^\infty_0$.
 \end{rem}
 
\begin{proof}
We denote by $r_k$ the rank of $E_k$.  In the category of topological spaces, let $G_{i}$ be the total space of the vector bundle $E_{i}$. The injective holomorphic bundle map $\iota_{i}:\E_{i}\to \E_{i+1}$ gives
the inclusion $G_{i}\to G_{i+1}$ as topological spaces. Let
$\tau_{i}$ be the usual topology of $G_{i}$. Let
\[
(G_{\infty},\tau):={\varinjlim}(G_{i},\tau_{i})
\]be the colimit object in the category of topological space. We briefly review the
construction of $G_{\infty}$. Let $G_{\infty}=\sqcup_{i}G_{i}/_{\sim},$ where $\sim$ is an equivalent relation that identifies $p\in G_{i},q\in G_{j}$ if there exists $k\geq\max\{i,j\}$ such that the images of $p$ and
$q$ are the same in $G_{k}.$ Let $g_{i}$ be the natural inclusion
map from $G_{i}$ to $G_{\infty}$. Put the topology $\tau$
on $G_{\infty}$ to be the finest topology such that $g_{i}:G_{i}\to G$
is continuous. In other words, a set $U$ is open in $G_{\infty}$
if and only if $U|_{G_{i}}$ is open in $G_{i}$ for any $i$. Notice that the projection maps $\{\pi_i:G_{i}\to X\}$ induce a continuous
projection map $\pi_{\infty}:G_{\infty}\to X$. For each $x\in X$, denote by $\pi_\infty^{-1}(x)$ and $\pi_i^{-1}(x)$ the fibers in $G_\infty$ and $G_i$ respectively.  

We claim that \[{\varinjlim}\pi_i^{-1}(x)=\pi_\infty^{-1}(x).\] If we switch to the category of compactly generated topological spaces  (\cite[Def. 1.1.1]{MayParametrizedhomotopytheory}), then fiber products commute with inductive limits (\cite[Remark 1.6.4]{MayParametrizedhomotopytheory}). Since $G_i$ and $G_\infty$ are compactly generated and $X$ is compact and Hausdorff, the fiber product $\pi_\infty^{-1}(x)=(\varinjlim G_i)\times_X \{x\}$ is the same in the compactly generated category (\cite[Prop. 1.6.1]{MayParametrizedhomotopytheory}). Thus \[{\varinjlim}\pi_i^{-1}(x)=\varinjlim (G_i\times_X \{x\})\underset{homeo}{\simeq} (\varinjlim G_i)\times_X \{x\}=\pi_\infty^{-1}(x). \] We have proved the claim. Since ${\varinjlim}\pi_i^{-1}(x)$ with the inductive topology becomes a locally convex topological vector space (isomorphic to $\C^\infty_0$), we can endow $\pi_{\infty}^{-1}(x)$ a natural structure  of locally convex topological vector space.
 
To show that $\pi_\infty :G_\infty \to X$ admits a vector bundle structure, we need to find a local trivialization. Pick a finite open ball cover $\{U_i\}_{i\in I}$ of $X$. Then each $E_j|_{U_i}$ is trivial. Let $U\in \{U_i:i\in I\}$. Let $\{e_{0,\alpha}:1\leq \alpha\leq r_0\}$ be a smooth local basis for $E_0|_U$. 
We may find a set of sections $\{e_{1,\beta }:1\leq \beta \leq r_1-r_0\}$ of $E_1|_U$ such that $\{\iota_0e_{0,\alpha}\}\cup\{e_{1,\beta}\}$ is a smooth basis for $E_1|_U$. By induction, we may find a set of smooth sections $\{e_\alpha\}_{\alpha=1}^{\infty}\subset  \varinjlim\mathcal{C}^{\infty}(E_{j})(U)$ such that
\begin{equation}
   \label{eq:span} \text{span}_\C\{e_\alpha(x):1\leq\alpha\leq r_k\}=(\iota_{k,\infty}E_k)_x,
\end{equation} where $\iota_{k,\infty}$ is the canonical inclusion $E_k\to E_{\infty}$. Then $\{\iota_{k,\infty}^{-1}e_\alpha\}$ defines a smooth local trivialization:
 \[\phi_k:\pi_k^{-1}(U)\to U\times \C^{r_k},\ (x,\sum_{\alpha=1}^{r_k} v_\alpha \iota_{k,\infty}^{-1}e_\alpha(x)) \mapsto (x,(v_1,\cdots ,v_{r_k} )).\]
We denote the inverse map of $\phi_k $ by   $ \psi_k:U\times \C^{r_k}\to \pi_k^{-1}(U).$  Moreover, the inclusion $\iota_k:E_k\to E_{k+1}$ induces the map 
\[  g'_k: U\times \C^{r_k}\hookrightarrow U\times \C^{r_{k+1}}\]
via $(x,(v_1,\cdots ,v_{r_k} ))\mapsto (x,(v_1,\cdots ,v_{r_k },0,\cdots,0))$ which is $\C$-linear in the second variable. The following diagram commutes

\begin{center}
\leavevmode \xymatrix{\pi_{0}^{-1}(U)\ar[r]^{g_{0}}\ar[d]^{\phi_{0}} & \pi_{1}^{-1}(U)\ar[r]^{g_{1}}\ar[d]^{\phi_{1}} & \pi_{2}^{-1}(U)\ar[d]^{\phi_2} \ar[r]^{g_2} & \cdots \\
U\times\C^{r_0}\ar[r]^{g'_{0}}\ar@<1ex>[u]^{\psi_{0}} & U\times\C^{r_1}\ar[r]^{g'_{1}}\ar@<1ex>[u]^{\psi_{1}} &  U\times \C^{r_1}\ar[r] ^{g'_2}\ar@<1ex>[u]^{\psi_{2}}& \cdots %U\times\C_{0}^{\infty}\ar@{-->}@<1ex>[u]^{\psi_{\infty}}\\
}
\end{center} 
%\pi_{\infty}^{-1}(U)\ar@{-->}[d]^{\phi_{\infty}}\\ 
The direct limit of the first row in the category of topological space is $\pi_{\infty}^{-1}(U)$ and the direct limit of the second row is $U\times \C^\infty_0$. Then the universal property of colimit objects gives continuous maps:
\begin{center}
\leavevmode \xymatrix{\pi_{\infty}^{-1}(U)\ar@<0.5ex>[r]^{\phi_{\infty}} & U\times\C_{0}^{\infty}\ar@<0.5ex>[l]^{\psi_{\infty}}\\
}
\end{center}
Since $\phi_\infty$ and $\psi_\infty$ are continuous and set theoretic inverses to each other, they are homeomorphisms. So we have obtained a local trivialization on $U$.   Suppose $V\in \{U_i\}_{i\in I}$ is another open ball.   We can also construct a trivialization $\phi'_\infty:\pi_\infty^{-1}(V)\to V\times \C^\infty_0$ following the same procedure. In $U\cap V$, the map 
\begin{align}
\label{eq:smoothstru}\phi'_\infty|_{U\cap V} \circ \psi_\infty|_{U\cap V} :V\cap U\times \C^\infty_0\to  V\cap U\times \C^\infty_0    
\end{align}
is  continuous  and   $\C$-linear in the second variable by  construction. Therefore, we have a vector bundle and denote it  $E_\infty$. This concludes the proof of (1).

The construction above also gives a smooth structure on $E_\infty$ that turns $E_\infty$ into a smooth vector bundle with fiber isomorphic to $\C^\infty_0$. We use the language of convenient calculus. For more details, see the monograph \cite{kriegl1997convenient}. Let $W$ be a locally convex vector space. Let $\gamma:\mathbb{R}\to W$ be a continuous map. We call $\gamma$ differentiable at $x$ if
\[
\gamma'(x)=\lim_{t\to 0}\frac{\gamma(x+t)-\gamma(x)}{t} 
\] exists. If $\gamma':\mathbb{R}\to W$ is also a continuous map, then call $\gamma$ a $C^1$ curve. If all higher-order derivatives of $\gamma$ exist and are continuous, we call $\gamma$  a smooth curve in $W$.  
Let $\mathcal{C}_W$ be the set of smooth curves in a locally convex space $W$. Let $c^\infty W$ be the topological vector space with the same underlying vector space $W$ and the final topology with respect to all maps of smooth curves in $W$. 
Following \cite[I.3.11]{kriegl1997convenient}, we give the definition of smooth maps. %The $C_X$ defines a smooth structure on $X$. 
\begin{defn} 
   Let $W,W'$ be locally convex vector spaces. Let $U$ be a $c^\infty$-open set of $W$. We call a map $F:U\to W'$ a \textbf{smooth} map if $F$ maps every smooth curve in $U$ to a smooth curve in $W'$. 
\end{defn}
For general locally convex space $W$, $c^\infty W\not=W$ as topological vector spaces. However, since $\C^\infty_0$ is the strong dual of $\prod_{i\in\N} \C$ which is a Fr\'{e}chet-Schwartz space, $c^\infty\C^\infty_0=\C^\infty_0$. See \cite[Theorem 4.11 (2)]{kriegl1997convenient}.

The bounded subsets of $\C^\infty_0$ are precisely those in some finite dimensional subspaces, \cite[ChII 6.5]{SchaeferTVS99}. In particular, each compact subset is contained in some $\C^N$. Therefore, if $Z$ is a locally compact topological space, for any continuous map $f:Z\to \C^\infty_0$ and $z\in Z$, there is a neighborhood $U_z$ of $z$ and $N\in\N$ such that $f(U_z)\subset \C^N$.  

Now, we show the smoothness of (\ref{eq:smoothstru}). Let $\Phi=\phi'_\infty|_{U\cap V} \circ \psi_\infty|_{U\cap V}$, and $W=\mathbb{\C}\times\C^\infty_0$. $\R$ is locally compact. A map $\gamma:\R\to U\cap V\times \C^\infty_0$ is a smooth curve (hence continuous) if and only if for any $t\in\R$, there exist a neighborhood $(t-\epsilon,t+\ep)$ and some $N\in\N$ such that $\gamma|_{(t-\epsilon,t+\epsilon)}$ is a smooth curve in $\C\times \C^N$. Then we pick any $r_k>N$ such that
\[
\Phi(\gamma|_{(t-\epsilon,t+\ep)})=\phi'_k|_{U\cap V}\circ\psi_k|_{U\cap V} (\gamma|_{(t-\epsilon,t+\ep)}),
 \]
is a smooth curve inside $W$. Thus, $\Phi\circ \gamma$ is a smooth curve in $\C\times \C^\infty_0$. Therefore, by definition,  $\Phi$ is a smooth map.  Apply the above procedure to each $U_i\cap U_j$ to obtain a smooth structure of $E_\infty$. 

We then show that $\mathcal{F}_\infty$ is canonically isomorphic to $\mathcal{C}^\infty (E_\infty)$. Any $U$ open subset of $X$ is locally compact. The same argument above shows that any smooth sections of $E_\infty$ are also locally contained in some finite dimensional subbundles. Notice $s\in \mathcal{F}_\infty (U)$  if for any $x\in U$ there exists a neighborhood $V_x$ of $x$ and $k\in\N$ such that $s|_{V_x}\in \mathcal{C}^\infty(E_k)(V_x)$. So $s$ gives a map from $U$ to $E_\infty$ that is a smooth section of $E_\infty$ in $U$. So we obtain a canonical morphism $q$ from $\mathcal{F}_\infty$ to the sheaf of smooth sections of $E_\infty$. Apparently, $q$ is stalkwise injective and surjective. Thus, $q$ defines an isomorphism. We have proved (2).

Finally, we prove (3) and (4). If $\{\mathcal{E}_i\}_i$ is a well-approximating sequence, then $H_{k,\infty}$ constructed in Theorem \ref{thm:For-any-,} defines a positive Hermitian form $H_\infty$ on $E_\infty$. It is non-degenerate as it restricts to a Hermitian metric on any $E_k$. Now $H_\infty|_{E_k}=H_{k,\infty}$ gives a smooth orthogonal splitting $
\Pi_k:E_k\xrightarrow[]{\simeq}\oplus_{i=0}^k Q_i,$
where $Q_i$  is the quotient bundle $E_k/E_{k-1}$ and $Q_0=E_0$. Moreover since $H_\infty|_{E_k}=H_{l,\infty}|_{E_k}
$  is compatible for any $k\leq l$, we have 
\begin{align}
\leavevmode \label{eq:EQcommutdiag}\xymatrix{E_0\ar[r]\ar[d]^{\Pi_{0}} & E_{1}\ar[r]\ar[d]^{\Pi_{1}} & E_{2}\ar[d]^{\Pi_2} \ar[r] & \cdots \\
Q_0\ar[r] & \oplus^1_{i=0} Q_i \ar[r] &   \oplus^2_{i=0} Q_i\ar[r]& \cdots %U\times\C_{0}^{\infty}\ar@{-->}@<1ex>[u]^{\psi_{\infty}}\\
}
\end{align} 
Using the same construction in (1), (2), we obtain a smooth topological vector bundle of infinite rank $Q_\infty=\hat{\oplus}_i Q_i$ which is the inductive limit of $\oplus_{i=1}^k Q_i$ as $k$ goes to infinity. $(Q_\infty)_x$ is the locally convex direct sum of $(Q_i)_x$  that is also isomorphic to $\C^\infty_0$. Then the diagram (\ref{eq:EQcommutdiag}) again gives a smooth isomorphism $\Pi_\infty$ between $E_\infty$ and $Q_\infty$. The smoothness of $\Pi_\infty$ can be proved similarly as the smoothness of (\ref{eq:smoothstru}). We have finished the proof.
\end{proof}
% \begin{rem}
% It is possible to use convenient calculus to construct a holomorphic topological vector bundle structure $E^{h}_\infty$ so that the sheaf of holomorphic sections of $E^h_\infty$ is isomorphic to the standard analytification  $\mathcal{E}^{an}_\infty$ of $\mathcal{E}_\infty$.  Here we briefly recall the definition of $\mathcal{E}^{an}_\infty$: Let $\iota:(X,\tau)\to (X,\tau_{Z})$ , $x\mapsto x$, where $\tau$ is the ordinary topology of $X$ and $\tau_{Z}$ is the Zariski topology of $X$.  For any quasi-coherent sheaf $\mathcal{G}$,  \[\mathcal{G}^{an}:=\iota^*\mathcal{G}=\mathcal{O}_X^{an}\otimes_{\iota^{-1}\mathcal{O}_X}\iota^{-1}\mathcal{G},\]
% is the standard analytification of $\mathcal{G}$. We will not construct $E^h_\infty$ in this setting, since we are eventually interested in the holomorphic structure of the metric completion of $E_\infty$.
% \end{rem}

 \section{Convergence of curvature \label{Section:Curvature convergence and gauge fixing}}
The purpose of this section is to show that $H_\infty$ given in Theorem \ref{thm:6.1tVS} induces formally a connection $\textbf{A}$ which can be viewed as a Hermitian--Einstein connection. 

Fix $k\in\N$. For each $i\in\N$ and $i<k$, we denote $\Q_{i}=\text{coker}(\iota_{i}:\E_{i}\to \E_{i+1})$
and $q_{i+1}$ the quotient map. 
\[
0\to \E_{i}\xrightarrow[]{\iota_i} \E_{i+1}\xrightarrow[]{q_{i+1}} \Q_{i+1}\to0.
\]
The inner product
$H_{k}|_{E_{i+1}}$ induces an adjoint map $\iota_{i}^{\dagger}:E_{i+1}\to E_{i}$
such that $\< \iota_{i}e,e'\>_{H_{i+1}}=\< e,\iota_{i}^{\dagger}e'\>_{H_{i}}.$
Since $\iota_{i}$ is isometric, $\iota_{i}^{\dagger}\circ\iota_{i}=\text{id}$.
For $j>i$, let $\iota_{j,i}=\iota_{j-1}\circ\cdots\circ\iota_{i}:E_{i}\to E_{j}.$
Let $\iota_{i,i}=\text{id}_{E_{i}}$. We denote the adjoint of $q_{i}$
by $q_{i}^{\dagger}$. Let $\pi_{0}^ {}=\iota_{k,0}^{\dagger}$, and
for $i\leq k$, $\pi_{i}:=q_{i}\circ\iota_{k,i}^{\dagger}$. Set $Q_{0}=E_{0}$.
Then, 
\[
\oplus_{i=0}^{k}\pi_{i}:E_{k}\to\bigoplus_{i=0}^{k}Q_{i}
\]
is a smooth isomorphism between smooth $\C$-vector bundles. Let $A_{H_{k}}$
be the Chern connection of $H_{k}$. We may write ${}_k{\beta}_{i,j}^{\dagger}=-\pi_{i}\dbar\pi_{j}^{\dagger}\in A^{0,1}(\text{Hom}(Q_{j},Q_{i}))$.
Similar to Lemma \ref{lem:If--is-1}, we have the following lemma.
\begin{lem}
\label{lem:Let--be-1}Let $H_{k}$ be a Hermitian--Einstein metric
on $\E_{k}$. Then, ${}_k{\beta}_{i,j}^{\dagger}=0$ if $i>j$, and $\dbar({}_k{\beta}_{i,j}^{\dagger})=\dbar_{H_{k}}^{*}{}_k{\beta}_{i,j}^{\dagger}=0$.
\end{lem}

The next result indicates that $H_{\infty}$
can be viewed almost as a "Hermitian--Einstein metric'' on $E_{\infty}$.
\begin{thm}
\label{thm:Let--be=000020ci}Let $\{(\E_{i},H_{i,\infty})\}$ be given
as in Theorem \ref{thm:For-each-,} with $r=2$. Then, for any fixed
$k$ and $j\to\infty$, we have 
\begin{equation}
\left\Vert \frac{\sqrt{-1}}{2\pi}F_{H_{j,\infty}}|_{\E_{k}}-\theta\omega\text{id}_{\E_{k}}\right\Vert _{2,H_{k,\infty},C^{0}}\to0.\label{eq:-48}
\end{equation}
\end{thm}

\begin{proof}
Suppose $i\geq j\geq k$. Let $H_{i}$ be the Hermitian Einstein metric
on $\E_{i}$. Using $H_{i}$,  
\[
E_{i}=Q_{0}\oplus Q_{1}\oplus Q_{2},
\]
where $Q_{0}=E_{k}$, $Q_{1}=E_{j}/E_{k}$ and $Q_{2}=E_{i}/E_{j}$.
Let $H=H_{i}$. Then we may write the Chern connection of $H$ as
\[
A=\left[\begin{array}{ccc}
A_{0} & -{}_i{\beta}_{0,1}^{\dagger} & -{}_i{\beta}_{0,2}^{\dagger}\\
{}_i{\beta}_{1,0} & A_{1} & -{}_i{\beta}_{1,2}^{\dagger}\\
{}_i{\beta}_{2,0} & {}_i{\beta}_{2,1} & A_{2}
\end{array}\right].
\]
For simplicity, we denote $P_{0}=E_{k},P_{1}=E_{j},P_{2}=E_{i}$.
Since $H$ is Hermitian--Einstein, the curvature matrix can be written
as 
\[
F=\left[\begin{array}{ccc}
F_{H|_{P_{0}}}-{}_i{\beta}_{0,1}^{\dagger}\wedge{}_i{\beta}_{1,0}-{}_i{\beta}_{0,2}^{\dagger}\wedge{}_i{\beta}_{2,0} & 0 & 0\\
0 & * & 0\\
0 & 0 & *
\end{array}\right].
\]
Notice that we have 
\[
F_{H|_{P_{1}}}|_{P_{0}}=F_{H}|_{P_{0}}+{}_i{\beta}_{0,2}^{\dagger}\wedge{}_i{\beta}_{2,0}.
\]
From Chern-Weil formula, we see that 
\begin{equation}
\frac{\sqrt{-1}}{2\pi}\int_{X}-{}_i{\beta}_{0,1}^{\dagger}\wedge{}_i{\beta}_{1,0}-{}_i{\beta}_{0,2}^{\dagger}\wedge{}_i{\beta}_{2,0}\leq(\mu(P_{2})-\mu(P_{1}))\text{rk}P_{1}.\label{eq:-46}
\end{equation}
From Proposition \ref{prop:Suppose--and}, we have 
\[
\|{}_i{\beta}_{2,0}\|_{2,H,L^{\infty}}^{2}\leq C(\mu(P_{2})-\mu(P_{1}))\text{rk}P_{1}\leq C(\theta-\mu(P_{1}))\text{rk}P_{1}.
\]
Here $C$ only depends on $X$. As $\|{}_i{\beta}_{2,0}\|_{\rho,H,L^{\infty}}\leq\|{}_i{\beta}_{2,0}\|_{2,H,L^{\infty}}$,
and $\theta$ is well-approximated, we have 
\begin{equation}
\|\sqrt{-1}F_{H|_{P_{1}}}|_{P_{0}}-2\pi\mu(E_{i})\omega\text{id}_{P_{0}}\|_{2,H|_{P_{0}},L^{\infty}}=O((\theta-\mu(P_{1}))\text{rk}P_{1}).\label{eq:-67}
\end{equation}
Because $\lim_{i\to\infty}H_{i}|_{E_{j}}=H_{j,\infty}$ in $C^{2}$ for fixed $j$, we may let $i$ goes to $\infty$ to have 
\begin{align}
\left\Vert \frac{\sqrt{-1}}{2\pi}F_{H_{j,\infty}}|_{P_{0}}-\theta\omega\text{id}_{P_{0}}\right\Vert _{2,H_{k,\infty},L^{\infty}} & =\left\Vert \lim_{i\to\infty}\left(\frac{\sqrt{-1}}{2\pi}F_{H_{i}|_{E_{j}}}|_{P_{0}}-\mu_{i}\omega\text{id}_{P_{0}}\right)\right\Vert _{2,H_{i}|_{P_{0}},L^{\infty}}\label{eq:-51}\\
 & \leq C((\theta-\mu(E_{j}))\text{rk}E_{j}).\nonumber 
\end{align}
Let $j\to\infty$. We then conclude the proof as $\lim_{j\to\infty}(\theta-\mu(E_{j}))\text{rk}E_{j}=0$.
\end{proof}

Using $H_\infty$, we also obtain a splitting of $E_\infty=\hat{\oplus}_{l=0}^\infty Q_l$ by Theorem \ref{thm:6.1tVS}. Let $\pi_{j}^{m}:E_{m}\to Q_{j}$ be the projection map. Note $\pi_{j}^m=q_{j}\circ\iota_{m,j}^{\dagger H_{m,\infty}}$ where $\iota_{m,j}^{\dagger H_{m,\infty}}$ is the adjoint of $\iota_{m,j}$ with respect to $H_{m,\infty}$. Denote $\pi_{j}^{\dagger {H_m,\infty}}$  the  adjoint of $\pi_{j}^m$ with respect to $H_{m,\infty}$. Let $i,j\leq m$ and \[ \gamma_{i,j}^{\dagger}:=-\pi_{i}^m\dbar \pi_{j}^{\dagger {H_m,\infty}}.\]
Then $\gamma_{i,j}^{\dagger }$ is independent of $m$. In fact, since $H_{m,\infty}|_{E_j}=H_{j,\infty}$, \[\dbar\pi_{j}^{\dagger {H_m,\infty}}=\dbar(\iota_{m,j}q_{j}^{\dagger {H_j,\infty}})=\iota_{m,j}\dbar q_{j}^{\dagger {H_j,\infty}}.\] Thus, $\pi_{i}^m\dbar \pi_{j}^{\dagger {H_m,\infty}}=q_{i}\iota_{i,j}^{\dagger {H_{j,\infty}}}\dbar q_{j}^{\dagger {H_j,\infty}}$ which is independent of $m$. Moreover, $\gamma_{i,j}^{\dagger }=0$ if $i>j$. We denote $\gamma_{j,i}$ the adjoint of $\gamma_{i,j}^\dagger$ with respect to $H_{m,\infty}$ and again $\gamma_{j,i}$ is independent of $m$.

Then, we may write formally the connection induced
by $H_{\infty}$ as

\begin{equation}
d+\mathbf{A}=\left[\begin{array}{cccc}
d+A_0 & -\gamma_{0,1}^{\dagger} & -\gamma_{0,2}^{\dagger} & \cdots \\
\gamma_{1,0} & d+A_{1} & \gamma_{1,2}^\dagger& \cdots\\
\gamma_{2,0} & \gamma_{2,1} & d+A_{2} & \cdots\\
\vdots&\vdots&\vdots&\ddots
\end{array}\right],\label{eq:-41}
\end{equation}
where $A_i$ is the Chern connection on $Q_i$ induced by $H_\infty|_{Q_i}$. Note, the restriction of  $d+\textbf{A}$ on each $E_k$ is clearly the induced Chern connection of $H_\infty|_{E_k}$.  However, for $s$ a smooth section of $E_\infty$ defined on an open set $V$, $(d+\mathbf{A})s$ does not necessarily belong to $C^{\infty}(V,T^*X\otimes E_{\infty})$
since it might have infinite many non-vanishing terms. On the other hand,  using $H_{\infty}$,
we may view $d+\mathbf{A}$ as a $\C$-linear map from $C^{\infty}(V,E_{\infty})$
to $C^{\infty}(V,T^{*}X\otimes\check{E}_{\infty}),$ where $\check{E}_{\infty}$ is the $\C$ anti-linear dual bundle of $E_{\infty}.$ More explicitly, $\< (d+\mathbf{A})s_{1}, s_{2}\>_{H_{\infty}}
$
is a smooth 1-form on $V$ for any $s_{1},s_{2}\in C^{\infty}(V,E_{\infty})$ and the pairing is $\C$-linear in the first variable and anti-linear in the second.

Now, we may formally write down the curvature  
\[
\mathbf{F}=d\mathbf{A}+\mathbf{A}\wedge\mathbf{A}.
\]
Notice that $d\mathbf{A}$ is well defined in $A^{2}(\text{Hom}(E_{\infty},\check{E}_{\infty}))$,
while $\mathbf{A}\wedge\mathbf{A}$ is possibly not since it involves
infinite matrix multiplication. However, we will show that $\mathbf{F}$ is well defined in $\text{Hom}(E_{\infty},A^{2}(E_{\infty}))$.  

For each $0\leq k\leq l\leq j$, we denote 
\[
\ \Gamma_{k,l:j}^{\dagger}=\left[\begin{array}{cccc}
\gamma_{0,l+1}^{\dagger} & \gamma_{0,l+2}^{\dagger} & \cdots & \gamma_{0,j}^{\dagger}\\
\vdots & \vdots & \vdots & \vdots\\
\gamma_{k,l+1}^{\dagger} & \gamma_{k,l+2}^{\dagger} & \cdots & \gamma_{k,j}^{\dagger}
\end{array}\right],\Gamma_{k,j:\infty}^{\dagger}=\left[\begin{array}{ccc}
\gamma_{0,j+1}^{\dagger} & \gamma_{0,j+2}^{\dagger} & \cdots\\
\vdots & \vdots & \vdots\\
\gamma_{k,j+1}^{\dagger} & \gamma_{k,j+2}^{\dagger} & \cdots
\end{array}\right].
\]
We may write
\begin{equation}
d+\mathbf{A}=\left[\begin{array}{ccc}
d+A_{H_{k,\infty}} & -\Gamma_{k,k:j}^{\dagger} & -\Gamma_{k,j:\infty}^{\dagger}\\
\Gamma_{k:j,k} & d+A_{E_{j}/E_{k}} & \cdots\\
\Gamma_{j:\infty,k} & \vdots & \ddots
\end{array}\right].\label{eq:-41-1}
\end{equation}
Then, \begin{equation}
\mathbf{F}=\left[\begin{array}{cc}
F_{H_{k,\infty}}-\Gamma_{k,k:j}^{\dagger}\wedge\Gamma_{k:j,k}-\Gamma_{k,j:\infty}^{\dagger}\wedge\Gamma_{j:\infty,k} & *\\
* & *
\end{array}\right].\label{eq:-64}
\end{equation}
Here $\Gamma_{k,j:\infty}^{\dagger}\wedge\Gamma_{j:\infty,k}=\lim_{l\to\infty}\Gamma_{k,j:l}^{\dagger}\wedge\Gamma_{j:l,k}$
if the limit exists. As a result of Theorem \ref{thm:Let--be=000020ci},
we have 
\begin{cor}
\label{cor:For-any-fixed}For any fixed $k,$ 
\begin{equation}
F_{H_{k,\infty}}-\lim_{j\to\infty}\Gamma_{k,k:j}^{\dagger}\wedge\Gamma_{k:j,k}=\frac{2\pi}{\sqrt{-1}}\theta\omega\text{Id}_{E_{k}},\label{eq:-43}
\end{equation}
under $C^{0}$-norm. In particular, for any $j\geq k$, $\Gamma_{k,j:\infty}^{\dagger}\wedge\Gamma_{j:\infty,k}$
defines a smooth section of $A^{1,1}(\text{End}(E_{k}))$ with $L^{\infty}$-estimate
\begin{equation}
\sup_{X}|\Gamma_{k,j:\infty}|_{2,L^{\infty}}^{2}\leq C(\theta-\mu(E_{j}))\text{rk}E_{j},\label{eq:-49}
\end{equation}
for some $C=C(X)$. Furthermore, $\mathbf{F}=2\pi\sqrt{-1}\theta\omega\text{Id}\in \text{Hom}(E_{\infty},A^{2}({E}_{\infty}))$ is a well-defined operator.
\end{cor}

\begin{proof}
(\ref{eq:-43}) is just a rephrasing of (\ref{eq:-48}). For fixed
$i,j$, $\gamma_{i,j}^{\dagger}$ is a smooth limit as $m\to\infty$ of the second fundamental
forms ${}_m\beta^{\dagger}_{i,j}$. Hence, $\Gamma_{k,j:l}^{\dagger}$ is smooth
for finite $j,l$. By Theorem \ref{thm:Let--be=000020ci}, 
\[
F_{H_{k,\infty}}-\Gamma_{k,k:\infty}^{\dagger}\wedge\Gamma_{k:\infty,k}=\frac{2\pi}{\sqrt{-1}}\theta\omega\text{Id}_{E_{k}}.
\]
Note that
\[
\Gamma_{k,j:\infty}^{\dagger}\wedge\Gamma_{j:\infty,k}=\lim_{l\to\infty}\Gamma_{k,j:l}^{\dagger}\wedge\Gamma_{j:l,k}
\]
is locally viewed as a finite dimensional matrix. Since $H_{k,\infty}$ is smooth,
we see that $\Gamma_{k,j:\infty}^{\dagger}\wedge\Gamma_{j:\infty,k}$
is also smooth. Since $\gamma_{i,j}^\dagger$ is entry-wise the limit of ${}_k\beta^\dagger_{i,j}$ , (\ref{eq:-49}) follows from (\ref{eq:-51}) and Fatou's lemma.  
\end{proof}

The next corollary shows that the curvature of $Q_l$ with respect to $H_\infty$ is bounded by a uniform constant. It will be used in Section \ref{section:Holomorphic Hilbert bundles}. 
\begin{cor}
$F_{H_{\infty}|_{Q_{l}}}\in A^{1,1}(End(Q_{l}))$ satisfies 
\[
\|F_{H_{\infty}|_{Q_{l}}}\|_{\rho,H_\infty,L^{\infty}}\leq C(\theta,X)
\]
for some uniform constant $C(\theta,X)$ independent of $l$ with
fiberwise operator norm induced by $H_{\infty}|_{Q_{l}}$. \label{cor:bound curvature}
\end{cor}

\begin{proof}
By $C^{2}$ convergence, 
\[
F_{H_{\infty}|_{Q_{j}}}=\lim_{k\to\infty}F_{H_{k}|_{Q_{j}}}.
\]
Notice that 
\[
F_{H_{k}|_{Q_{j}}}=F_{H_{k}}|_{Q_{j}}-\sum_{l\leq j-1}{}_k\beta_{j,l}\wedge{}_k\beta_{l,j}^{\dagger}-\sum_{l=j+1}^{k}{}_k\beta_{j,l}^{\dagger}\wedge {}_k\beta_{l,j}.
\]
Then, by Chern-Weil formula, 
\begin{align*}
\sum_{l\leq j-1}\| {}_k\beta_{j,l}\|_{H_k,2,L^{2}}^{2} & \leq\sum_{l\leq j-1}\sum_{i\geq j}\| {}_k\beta_{i,l}\|_{H_k,2,L^{2}}^{2}=(\mu(E_{k})-\mu(E_{j-1}))\text{rk}E_{j-1},\\
\sum_{l=j+1}^{k}\| {}_k\beta_{j,l}\|_{H_k,2,L^{2}}^{2} & \leq\sum_{l\geq j+1}\sum_{i\leq j}\| {}_k\beta_{l,i}\|_{H_k,2,L^{2}}^{2}=(\mu(E_{k})-\mu(E_{j}))\text{rk}E_{j}.
\end{align*}
Together with Lemma \ref{lem:Let--be-1} and \ref{prop:Suppose--and},
we have 
\begin{align}
|F_{H_{k}|_{Q_{j}}}|_{\rho,H_k} & \leq2\pi|\mu(E_{k})|+C(\mu(E_{k})-\mu(E_{j}))\text{rk}E_{j}+C(\mu(E_{k})-\mu(E_{j-1}))\text{rk}E_{j-1}\label{eq:-48-1}\\
 & \leq C(\theta,X).\nonumber 
\end{align}
Thus, $F_{H_{k}|_{Q_{j}}}$ is uniformly bounded independent of both
$j,k$. Hence, $|F_{H_{\infty}|Q_{j}}|_{\rho,H_\infty}\leq C(\theta,X)$
is also uniformly bounded, independent of $j$.
\end{proof}

\section{Holomorphic Hilbert bundles}\label{section:Holomorphic Hilbert bundles}

In this section, we will take the metric completion of $(E_{\infty},H_{\infty})$
and prove that the completion space has a projectively flat holomorphic Hilbert bundle structure. This will conclude the proof of Theorem \ref{thm:main thm in intro} and Theorem \ref{thm:main2}. 
 
Let $W$ be a separable complex Hilbert space. We denote $L(W)$ the bounded linear operators on $W$ equipped with the operator norm $|\cdot|_\rho$. Let $\textbf{GL}(W)$ be the general linear group of $W$ which consists of invertible bounded linear operators. Let $\textbf{U}(W)$ be the unitary group of $W$ that consists of unitary operators in $\textbf{GL}(W)$. The definition of a Hilbert bundle is a mimic of the ordinary vector bundle using local trivializations or local transition groups. For example, we can define it as follows.
\begin{defn}
\label{def:hilbertbund} A continuous Hilbert bundle structure with fiber isomorphic to a Hilbert
space $W$ over a finite dimensional compact complex manifold $Y$ is given by an
open cover $\{U_{\alpha}\}_{\alpha\in I}$ of $Y$ and a collection
of transition map $\{g_{\alpha\beta}\}_{\alpha,\beta\in I}$ such
that 
\begin{enumerate}
\item $g_{\alpha\beta}$ is a continuous map from $U_{\alpha}\cap U_{\beta}$
to the structure group  $\mathbf{U}(W)$;
\item If $x\in U_{\alpha}\cap U_{\beta}\cap U_{\gamma}$, the cocycle condition
holds:
\[
g_{\alpha\beta}\circ g_{\beta\gamma}\circ g_{\gamma\alpha}|_{x}=\text{Id}_{W}.
\]
\end{enumerate}
We then take $\mathbf{W}=\left(\sqcup_{\alpha}U_{\alpha}\times W\right)/(x,\xi)\sim(x,g_{\beta\alpha}(x)\xi)$
with the projection $\pi_{\mathbf{W}}([(x,\xi)])=x$ to be the Hilbert
bundle over $Y$. If $g_{\alpha\beta}$ are smooth maps, we call $\mathbf{W}$ a smooth Hilbert bundle. 
\end{defn}
Similarly, holomorphic Hilbert bundles can be defined.
\begin{defn}
    Suppose that $\textbf{W}$ is a smooth Hilbert bundle over $Y$ with characteristic fiber isomorphic to $W$.  An
open cover $\{U_{\alpha}\}_{\alpha\in I}$ of $Y$ and a collection
of transition map $\{g_{\alpha\beta}\}_{\alpha,\beta\in I}$  defines a holomorphic structure on $\textbf{W}$ if each $g_{\alpha\beta}$ is a holomorphic map from $U_{\alpha}\cap U_{\beta}$
to the structure group  $\mathbf{GL}(W)$ and satisfies the cocycle condition as in Definition \ref{def:hilbertbund}. A smooth Hilbert bundle with a given holomorphic structure is called a holomorphic Hilbert bundle.
\end{defn}
Given 2 continuous (resp. smooth/holomorphic) Hilbert bundles $\textbf{W}_1$ and $\textbf{W}_2$,  continuous (resp. smooth/holomorphic) bundle maps between $\textbf{W}_1$ and $\textbf{W}_2$ are defined in the usual way using local trivializations. A bundle map $f:\textbf{W}_1\to\textbf{W}_2$ is a continuous (resp. smooth/holomorphic) isomorphism if $f$ admits a continuous (resp. smooth/holomorphic) inverse.

From the discussion in Section \ref{sec:smoothtopo}, for each $x\in X$, the fiber $E_{\infty}$ at $x$ is 
isomorphic to $\C^\infty_0$. Then under $H_{\infty}$,
the fiberwise completion of $E_{\infty}$ can be given a holomorphic Hilbert bundle structure. The following theorem rephrases Theorem \ref{thm:main2} in a slightly different way.

\begin{thm}\label{thm:holomorphic subbundle}
\label{thm:-can-be} Assume the same conditions as in Theorem \ref{thm:main2}. Let $(E_{\infty},H_{\infty})$ be given as in Theorem \ref{thm:6.1tVS}, (3). Then $(E_\infty,H_\infty)$ can be completed to
a holomorphic Hilbert bundle $(\mathbf{E},H_{\infty})$ and the induced unitary connection $\mathbf{A}$ is a Hermitian--Einstein connection. Furthermore, the standard analytification $\mathcal{E}^{an}_\infty$ of $\mathcal{E}_\infty$ holomorphically embeds into the sheaf of holomorphic sections of $\textbf{E}$. 
\end{thm}

\begin{rem}
\label{rem:holomorphicembedding}    Here we briefly recall the definition of $\mathcal{E}^{an}_\infty$. Let $\iota:(X,\tau)\to (X,\tau_{Z})$ , $x\mapsto x$, where $\tau$ is the ordinary topology of $X$ and $\tau_{Z}$ is the Zariski topology of $X$.  For any quasi-coherent sheaf $\mathcal{G}$,  \[\mathcal{G}^{an}:=\iota^*\mathcal{G}=\mathcal{O}_X^{an}\otimes_{\iota^{-1}\mathcal{O}_X}\iota^{-1}\mathcal{G},\]
is the standard analytification of $\mathcal{G}$. The sheaf of holomorphic sections of $\textbf{E}$ naturally contains each $\mathcal{E}_i$ as an analytic subsheaf. In fact, for any holomorphic local section $t$ of $\mathcal{E}_k$, we have $\dbar_\textbf{A} t=0$ by the construction of $\textbf{A}$ in (\ref{eq:-41}).   Let $s$ be a holomorphic section of $\mathcal{E}_\infty^{an}(U)$ for some open set $U\subset X$. By construction, for any $x\in U$, there exists an open neighborhood $ V_x\subset U$  and $k\in\N$ such that $s|_{V_x}$ is a holomorphic section in some $\mathcal{E}_k(V_x)$.  Then $s|_{V_x}$ is embedded into a local holomorphic section of $\textbf{E}$ . Thus $\mathcal{E}_\infty^{an}$ is naturally injected to the sheaf of holomorphic sections of $\textbf{E}$.   
\end{rem}

\begin{rem}
 In fact, we will prove in \cite{Holomorphicsimplicity} that, for almost every irrational number $\theta$, every holomorphic Hilbert bundle in Theorem \ref{thm:holomorphic subbundle} is simple. A famous theorem of Kuiper \cite{Kuiper1965} indicates that any continuous/smooth infinite dimensional separable Hilbert bundle is trivial. Hence, the holomorphic structure is essential here.
\end{rem}

Our proof of Theorem \ref{thm:-can-be} consists of 3 steps. First, we use Uhlenbeck's theorem to uniformly control the norm of connections of finite rank vector bundles. 
Next, with suitable estimates, we prove that the fiberwise completion with respect to $H_\infty$ has a continuous Hilbert bundle
structure. In the last step, we show that the connection $\textbf{A}$ is smooth using elliptic regularity theory to put a compatible smooth structure
on $\mathbf{E}$. The holomorphic structure is a direct consequence of the Koszul-Malgrange theorem \cite{KoszulMalgrange58}. 

\subsection{Gauge fixing lemma}

We need a gauge fixing theorem proved by Uhlenbeck \cite{UhlenbeckLp}. Let $Z$ be
a trivial vector bundle of rank $\text{rk}Z$ with a smooth Hermitian
metric $h$ over $Q=[-1,1]\times[-1,1]\subset\C$. Let $A$ be a unitary connection
compatible with $h$. Let $F(A)=F_{xy}(A)dx\wedge dy$ be the curvature
of $A$. We call a connection $A=A_{x}dx+A_{y}dy$ on $Z$ a Coulomb
gauge if $\iota_\nu A|_{\pdv Q}=0$ , and
\[
d^{*}A=0,
\]
where $d^{*}A=\frac{\pdv}{\pdv x}A_{x}+\frac{\pdv}{\pdv y}A_{y}$.
\begin{lem}
\label{lem:Let--be-4}Let $Z$ be a vector bundle of rank $\text{rk}Z$
with a smooth Hermitian metric $h$ over $Q\subset\C$. Let
$A$ be a unitary connection of $Z$. Let $F(A)$ be the curvature
of $A$. Suppose that $\|F(A)\|_{\rho,L^{2}(Q)}\leq\epsilon$.
Then, there exists $\epsilon_0>0$ such that  if $\epsilon<\epsilon_{0}$ , $A$ is gauge equivalent to Coulomb gauge $\tilde{A}$  that satisfies
\begin{equation}
\|\tilde{A}\|_{\rho,W^{1,m}(Q)}\leq C\|F(A)\|_{\rho,L^{m}(Q)},\label{eq:-53}
\end{equation}
for $m\geq2$. If $m>2$, then for some $\alpha\in(0,1)$ (if $m=\infty$, then for any $\alpha\in (0,1)$),
we have 
\begin{equation}
\|\tilde{A}\|_{\rho,C^{\alpha}(Q)}\leq C(\alpha)\|F(A)\|_{\rho,L^{m}(Q)}.\label{eq:-54}
\end{equation}
Moreover, the constant $\epsilon_{0}$ and $C$ are independent of
$\text{rk}Z$.
\end{lem}
In fact, this lemma can be generalized to Hilbert bundles. We include a proof in Proposition \ref{prop:coulombgaugeuniform} in the appendix. (\ref{eq:-54}) is a direct consequence of (\ref{eq:-53}) by Morrey's embedding theorem where the embedding constant is also independent of $Z$, \cite[Theorem 6.2]{Arendt2016MappingTF}.  

\begin{lem}
\label{lem:Let--be-2} Let $A_{j}$ be the Chern connection of $H_{\infty}|_{Q_{j}}$.
Then there exists $r>0$ such that 
\begin{enumerate}
\item $X$ can be covered by finite many geodesic balls $\{B_{r}(x_{i})\}_{i=1}^{N}$.
\item For any $j\in\N$, there exists a unitary frame $\mathbf{s}_{j}=\{s_{j,l}\}_{l=1}^{\text{rk}Q_{j}}$
for on each $B_{r}(x_{i})$ such that under this frame
\begin{equation}
\|A_{j}\|_{\rho,C^{\alpha}(B_{r}(x_{i}))}\leq1.\label{eq:-61}
\end{equation}
\end{enumerate}
\end{lem}

\begin{proof}
By Corollary \ref{cor:bound curvature}, we know that $\|F_{H_{\infty}|Q_{j}}\|_{\rho(Q_{j}),L^{\infty}(X)}\leq C(\theta,X)$.
Fix some $r<$injectivity radius of $X$ to be chosen later. We may
identify $B_{2r}(x)$ with $B_{2r}(0)\subset\C^{2}$. Then $Q_{j}|_{B_{2r}(x)}$
is a trivial bundle. By rescaling $x\mapsto rx$, we have 
\[
\|\tilde{F}_{H_{\infty}|Q_{j}}\|_{\rho,L^{\infty}(B_{2}(0))}\leq C(\theta,X)r^{2}.
\]
Now choose $r$ small such that $C(\theta,X)r^{2}\leq\epsilon_{0}$. Apply Lemma \ref{lem:Let--be-4} to $B_1(0)\subset Q$ to obtain a unitary frame
$\mathbf{s}_{j}$ such that 
\[
\|\tilde{A}_{j}\|_{\rho,C^{\alpha}(B_{1}(0))}\leq\|\tilde{A}_{j}\|_{\rho,C^{\alpha}(B_{1}(0))}\leq C(\alpha)C(\theta,X)r^{2}.
\]
Scale back to have $\|A_{j}\|_{\rho,C^{\alpha}(B_{r}(x_{i}))}\leq C(\alpha)C(\theta,X)r$.
Then choose a smaller $r$ to obtain (\ref{eq:-61}). The above constants are all independent of $Q_{j}$. By the compactness of $X$, we can cover $X$
by finite many balls of radius $r$. 
\end{proof}

\subsection{Continuous structure}
The topological space of the Hilbert bundle $\textbf{E}$ is simply the fiberwise metric completion of $E_\infty$ with respect to $H_\infty$. Therefore, fibers of $\textbf{E}$ are isomorphic to $l^2$, the space of square summable complex number sequences. To give local trivializations, it suffices to choose a good local unitary basis of $E_\infty$ which has been given in Lemma \ref{lem:Let--be-2}. So we have the following proposition.
\begin{prop}
\label{prop:There-exists-a-1}There exists a continuous Hilbert bundle
$\pi_{\mathbf{E}}:\mathbf{E}\to X$ such that each fiber $\mathbf{E}_{x}=\pi_{\mathbf{E}}^{-1}(x)$
is a separable Hilbert space equal to the metric completion of $(E_\infty)_{x}$
with respect to $H_{\infty}$.
\end{prop}

\begin{proof}
For each $l\in\N$, let $\{s_{l,\alpha}(x)\}_{\alpha=1}^{\text{rk}Q_{l}}$
be a unitary frame in some geodesic ball $B_{r}(x_{i})$ given as
in Lemma \ref{lem:Let--be-2}. Let $\mathbf{E}_{x}$ be the completion
of $(E_\infty)_x\simeq\hat{\oplus}_{l=0}^{\infty}(Q_{l})_{x}$ under $H_{\infty}$. Note $\mathbf{E}_{x}$
is a separable Hilbert space which is isomorphic to $l^{2}$ by choosing
a unitary basis. With the unitary frame $\mathbf{s}(x)=\{s_{l,\alpha}(x)\}_{l,\alpha}$
at $x$, we have an isomorphism of Hilbert spaces:
\[
\phi_{\mathbf{s}}(x):l^{2}\to\mathbf{E}_{x},\ (a_{k,\alpha})_{k,\alpha}\mapsto\sum_{i,\alpha}a_{k,\alpha}s_{k,\alpha}(x),
\]
Let $\mathbf{E}=\sqcup_{x\in X}\mathbf{E}_{x}$. Denote $\pi_{\mathbf{E}}:\mathbf{E}\to X,\mathbf{E}_{x}\mapsto x$
the projection.

For $i\not=j$, let $\{s'_{l,\alpha}(y)\}_{\alpha=1}^{\text{rk}Q_{l}}$
be the unitary frame for $Q_{l}$ in $B_{r}(x_{j})$. Then, in $B_{r}(x_{i})\cap B_{r}(x_{j})$,
$\{s'_{l,\alpha}(y)\}_{\alpha=1}^{\text{rk}Q_{l}}$ differs from $\{s_{l,\alpha}(y)\}_{\alpha=1}^{\text{rk}Q_{l}}$
by a smooth unitary transform $u_{l}(y)$. Note
\begin{equation}
du_{l}=A'_{l}u_{l}-u_{l}A_{l},\label{eq:-63}
\end{equation}
where $A_{l},A'_{l}$ are connection forms under $\{s'_{l,\alpha}\}$
and $\{s_{l,\alpha}\}$. Since $A_{l}'$ and $A_{l}$ are uniformly
bounded independent of $l$ (by Lemma \ref{lem:Let--be-2}), we see
that $u_{l}(y)$ is smooth in the unitary group $U(\text{rk}Q_{l})$
with bounded Lipschitz constant independent of $l$. Thus,
\begin{equation}
y\mapsto u_{ij}(y):=\text{diag}(u_{0}(y),u_{1}(y),\cdots)\label{eq:-50}
\end{equation}
defines a continuous map from $B_{r}(x_{i})\cap B_{r}(x_{j})$ to
the unitary group of $l^{2}$.

With the local trivialization $\phi_{\mathbf{s}}:l^{2}\times B_{r}(x_{i})\to\pi_{\mathbf{E}}^{-1}(B_{r}(x_{i}))$
and transition groups defined in (\ref{eq:-50}), we see that $\pi_{\mathbf{E}}:(\mathbf{E},H_{\infty})\to X$
becomes a continuous Hilbert bundle.
\end{proof}
We can further show that transition maps are $C^{1,\alpha}$ regular.
\begin{cor} \label{cor:GaugeC_alpha}
Let $u_{ij}:B_{r}(x_{i})\cap B_{r}(x_{j})\to \textbf{U}(l^{2})$ be the transition
map constructed in Proposition \ref{prop:There-exists-a-1}. Then
$u_{ij}$ lies in $C^{1,\alpha}(B_{r}(x_{i})\cap B_{r}(x_{j}),\textbf{U}(l^{2}))$ for any $\alpha\in(0,1)$.
\end{cor}

\begin{proof}
By (\ref{eq:-63}), for each $l$,
$\|du_{l}\|_{\rho,C^{0}}\leq2$
is bounded as $\|A_{l}\|_{\rho,C^{0}}\leq1$ and $|u_{l}|_{\rho}=1$.
Hence, $u_{l}$ is Lipschitz continuous with Lipschitz constant $\leq2$.
Since $\|A_{l}\|_{\rho,C^{\alpha}}\leq1$, we have $\|du_{l}\|_{\rho,C^{\alpha}}\le6$.
Hence, we have $\|u_{l}\|_{\rho,C^{1,\alpha}}\leq7$. Then
\[
\|u_{ij}\|_{\rho,C^{1,\alpha}}=\sup_{l}\|u_{l}\|_{\rho,C^{1,\alpha}}\leq7.
\]
We have finished the proof.
\end{proof}

\subsection{Smooth and holomorphic structures}

Next, we prove that the transition maps defined in Proposition
\ref{prop:There-exists-a-1} give a smooth structure of $\mathbf{E}$. More importantly, we show that the smooth structure is compatible with the
connection $\mathbf{A}$, i.e. 
$d+\mathbf{A}$ acts on the smooth sections of $\mathbf{E}$. 

Notice that replacing by a subsequence $(E
_k,H_{k,\infty})$ does not change the metric completion, So passing to a subsequence,  we assume 

\begin{equation}
\label{eq:l2summable}
\sum_{j=0}^{\infty}\{(\theta-\mu(E_{j}))\text{rk}E_{j}\}^{\frac{1}{2}}<C'.
\end{equation}
We express the connection under a local trivialization $\phi_{\mathbf{s}}$ in $B_r(x_i)$ as an infinite matrix $\mathbf{A}(x)$ as in (\ref{eq:-41}). The next lemma shows
that $\mathbf{A}$ defines a bounded linear
operator. 
\begin{lem}
\label{lem:Let--be-3}Let $\xi$ be smooth vector field on $X$. Let
$\mathbf{A}_{\xi}(x):=\iota_{\xi(x)}\mathbf{A}(x)$. Then, $\mathbf{A}_{\xi}(x)\in C(B_{r}(x_{i}),L(l^{2}))$.
\end{lem}

\begin{proof}
Let $\mathbf{A}'=\text{diag}\{A_{0},A_{1},\cdots\}$. Since
 $A_{j}$ are uniformly bounded in $C^{\alpha}(B_{r}(x_{i}),|\cdot|_{\rho})$ by $1$ (Lemma \ref{lem:Let--be-2}),
\begin{align*}
|\mathbf{A}'(x)-\mathbf{A}'(y)|_{\rho} \leq\sup_{j\in\N}|A_{j}(x)-A_{j}(y)|_{\rho}\leq |x-y|^\alpha
\end{align*}
Hence, $\iota_\xi\mathbf{A}'\in C^{\alpha}(B_{r}(x_{i}),L(l^{2}))$.
Let 
\begin{align*}
\mathbf{B}_{i}= & \left[\begin{array}{ccccc}
0 & \cdots & 0 & 0 & \cdots\\
\vdots & 0 & -\gamma_{i,i+1}^{\dagger} & -\gamma_{i,i+2}^{\dagger} & \cdots\\
0 & \gamma_{i+1,i} & 0 & 0 & \ldots\\
0 & \gamma_{i+2,i} & 0 & 0 & \cdots\\
\vdots & \vdots & \vdots & \vdots & \ddots
\end{array}\right]\begin{array}{c}
\\i\text{}\\
\\\\\\\end{array}\\
 & \begin{array}{ccccc}
\  &  &  & i\end{array}
\end{align*}
Then, $\mathbf{A}_{\xi}(x)=\iota_{\xi}\mathbf{A}'+\sum_{i=0}^{\infty}\iota_{\xi}\mathbf{B}_{i}.$
By Corollary \ref{cor:For-any-fixed}, 
\begin{align*}
|\mathbf{B}_{i}|_{2}^{2} %&  =\sum_{i=0}^{\infty}-\text{tr}\sqrt{-1}\Lambda\gamma_{i,i+1}^{\dagger}\wedge\gamma_{i+1,i}\\
  \leq-\text{tr}\sqrt{-1}\Lambda\Gamma_{i,i:\infty}^{\dagger}\wedge\Gamma_{i,i:\infty}
  \leq C(X)(\theta-\mu(E_{i}))\text{rk}E_{i}.
\end{align*}
Here, the last inequality comes from (\ref{eq:-49}). Hence, 
\[
|\mathbf{B}_{i}|_{\rho}\leq|\mathbf{B}_{i}|_{2}\leq C(X)\{(\theta-\mu(E_{i}))\text{rk}E_{i}\}^{\frac{1}{2}}.
\]
$\mathbf{B}_{i}$ are continuous operators on $l^{2}$. By (\ref{eq:l2summable}), $\sum_i|\textbf{B}_i|_\rho<C'$.
Hence, $\mathbf{A}_{\xi}(x)$ is also continuous by Weierstrass M-test. We
have finished the proof.
\end{proof}
Corollary \ref{cor:For-any-fixed}
justify the definition of the formal curvature form: 
\begin{align}
\label{lem:projectively flat connection}
\mathbf{F}=d\mathbf{A}+\mathbf{A}\wedge\mathbf{A}=\frac{2\pi}{\sqrt{-1}}\theta\omega\text{Id}
\end{align}
is well defined as a continuous 2-form with values locally in $L(l^2)$. 
\begin{prop}
   \label{prop:smoothconnection}    The connection matrix \textbf{A} is smooth in each $B_r(x_i)$. \end{prop}
\begin{proof}
Identify $B_r(x_i)$ with a ball in $\C$. Let $\omega_0$ be the standard metric on  $\C$. Denote $\sigma=\omega/\omega_0$. Let $\delta$  be the divergence operator with respect to $\omega$ and $d^*$ the standard divergence operator. Then
 \[
 \delta \textbf{A}=d^*\textbf{A}+\< d\log\sigma,\textbf{A}\>_{\omega_0},
 \]
 where the pairing is entry-wise with respect to the standard metric.
 Then  we have $d^*\textbf{A}'=0$ as $d^*A_j=0$ .   Write $\textbf{D}=\textbf{A}^{1,0}-(\textbf{A}')^{1,0}$. Since $\Lambda\dbar_\textbf{A}D=0$ by the Hermitian--Einstein equation, we have
\begin{align}
    \delta\textbf{D}=-\sqrt{-1}\Lambda(\dbar_\textbf{A} \textbf{D}-[\textbf{D},\textbf{A}^{0,1}])\nonumber 
     =\sqrt{-1}\Lambda[\textbf{D},\textbf{A}^{0,1}].
\end{align}
Hence, $\textbf{A}$ satisfies the following elliptic system
\[
\begin{cases} 
        d\textbf{A}=\textbf{F}-\textbf{A}\wedge\textbf{A},\\
        d^*\textbf{A}=\sqrt{-1}\Lambda([\textbf{D},\textbf{A}^{0,1}]+[\textbf{D}^\dagger ,\textbf{A}^{1,0}])-\< d\log\sigma,\textbf{D}-\textbf{D}^\dagger\>_{\omega_0}.
\end{cases}
\]
By Lemma \ref{lem:Let--be-3} and (\ref{lem:projectively flat connection}), $\mathbf{A}\wedge\mathbf{A}$ is
continuous and $d\mathbf{A}\in C(B_{r}(x_{i}),\Omega^{2}(L(l^{2})))$. Apply Proposition \ref{prop:regularityofhodgesystem} with continuous right hand side. We see that $\textbf{A}$ is H\"older continuous. Then we may apply Proposition \ref{prop:regularityofhodgesystem} again with H\"older continuous right hand side to obtain $C^{1,\alpha}$ regularity. By the bootstrapping argument, $\textbf{A}$ is smooth.
 
\end{proof}
We now finish the proof of Theorem \ref{thm:-can-be} as well as Theorem \ref{thm:main2}. 
\begin{proof}[Proof of Theorem \ref{thm:-can-be} and Theorem \ref{thm:main2}]
    By Proposition \ref{prop:smoothconnection}, $\textbf{A}$ is smooth in $B_r(x_i)$. Let $\tilde{\textbf{A}}$ be the connection matrix in $B_r(x_j)$. Let $u_{ij}$ be the unitary transform in Corollary \ref{cor:GaugeC_alpha}. Recall that $u_{ij}$ is entry-wise smooth.  We have 
    \[
    du_{ij}=\tilde{\textbf{A}}u_{ij}-u_{ij}\textbf{A}.
    \]
Since both $\tilde{\textbf{A}}$ and $\textbf{A}$ are smooth, we see that $u_{ij}$ is a smooth section of $\textbf{U}(l^2)$.  The transition maps together give a smooth structure on $\textbf{E}$.

As the curvature matrix is of $(1,1)$-type, $\dbar^2_\textbf{A}=(\dbar+\textbf{A}^{0,1})^2=0$. The existence of a holomorphic structure is then given by  Koszul-Malgrange theorem \cite{KoszulMalgrange58}. See \cite[Chapter X, Theorem 1]{malgrange1958lectures} (the proof utilizes Cauchy's integral formula which is valid with values in a Banach space).

\end{proof}
% \begin{rem}
%     Since the smooth Hilbert bundle we constructed is projectively flat, one can directly construct the holomorphic structure. For example, by tensoring a  trivial Hermitian line bundle $L$ with curvature $-\frac{2\pi}{\sqrt{-1}}\theta\omega$ in a disk domain $D$, $\textbf{E}\otimes L$  is flat in $D$. Then we can find a trivialization over $D$ with vanishing curvature and trivial holomorphic structure.  This trivialization of $\textbf{E}\otimes L$ in turn gives  $\textbf{E}$   a holomorphic structure.  
% \end{rem}

Finally, we prove Theorem \ref{thm:main thm in intro}.
\begin{proof}[Proof of Theorem \ref{thm:main thm in intro}]
    By Proposition \ref{prop:examples on curves}, the sequence constructed in Theorem \ref{thm:first result in intro} has a  well-approximated subsequence. Then we can conclude the proof using Theorem \ref{thm:main2}. 
\end{proof}

\section{Final remarks}\label{section:final remarks}
\subsection{Another type of infinite dimensional vector bundles\label{subsec:final1}}
In this subsection, we will discuss another construction of infinite dimensional vector bundles which serve as examples that are \textit{not} expected to admit Hermitian--Einstein metrics.

For the arithmetic stability condition $(\Coh \mathrm{X}, -\deg+\sqrt{-1}\cdot \mathrm{rk})$ on a complex elliptic curve $X$, and any two primitive integral vectors $(p_1,q_1),(p_2,q_2)$ connected by a Farey geodesic, we can construct a vector bundle of infinite rank on the elliptic curve $X$.
Without loss of generality, we focus on the case $(p_1,q_1)=(0,1),(p_2,q_2)=(-1,1)$. Consider the following infinite sequence of holomorphic vector bundles on $X$. $$\E_1\xrightarrow[]{f_1} \E_2\xrightarrow[]{f_2} \E_3\xrightarrow[]{f_3}\cdots \E_n\xrightarrow[]{f_n}\cdots,$$ where $\E_n$ is a vector bundle of rank $n$ and degree $-1$, and $f_i$ are nonzero morphisms, hence injections by Proposition \ref{prop:morphisms in farey triangles}, and we have $\coker(f_i)\in \mathrm{Pic}^0(X)$ by the same proposition. Moreover, if $\coker(f_i)$ is not isomorphic to $\coker(f_j)$ for any $i,j\gg 0$, we call such a sequence in the \textit{generic} case.

By the proof of \cite[Corollary 5.5]{Continuumenvelops}, one can show that the colimit of such a sequence, denoted by $\E_{\infty}$, is an infinite dimensional vector bundle in the sense of Drinfeld (see \cite{Infinitedimensionalvectorbundles}). And if the sequence is in the generic case, one can show that the colimit vector bundle is simple by the proof of \cite[Theorem 5.3]{Continuumenvelops}.

However, one can show that this sequence does not have a well-approximated subsequence in Definition \ref{def:We-say-intro}. Moreover, this simple vector bundle should not admit any Hermitian--Einstein metrics such that its completion is still a holomorphic simple vector bundle. As if it did, the associated curvature should be $\lim\limits_{n\to \infty}\frac{1}{n}=0.$ Then, after metric completion, it would provide us with an irreducible unitary representation of $\pi_1(X)\simeq \mathbb{Z}^2$ on an infinite dimensional Hilbert space. However, it is well known that such representations are always 1 dimensional (see \cite[Theorem B]{Theoriedeladuality}).

Though such an infinite dimensional vector bundle should not admit Hermitian--Einstein metrics, it is still worthy of studying. For example, a similar vector bundle on Fargues--Fontaine curve plays an important role in the fundamental lemma of p-adic Hodge theory (see \cite[Section 2.2]{farguesfontaine-courbes}).

On the other hand, the projective monodromy representations of our stable Hilbert bundles are very interesting. The general case will be studied in a separate paper; we briefly illustrate the special case of complex elliptic curves in the following subsection.
{
\subsection{A maze to noncommutative tori}\label{subsection: a maze to noncommutative tori}
In this subsection, we briefly illustrate three different ways to obtain noncommutative tori. 

\smallskip
\noindent\textbf{Categorical way} Firstly, let us recall Alain Connes' construction of noncommutative tori. Let $T$ be a two dimensional torus $T=\mathbb{R}^2/\mathbb{Z}^2,$ with a Kronecker foliation associated to an irrational number $\theta'$, i.e. given by the differential equation $$dy=\theta'dx.$$ 

In the theory of noncommutative geometry, we can associate a $C^*$-algebra to such a foliation (see \cite[Chapter 2.9]{connes1994noncommutative}). One can show that the $C^*$-algebra $A_{\theta}$ associated with the Kronecker foliation is generated by two unitaries $U$ and $V$ such that $$VU=\exp(2\pi i\theta)UV,$$ where $\theta$ is another irrational number which is equivalent to $\theta'$ under the $SL(2,\Z)$ action (the Morita equivalence of noncommutative tori). Moreover, one can equip $A_{\theta}$ with a complex structure given by some $\tau\in\mathbb{C}\backslash \mathbb{R}$, such a $\tau$ also provides us with a complex elliptic curve $X_{\tau}$.

Given a holomorphic structure $\tau$ on $A_\theta$, Polishchuk studied the category of holomorphic vector bundles with respect to such a holomorphic structure in \cite{Classficationofholomorphicvectorbundles}. It is related to the Bridgeland stability condition on $D^b(X_{\tau})$ in the following way. 
\begin{thm}[{\cite[Corollary 1.2]{Classficationofholomorphicvectorbundles}}]
    Let $X_{\tau}$ be the associated complex elliptic curve and $\sigma=(\mathcal{P},Z)$ be the standard Mumford stability condition on $D^b(X_{\tau})$, then the category of holomorphic vector bundles with respect to $A_{\theta,\tau}$ is equivalent to $\mathcal{P}(\phi,\phi+1]$, where $\theta=-\cot(\pi\phi)$.
\end{thm}  
This is the categorical way to noncommutative tori.
\smallskip

\noindent\textbf{Monodromy representation.}
The second way is through the monodromy representation of the stable Hilbert bundle we constructed on the complex elliptic curve $X_{\tau}$. 

For a given irrational number $\theta,$ let $\E_{\infty}$ be the infinite dimensional vector bundle we constructed and $\textbf{E}$ be its completion. By Lemma \ref{lem:projectively flat connection}, we know that $\textbf{E}$ admits a projectively flat connection. Hence, its monodromy representation provides a group homomorphism. $$\pi_1(X_{\tau},x_0)\simeq \Z^2\rightarrow \textbf{U}(W)/S^1,$$ where the separable Hilbert space $W$ is the fiber of $\textbf{E}$ at $x_0$, $\textbf{U}(W)$ denotes the space of unitary operators on $W$.

Let $U,V$ be the images of two generators of $\Z^2$, and $\tilde{U},\tilde{V}\in\textbf{U}(W)$ be the lifts of $U,V$. By the commutativity of $U,V$, we have $$\tilde{V}\tilde{U}=\exp(2\pi i\theta')\tilde{U}\tilde{V},$$ where $\theta'\in \R/\Z$ is independent of the lifts and the choice of generators up to sign, and one can choose $\theta'\equiv\theta
 \ (\text{mod} \ \Z)$. Hence, again, we obtain the $C^*$-algebra of noncommutative tori.
\smallskip

\noindent\textbf{Homological mirror symmetry.}
The third way is through homological mirror symmetry. We will present it in an intuitive and very inaccurate way (we use the complex elliptic curve rather than the Tate curve over Novikov field and ignore the local systems on Lagrangians). Roughly speaking, for a stable vector bundle $\E$ with rank $r$ and degree $d$ on an elliptic curve $X$, it corresponds to a closed Lagrangian in the symplectic manifold $(\R^2/\Z^2,dx\wedge dy)$, in fact, it is Hamiltonian isotopic to the image of a straight line with slope $-\frac{d}{r}$ in $\R^2/\Z^2$. 

Hence, intuitively, the infinite dimensional vector bundle $\E_{\infty}$ should correspond to the image of the straight line with slope $-\theta$, which is a nonclosed Lagrangian in $\R^2/\Z^2$. This gives back the Kronecker foliation, note that $-\theta$ and $\theta$ are equivalent under $SL(2,\Z)$. Hence we find the same noncommutative tori for the third time. 

}
\subsection{Hilbert space of holomorphic states}

Let $X$ be a complex projective curve with positive genus $g$, $\omega$ be its volume form, and $\theta$ be an irrational number. By Theorem \ref{thm:main thm in intro}, the colimit object $E(\theta^{-},\{f_{2i}\})$ is a simple vector bundle of infinite rank and admits the Hermitian--Einstein metric. 

If we take $\theta>g-1$, the dimension of the space of the global holomorphic sections $\mathrm{h}^0(X,E(\theta^{-},\{f_{2i}\}))$ can be calculated in the following way.

$$\mathrm{h}^0(X,E(\theta^{-},\{f_{2i}\}))=\hom(\mathcal{O}_X,\varinjlim E_{2i})=\varinjlim \mathrm{h}^0(X,E_{2i})=\varinjlim p_{2i}+q_{2i} (1-g),$$ the second equality holds because $\mathcal{O}_X$ is a compact object, the third equality holds by Riemann--Roch thoerem. The sequence $p_{2i}+q_{2i} (1-g)$ diverges to $+\infty$ as $\theta> g-1$ and $\frac{p_{2i}}{q_{2i}}$ is the convergent of $\theta$.

There is a natural Hermitian metric on the space of global holomorphic sections $$\mathrm{H}^0(X, E(\theta^{-},\{f_{2i}\})).$$ Indeed, let $h$ be the Hermitian--Einstein metric on $E(\theta^{-},\{f_{2i}\})$ as in Theorem \ref{thm:main thm in intro}. For any two sections $s_1,s_2\in \mathrm{H}^0(X,E(\theta^{-},\{f_{2i}\}))$, we define the inner product to be $$(s_1,s_2)_{\omega} \coloneqq \int_X h(s_{1,x},s_{2,x})\omega.$$ 

This defines a Hermitian metric on the space $\mathrm{H}^0(X, E(\theta^{-},\{f_{2i}\}))$.

\begin{defn}
    The completion of $\mathrm{H}^0(X, E(\theta^{-},\{f_{2i}\}))$ with respect to the metric $(-,-)_{\omega}$ is a separable Hilbert space, and it is called the Hilbert space of \textit{holomorphic states} associated with $E(\theta^{-},\{f_{2i}\}).$
\end{defn}

Our definition of this Hilbert space is similar to von Neumann's definition of a direct integral of Hilbert spaces but with much more stringent coherent restrictions.\smallskip

\noindent\textbf{Acknowledgments.} We would like to thank Hao Fang, Chunyi Li, Hanfeng Li, Fei Si, Junrong Yan, Chuanjing Zhang, Lin Zhang, Runlin Zhang for helpful discussions. Yucheng Liu is financially supported by NSFC grant 12201011; Biao Ma is supported by NSFC grant 12471052.\smallskip

\noindent\textbf{Data Availability.} No data was used for the research described in the article.\smallskip

\noindent\textbf{Declarations.}\smallskip

\noindent\textbf{Conflict of interest.} There is no potential Conflict of interest concerning this paper.
\smallskip

%\noindent\textbf{Publisher’s Note} Springer Nature remains neutral with regard to jurisdictional claims in published maps and institutional affiliations.\smallskip

%\noindent Springer Nature or its licensor (e.g. a society or other partner) holds exclusive rights to this article under a publishing agreement with the author(s) or other rightsholder(s); author self-archiving of the accepted manuscript version of this article is solely governed by the terms of such publishing agreement and applicable law.

\appendix

\section{Hilbert bundles over planar domains }
\subsection{Norms for operators.\label{subsec:norms}}

Let $W$ be a Hilbert space. Let $L(W)$ denote
the bounded linear operators on $W$. Let $A\in L(W)$.
We denote the operator norm
\begin{equation}
|A|_{\rho}=\sup_{0<|x|\le1}\frac{|Ax|}{|x|}.\label{eq:-42}
\end{equation}
Denote $\text{Tr}_{\mathbf{}}$ be the trace on $L(W)$. $A$ is a trace class operator if $\sqrt{A^{\dagger}A}$ has
finite trace and denote as $T_{1}(W)$. If $A^{\dagger}A$
is a trace class operator, we denote 
\begin{equation}
|A|_{2}=(\text{Tr}_{\mathbf{}}A^{\dagger}A)^{\frac{1}{2}},\label{eq:-44}
\end{equation}
and denote $T_{2}(W)$. Then, we have the inequality for
$A\in T_{2}(W)$, $
|A|_{\rho}\leq|A|_{2}.$
Note $T_{2}(W)$ is a dense subspace of $L(W)$
and is a Hilbert space with norm $|\cdot|_{2}$.

% We consider another class of operator denoted as $\hat{T}_{2}(\mathbf{H})$
% which consist of operators of the form
% \begin{equation}
% B=c\text{Id}+A,\label{eq:-45}
% \end{equation}
% for $c\in\C$ and $A\in T_{2}(\mathbf{H})$. We denote 
% \begin{equation}
% |B|_{\hat{T}_{2}}=\inf\{(|c|^{2}+|A|_{2}^{2})^{\frac{1}{2}}:c\text{Id}+A=B\}.\label{eq:-52}
% \end{equation}
% Notice that if $B\in\hat{T}_{2}$ and $\mathbf{H}$ is infinite dimensional,
% then the decomposition (\ref{eq:-45}) is unique as $\text{Id}$ is
% not a trace class. If $\mathbf{H}$ is of finite rank $n$, we see
% that 
% \[
% |B|_{\hat{T}_{2}}^{2}=\text{Tr}B^{\dagger}B-\frac{|\text{Tr}B|^{2}}{n+1}.
% \]
% Then, $\hat{T}_{2}$ with the given norm defines a dense subspace
% of $L(\mathbf{H})$ and becomes a Hilbert space with $|\cdot|_{\hat{T}_{2}}$.
% $T_{2}$ is a closed subspace in $\hat{T}_{2}$, and in $T_{2}$,
% we have $|A|_{2}=|A|_{\hat{T}_{2}}$. Thus,
% \[
% T_{2}(\mathbf{H})\subset\hat{T}_{2}(\mathbf{H})\subset L(\mathbf{H}).
% \]

Let $\Omega\subset\mathbb{R}^{m}$ be an open set. Let $V$ be a Banach space with norm  $|\cdot|_{*}$ ( for instance, $V=L(W)$ or $T_2(W)$ and $|\cdot|_*$ is
the norm $|\cdot|_{\rho}$ or $|\cdot|_{2}$).
Suppose that $A(x)\in C_{0}^{\infty}(\Omega,V)$. We denote
\[
|\nabla^{k}A(x)|_{*}:=\sum_{\alpha\in\N^{m},|\alpha|_1=k}|\prod_{i=1}^{m}\pdv_{i}^{\alpha_{i}}A(x)|_{*},\ 
\|A\|_{*,C^{k}(\Omega)}:=\sup_{x\in\Omega}\sum_{l=0}^{k}|\nabla^{l}A(x)|_{*}.
\]
For Holder norm, we defines for $\alpha\in(0,1)$,
\begin{equation}
[A]_{*,C^{\alpha}(\Omega)}:=\sup_{x,y\in\Omega}\frac{|A(x)-A(y)|_{*}}{|x-y|^{\alpha}},\label{eq:-57} \ 
\|A\|_{*,C^{k,\alpha}}:=\|A\|_{*,C^{k}}+[\nabla^{k}A]_{*,C^{\alpha}}.
\end{equation}
We then define the corresponding function spaces by taking the completion
with respect to each norm. Define the $L^{p}$ and Sobolev norm as
follows.

\begin{equation}
\|A\|_{*,L^{p}(\Omega)}:=\int_{\Omega}|A|^p_{*}dx,\label{eq:-55}
\end{equation}
\begin{equation}
\|A\|_{*,W_{0}^{k,p}(\Omega)}:=\int_{\Omega}\sum_{\alpha\in\N^{m},|\alpha|_{1}\leq k}|\prod_{i=1}^{m}\pdv_{i}^{\alpha_{i}}A(x)|_{*}^{p}dx.\label{eq:-58}
\end{equation}
Define the $L^{p}(\Omega,V)$ and $W_{0}^{k,p}(\Omega,V)$ space to be the metric completion
of $C_{0}^{\infty}(\Omega,V)$ with respect to the corresponding norm
modulo almost everywhere vanishing functions. 

We have the following Sobolev embedding theorem. Notice that the
constant is independent of the choice of $|\cdot|_{*}$, \cite[Theorem 6.1]{Arendt2016MappingTF}.
\begin{prop}
[Sobolev embedding]
Suppose that $A\in W_{0}^{k,p}(\Omega,(V,|\cdot|_*))$. Then, for $\frac{1}{q}\geq\frac{1}{p}-\frac{k}{m}$,
\begin{equation}
\|A\|_{*,L^{q}(\Omega)}\leq C(m,q,p,k,\Omega)\|A\|_{*,W_{0}^{k,p}(\Omega)}.\label{eq:-59}
\end{equation}
If $p>\frac{m}{k}$, then for $0\leq\alpha<k-\frac{m}{p}$,
\[
\|A\|_{*,C^{\alpha}(\Omega)}\leq C(\alpha,p,k,m,\Omega)\|A\|_{*,W_{0}^{k,p}(\Omega)}.
\]
\end{prop}

\subsection{Hodge system with Dirichlet boundary values}
Let $V$ be a Banach space and $\check{V}$ the dual space. We denote $\Omega^p(V)=\Lambda^p T^*\R^k\otimes V$. Fix an orthonormal basis $\{\eta_i\}$ for $\Lambda^p T^*\R^k$. For any $\alpha\in \Omega^p(V)$, we can write $\alpha =\sum_i \eta_i\otimes v_i$. Then, we put a norm on $\Omega^p(V)$ by taking $\|\alpha\|_{\Omega^p(V)}=\sum_i\|v_i\|_V$. Since $\Lambda^pT^*\mathbb{R}^k$  is finite dimensional, the norm $\|\cdot\|_{\Omega^p(V)}$  is equivalent to the injective and projective tensor norm on  $\Lambda^p T^*\R^k\otimes V$. 

Let $\Omega\subset \mathbb R^k$ be a bounded domain with smooth boundary. Suppose that $\varphi(x)=\varphi_{\alpha\beta}(x) dx^\alpha\wedge dx^\beta$ is a $2$-form in $\Omega$ with values in $V$.  Let $\psi(x)$ be a function in $\Omega$ with value in $V$. Let $w(x)$ be a $1$-form on $\overline{\Omega}$  with value in $V$. Consider the following Hodge system for $u:\Omega \to \Omega^1(V)$:
\begin{align}
 \label{eq:hodgesystem1}   \pdv_\beta u_\alpha-\pdv_\alpha u_\beta & =\varphi_{\alpha\beta},\\
    \sum_\alpha \pdv_\alpha u_\alpha &=\psi,\\
    u|_{\pdv\Omega}& =w.
\end{align}

We call $u$ a strong solution if $u$ is in $ C^{1}(\Omega,\Omega^1(V))\cap C^{0}(\overline\Omega,\Omega^1(V))$.  We call $u$ a weak solution to the Hodge system if, for any $a\in\check{V}$,  $v=a(u)$ is a $2$-form in $\Omega$ satisfies

\begin{align}
\label{eq:hodgesystemweak}
    \pdv_\beta v_\alpha-\pdv_\alpha v_\beta & =a(\varphi_{\alpha\beta}),\\
    \sum_\alpha \pdv_\alpha v_\alpha &=a(\psi),\\
    v|_{\pdv\Omega}& =a(w),
\end{align}
in the sense of distribution. 
\begin{prop}
\label{prop:regularityofhodgesystem}Notations as above. There exists a unique weak solution $u$ to the Hodge system.  
Furthermore, if $\varphi\in  C^0(\overline\Omega,\Omega^2(V))$, then $u\in C^{\alpha}(\Omega,\Omega^1(V))\cap C^{0}(\overline\Omega,\Omega^1(V))$  for any $\alpha\in (0,1)$. If for some $\alpha\in (0,1)$ and $k\in \N$,  $\varphi$ is in $C^{k,\alpha}(\overline\Omega,\Omega^2(V))$, then $u\in C^{k+1,\alpha}(\Omega,\Omega^1(V))\cap C^{0}(\overline\Omega,\Omega^1(V)) $  is a strong solution.     
\end{prop}
\begin{proof}
A solution of  (\ref{eq:hodgesystemweak}) in the sense of distribution is clearly unique by standard elliptic theory. If $u_1,u_2$ are 2 weak solutions to (\ref{eq:hodgesystem1}), for any $a\in \check{V}$, $a(u_1)=a(u_2)$ pointwisely. Hence, by the Hahn-Banach theorem, $u_1\equiv u_2$.  
Let $G(x,y)$  be the Green function of $\Delta$ for the Dirichlet problem on $\Omega$ with $G(x,\cdot)|_{\pdv\Omega}=0$. Note that a weak solution $u$ satisfies \[\Delta u_\alpha =\sum_\beta \pdv_\beta \varphi _{\alpha\beta}+\pdv_\alpha \psi\] in the sense of distribution.  We denote 
 \[ 
 \tilde{u}_\alpha(x)=\int_\Omega \pdv_\beta G(x,y)\varphi_{\alpha \beta }(y)dv(y)+\int_{\pdv\Omega} (w_\alpha\pdv_\nu G -\varphi_{\alpha\beta }\pdv_\beta G)d\sigma.
 \]
Then, a straightforward calculation shows that $\tilde{u}=\tilde{u}_\alpha dx^\alpha$ is a weak solution. The corresponding regularity results follow immediately. 
\end{proof}
 
\subsection{Elliptic regularity for Hodge system in 2D}

As before, we assume that $V$ is a Banach space. Let $Q=[0,1]\times [0,1]\subset \R^2$. Let $L^p(Q,V), W^{1,p}(Q,V)$ the $L^p$ and $W^{1,p}$  space defined using the norm of $V$ as in (\ref{eq:-55}) and (\ref{eq:-58}).

Let $u: Q\to \Omega^1(V)$. We consider the Hodge system for $u=u_xdx+u_ydy$, 
\begin{align}
   \label{eq:hodge2} \pdv_x u_y-\pdv_y u_x & =\varphi\ ,\\
    \label{eq:hodge3}\pdv_x u_x+\pdv_y u_y&=\psi  ,\\
   \label{eq:hodge4} \iota_\nu u|_{\pdv Q}& =0 .
\end{align}
Here $\nu$ is the unit outer normal vector.  

 We have the following $L^p$ regularity theory of 
\begin{prop}
 \label{prop:ellipticestimate}  Let $1<p<\infty$.  Suppose that $u\in W^{1,p}(Q,\Omega^1(V))$ is a weak solution to (\ref{eq:hodge2})-(\ref{eq:hodge4}). Then, there exists a constant $C$ depending only on $p$ such that \[
   \|u\|_{W^{1,p}(Q,\Omega^1{(V)})}\leq C(\|\varphi\|_{L^p}+\|\psi\|_{L^p})
   \]
\end{prop}
\begin{proof}
    Notice that by reflection $u,\varphi,\psi$ along the $x$ -axis, we may assume that $u$ is a solution to the Hodge system in $[-1,1]\times [0,1]$ in the sense of distribution (by testing against a compactly supported smooth function with values in $\check{V}$). Contracting (\ref{eq:hodge2}) against $\pdv_x$ and using (\ref{eq:hodge3}), we have 
\[\Delta u_y=\pdv_x\varphi+\pdv_y\psi\]
in the sense of distribution and satisfies $u_y=0$ on $[-1,1]\times\{0,1\}$. Identifying  $\{-1\}\times [0,1]$  and $\{1\}\times [0,1]$ , we obtain $S^1\times [0,1]$ with $S^1$ a circle of length $2$. Then $u_y$ solves a homogeneous Dirichlet problem on $T:=S^1\times [0,1]$ in the sense of distribution. So we let $G_z(w)$ be the Green function for the Dirichlet problem on $T$.   Let 
\[ \tilde  u_y(z)=-\int_{T}\pdv_xG_z(w)\varphi(w)+\pdv_yG_z(w)\psi(w)dv(w). \]
Then, $\Delta_{T} (u_y-\tilde u_y)=0$ in the sense of distribution and $u_y=\tilde{u}_y=0$ on the boundary of $T$. Thus, for any $w\in\check{V}$, $w(u_y-\tilde{u}_y)$ is a harmonic function with zero boundary value which must be identically $0$. By Hahn-Banach theorem, $u_y-\tilde{u}_y\equiv0$.  Since $|\nabla^i G_z(w)|\leq c|z-w|^{-i}$ for $i=1,2$, by standard $L^p$ theory of Newton potential (see for instance \cite[Chapter 9]{GilbargTrudinger}), we have
\[\|u_y\|_{W^{1,p}}\leq C(\|\varphi\|_{L^p}+\|\psi\|_{L^p}),\] where the constant $C$ only depends on $p$.  Similar result holds for $u_x$. Hence, we have finished the proof.
\end{proof}
\subsection{Existence of a Coulomb Gauge for Hilbert bundles }

Let $W$  be a separable complex Hilbert space. Let $L(W)$ be the $C^*$-algebra of bounded operators on $W$ equipped with the operator norm $|\cdot|_\rho$. We denote by $U(W)$ the unitary group in $L(W)$. Then with the norm topology $U(W)$ is a Banach Lie group, and the Lie algebra  $\mathfrak{u}$   consists of skew symmetric operators in $L(W)$ (\cite{schottenloher2018unitary}).  

Let $Q=[0,1]\times [0,1]\subset \R^2$.  Let $A\in C^\infty(Q,\Omega^1(\mathfrak{u}))$ be a smooth operator valued 1-forms which satisfies $A^*=-A$.  Then $d+A$ is a unitary connection for the trivial bundle $Q\times W$ over $Q$. We denote the curvature operator $F(A)=dA+A\wedge A$ which is a $\mathfrak{u}$ valued $2$-form. Let $u\in C^\infty (Q,U(W) $ be a gauge transform. Then $u$ acts on $A$ by \[u(A)=uAu^{-1}-duu^{-1}.\] The curvature transforms by
\[F(u(A))=uF(A)u^{-1}.\]Since $u$ is unitary,  $|F(u(A))|=|F(A)|$. If two unitary connections $A_1$ and $A_2$ are related by a smooth gauge transform, then $A_1$ and $A_2$
 are gauge equivalent. 

\begin{defn}
    We call a connection $A$ a Coulomb gauge on $Q$  if $d^*A=0$ in the interior of  $Q$  and $\iota_\nu A=0$ on the boundary $\pdv Q$. Here if $A=A_xdx+A_ydy$, then \[d^*A=\pdv_x A_x+\pdv_y A_y.\]
\end{defn}

The following proposition proved originally by Uhlenbeck allows us to construct the Coulomb gauge with a small $L^p$ ($p\geq 2$) curvature bound.
\begin{prop}
    \label{prop:coulombgaugeuniform}Let $p\geq 2$. There exists a constant $\epsilon_0=\epsilon_0(p)>0$ that only depends on $p$ with the following properties: If $A$ is a smooth connection with \[\|F(A)\|_{L^p}=(\int_Q |F(A)|^p )^\frac{1}{p}\leq \epsilon_0, \]
then $A$  is gauge equivalent to a Coulomb gauge $A_c$ such that
\[\|A_c\|_{W^{1,p}}\leq C \|F(A)\|_{L^p}.\]
The constant $C$ only depends on $p$.
\end{prop} 
The existence of a Coulomb gauge in higher dimensional cubes can be proved similarly. We should, however, remark that if the base space has dimension larger than 4, instead of $L^p$ and $W^{1,p}$, more natural function spaces to consider are Morrey-Sobolev spaces. The techniques involving Morrey-Sobolev spaces are more sophisticated. See  \cite{TaoTian2002dj,Meyer2003}.

 We should remark that  the constants in  Proposition \ref{prop:coulombgaugeuniform} do not depend on $W$. The reason is twofold: elliptic estimates can be achieved by a singular integral representation and do not rely on $W$;  Sobolev inequalities hold for functions valued in Banach space with same constants for scalar functions. We will only sketch the proof as it is essentially the same in \cite{UhlenbeckLp}. 

\begin{proof}
  Let $K>0$  be a large constant and $\epsilon$  sufficiently small to be chosen later. Let $V_\epsilon$ be the space of smooth unitary connections in $Q$ that satisfies $\|F(A)\|_{L^p}\leq \epsilon$. Denote $V^*_\epsilon$ the space of smooth connections gauge equivalent to a Coulomb gauge $A_c$ such that 
\[\|A_c\|_{W^{1,p}}\leq K\epsilon.\]

   Step1.  We have the following bootstrap estimate. 
   \begin{lem}    
   \label{lem:bootstrap} If $A\in V^*_\epsilon$, then 
  \begin{align}
    \label{eq:bootstrap}   \|A_c\|_{W^{1,p}}\leq \frac{K}{2}\epsilon.
  \end{align}

   \end{lem} The proof of the lemma is similar to that in \cite[Lemma 9.1]{TaoTian2002dj}. Using Proposition \ref{prop:ellipticestimate}, we have 
    \[
   \|A_c\|_{W^{1,p}}\leq C( \|F(A_c)\|_{L^p}+ \|A_c\wedge A_c\|_{L^p}).
   \]
  Notice that by the Sobolev embedding theorem, $\|A_c\|_{L^{2p}}\leq C_{Sob}\|A_c\|_{W^{1,p}} $. The constant $C_{Sob}$ only depends on $p$  (\cite[Theorem 6.1]{Arendt2016MappingTF}.  Hence, 
  \[ \|A_c\|_{W^{1,p}}\leq C( \|F(A_c)\|_{L^p}+ C_{Sob}\|A_c\|^2_{W^{1,p}} ).\]
By choosing a large $K$ depending on $C,C_{Sob}$, and a sufficiently small $\epsilon$ depending on $K$, (\ref{eq:bootstrap}) holds. 

 Step2.  We initiate the standard continuity method.  Let $A^t(z):=tA(tz)$. Let $I=\{t\in[0,1]: \forall s\in[0,t],A^s\in V^*_\epsilon\}$.   If $t=0$, then $A^t$  is a trivial connection that certainly belongs to $V^*_\epsilon$, hence $I\not =\emptyset$. Let $s\in I.$ Suppose that $A_c^s=u_s(A^s)$. Then \[du_s=u_sA^s-A_c^su_s.\] $u_s$  has bounded pointwise norm since it is unitary.  From the pointwise estimate and Sobolev inequality, we have 
 \[
\|u_s\|_{W^{2,p}}\leq C(\|A^s\|_{W^{1,p}},\| A^s_c\|_{W^{1,p}}). 
 \]
Then Lemma \ref{lem:bootstrap}  shows that $u_s$ is uniformly bounded in $W^{2,p}$. If $s$ increases to $t^*=\sup I$, then $u_s$ is uniformly bounded in $W^{2,p}$ and hence has convergent subsequence in $W^{1,p}$ to some $u_{t^*}$ in $W^{2,p}$. Then $A_c^{t^*}=u_{t^*}(A^{t^*})$ is a Coulomb gauge with $\|A^{t^*}_c\|_{W^{1,p}}\leq K\epsilon$ by Fatou's lemma. Hence $I$ is closed. 

Step3. It remains to show the openness of $I$ to conclude the proof. This is essentially the implicit function theory. The linearized equation is a Neumann boundary value equation and the solvability can be achieved using the Green function representation.  See for example \cite[pp 591-592]{TaoTian2002dj}. 
\end{proof}

\bibliographystyle{alpha}
	\bibliography{bibfile}
\end{document}